\documentclass[11pt, a4paper]{amsart}

\usepackage{epsfig, amsfonts, amssymb, url, amsthm, subcaption, tikz, amsmath, calrsfs, bigdelim, multirow, xcolor, appendix, float,mathtools, bm}
\usepackage[utf8]{inputenc}
\usepackage[T1]{fontenc}
\usepackage{tgschola}
\usepackage{stmaryrd, dsfont}

\usepackage[top=3cm, bottom=3cm, left=2.5cm, right=2.5cm]{geometry}

\usepackage{hyperref}

\DeclareMathAlphabet{\pazocal}{OMS}{zplm}{m}{n}

\usetikzlibrary{matrix}

\usepackage{ulem}

\makeatletter
\newcommand\RSloop{\@ifnextchar\bgroup\RSloopa\RSloopb}
\makeatother
\newcommand\RSloopa[1]{\bgroup\RSloop#1\relax\egroup\RSloop}
\newcommand\RSloopb[1]%
  {\ifx\relax#1%
   \else
     \ifcsname RS:#1\endcsname
       \csname RS:#1\endcsname
     \else
       \GenericError{(RS)}{RS Error: operator #1 undefined}{}{}%
     \fi
   \expandafter\RSloop
   \fi
  }
\newcommand\X{0}
\newcommand\RS[1]%
  {\begin{tikzpicture}
     [every node/.style=
       {circle,draw,fill,minimum size=1.5pt,inner sep=0pt,outer sep=0pt},
      line cap=round
     ]
   \coordinate(\X) at (0,0);
   \RSloop{#1}\relax
   \end{tikzpicture}
  }
\newcommand\RSdef[1]{\expandafter\def\csname RS:#1\endcsname}
\newlength\RSu
\RSdef{i}{\draw (\X) -- +(90:\RSu) node{};}
\RSdef{l}{\draw (\X) -- +(110:\RSu) node{};}
\RSdef{r}{\draw (\X) -- +(70:\RSu) node{};}
\RSdef{I}{\draw (\X) -- +(90:\RSu) coordinate(\X I);\edef\X{\X I}}
\RSdef{L}{\draw (\X) -- +(120:\RSu) coordinate(\X L);\edef\X{\X L}}
\RSdef{R}{\draw (\X) -- +(60:\RSu) coordinate(\X R);\edef\X{\X R}}

\begin{document}

\title{The KPZ Equation on the Real Line}

\newtheorem{theorem}{Theorem}[section]
\newtheorem{lemma}[theorem]{Lemma}
\newtheorem{proposition}[theorem]{Proposition}
\newtheorem{corollary}[theorem]{Corollary}
\newtheorem{definition}[theorem]{Definition}
\newtheorem{assumption}[theorem]{Assumption}

\theoremstyle{definition}
\newtheorem{remark}[theorem]{Remark}
\newtheorem{example}[theorem]{Example}
\newtheorem{exercise}[theorem]{Exercise}

\setlength{\parindent}{0pt}

\newcommand{\RR}{\mathbb{R}}
\newcommand{\NN}{\mathbb{N}}
\newcommand{\ZZ}{\mathbb{Z}}
\newcommand{\TD}{n\mathbb{T}^d}
\newcommand{\YY}{\mathbb{Y}}
\newcommand{\EE}{\mathbb{E}}
\newcommand{\QQ}{\mathbb{Q}}
\newcommand{\PP}{\mathbb{P}}

\newcommand{\LL}{\mathcal{L}}
\newcommand{\CC}{\mathcal{C}}
\newcommand{\WW}{\mathcal{W}}
\newcommand{\MM}{\mathcal{M}}
\newcommand{\TT}{\mathcal{T}}
\newcommand{\II}{\mathcal{I}}
\newcommand{\mY}{\mathcal{Y}}
\newcommand{\FF}{\mathcal{F}}
\newcommand{\mD}{\mathcal{D}}
\newcommand{\mX}{\mathcal{X}}
\newcommand{\mA}{\mathcal{A}}


\newcommand{\DD}{\mathfrak{D}}
\newcommand{\ODD}{\ola{\mathfrak{D}}}
\newcommand{\LLL}{\pazocal{L}}

\newcommand{\bb}[1]{\mathbb{#1}}
\newcommand{\mc}[1]{\mathcal{#1}}


\newcommand{\ppara}{\mathord{\prec\hspace{-6pt}\prec}}
\newcommand{\para}{\varolessthan}
\newcommand{\reso}{\varodot}


\newcommand{\ola}[1]{\overleftarrow{#1}}
\newcommand{\ol}[1]{\overline{#1}}


\newcommand{\cc}{\frac{1}{(2 \pi)^{\frac{d}{2}} }}
\newcommand{\hh}{\frac{1}{2}}
\newcommand{\ve}{\varepsilon}

\newcommand{\bigslant}[2]{{\raisebox{.1em}{$#1$}\left/\raisebox{-.1em}{$#2$}\right.}}
\newcommand{\lqt}{``}
\newcommand{\rqt}{'' }


\newcommand{\smooth}{\LLL C^{\alpha /2} (\RR;C^{\infty}_b(\RR))}
\newcommand{\rel}{\Lambda_{\lambda}}
\newcommand{\wQ}{\QQ_{x_0}}
\newcommand{\olh}{\overleftarrow{h}}
\newcommand{\oh}{\overline{h}}
\newcommand{\hb}{\hat{\beta}}
\newcommand{\bp}{\beta^\prime}
\newcommand{\myc}{\mY_{kpz}^{\zeta, b}}
\newcommand{\myi}{\mY_{kpz}^{\infty}}
\newcommand{\myp}{\mY_{kpz}^{poly}}
\newcommand{\myk}{\mY_{kpz}}
\newcommand{\weights}{\boldsymbol{\rho}(\omega)}
\newcommand{\nnorm}[1]{| \!  | \!  |  #1 | \! | \! |}
\newcommand{\SR}{\mathcal{S} \! \left( \RR \right) \!}
\newcommand{\SO}{\mathcal{S}_{\omega} \! \left( \RR \right) \!}
\newcommand{\SDO}{\mathcal{S'}_{ \! \! \omega} \! \left( \RR \right) \!}
\newcommand{\orat}[1]{\overrightarrow{{#1}_{\tau}}}


\newcommand{\BA}{\scalebox{0.8}{\RS{lr}}}
\newcommand{\BB}{\scalebox{0.8}{\RS{rLrl}}}
\newcommand{\BC}{\scalebox{0.8}{\RS{rLrLrl}}}
\newcommand{\BD}{\scalebox{0.8}{\RS{{Lrl}{Rrl}}}}
\newcommand{\BZ}{\scalebox{0.8}{\RS{r}}}


\newcommand{\norm}[1]{\left\lVert #1 \right\rVert}

\newcommand{\bv}[1]{\left\vert #1 \right\vert}
\newcommand{\bq}[1]{\left[ #1 \right]}
\newcommand{\bg}[1]{\left\{ #1 \right\} }
\newcommand{\ba}[1]{\langle #1 \rangle}

\newcommand{\red}[1]{\textcolor{red}{#1}}
\normalem

\author{Nicolas Perkowski}
\address{Max-Planck-Institut f\"ur Mathematik in den Naturwissenschaften Leipzig \& Humboldt-Universit\"at zu Berlin}
\email{perkowsk@math.hu-berlin.de}
\thanks{NP gratefully acknowledges financial support by the DFG via the Heisenberg program and via Research Unit FOR 2402.}

\author{Tommaso Cornelis Rosati}
\address{Humboldt-Universit\"at zu Berlin}
\email{rosatito@math.hu-berlin.de}
\thanks{TCR was funded by the IRTG 1740 and this paper was developed within the scope of the IRTG 1740 / TRP 2015/50122-0, funded by the DFG / FAPESP}

\date{August 1, 2018.}


\begin{abstract}
We prove existence and uniqueness of distributional solutions to the KPZ equation globally in space and time, with techniques from paracontrolled analysis. Our main tool for extending the analysis on the torus to the full space is a comparison result that gives quantitative upper and lower bounds for the solution. We then extend our analysis to provide a path-by-path construction of the random directed polymer measure on the real line and we derive a variational characterisation of the solution to the KPZ equation.
\end{abstract}

\maketitle

\setcounter{tocdepth}{1}
\tableofcontents

\section{Introduction}

In this work we solve the Kardar-Parisi-Zhang (KPZ) equation on the real line, i.e. we construct a unique $h\colon \RR_{\ge 0} \times \RR \to \RR$ such that
\begin{equation}\label{eqn:kpz}
(\partial_t - \hh \Delta_x)h = \hh (\partial_xh)^2 + \xi, \ \ h(0) = \oh,
\end{equation}
where $\xi$ is a Gaussian space-time white noise, the generalized Gaussian process on $\RR_{\ge 0} \times \RR$ with singular covariance structure $\EE[\xi(t,x) \xi(s, y)] = \delta(t{-}s) \delta(x{-}y)$.

The KPZ equation was introduced in \cite{kpz_original} as a model for the growth of a one-dimensional interface that separates two two-dimensional phases of which one invades the other. The conjecture of \cite{kpz_original}, now called \emph{strong KPZ universality conjecture}, was that any $(1+1)$--dimensional (one time and one space dimension) interface growth model that is subject to random influences, surface tension, and lateral growth, shows the same large scale behavior under the now famous $1{-}2{-}3$ scaling, and that the KPZ equation provides a prototypical example of such a model. Since then it became apparent that there is a second, weaker universality in the class of $(1+1)$--dimensional interface growth models: If the lateral growth or the random influence is very weak, then according to the \emph{weak KPZ universality conjecture} the model is expected to be well approximated by the KPZ equation on large scales. See \cite{Quastel2011, Corwin2012, quastel_intro} for nice introductions to these universality questions.

The difficulty with the KPZ equation is however that its solution $h$ is not a differentiable function in $x$, and therefore it is not clear how to interpret the nonlinearity $(\partial_xh)^2$ on the right hand side of \eqref{eqn:kpz}. This problem can be avoided by applying the \emph{Cole-Hopf transform}: The process $w = e^h$ formally solves the stochastic heat equation
\begin{equation}\label{eqn:she}
(\partial_t - \hh \Delta_x)w = w\xi, \ \ w(0) = e^{\oh},
\end{equation}
which can be analyzed using It\^o integration \cite{Walsh1986}. This was already observed by Kardar, Parisi and Zhang \cite{kpz_original}, and the first mathematically rigorous formulation is due to Bertini, Cancrini and Jona-Lasinio~\cite{Bertini1994} who simply define $h = \log w$ for the solution $w$ to \eqref{eqn:she}, a process that is strictly positive by a strong comparison result of Mueller \cite{Mueller}. Results such as the scaling limit proven by Bertini and Giacomin \cite{bertini} suggested that $h$ is the physically relevant solution process, but nonetheless it remained unclear if and in what sense the Cole-Hopf solution actually solves the KPZ equation.

A rigorous proof of the existence of distributional solutions to the KPZ equation on the torus was a milestone in the theory, reached by Hairer via rough paths \cite{hai:solving_kpz} as well as through his theory of regularity structures \cite{hairer_reg_struct, frizRP}. Similar results have been obtained by Gubinelli, Imkeller and Perkowski via paracontrolled distributions \cite{singular_GIP, kpz}. These theories were the starting point for the new research field of singular SPDEs, with many developments in recent years that allow to study more singular or quasilinear equations, with boundary conditions or on manifolds, and to derive qualitative properties of the solutions. 

At the center of this new pathwise approach to SPDEs lies the idea of expanding the solution on small scales via the driving noise and higher order terms constructed from it. The non-linearity can then be controlled in terms of the ``enhanced noise'', called \emph{model} in regularity structures, i.e. the noise together with the higher order terms appearing in the expansion of the solution, and the solution to the equation becomes a continuous functional of the enhanced noise. However, in many situations (including for the KPZ equation) the higher order terms can only be constructed with the help of a suitable renormalisation, and this means that the solution we eventually find does not solve the original equation, but a renormalised version of it \cite{Bruned2017}.

%

But while we now have a good understanding in what sense the Cole-Hopf solution solves the KPZ equation and how to interpret its renormalisation, all this is  restricted to the equation on the torus or in a finite volume with boundary conditions \cite{hai_gerenc, Corwin2016, Goncalves2017}. Since one of the main interests in the KPZ equation comes from its large scale behavior it would be more natural to solve it on $\RR$, a space that can be arbitrarily rescaled. Using the probabilistic notion of \emph{energy solutions} \cite{Goncalves2014, Gubinelli2013} it is possible to give an intrinsic formulation of the KPZ equation on $\RR$, but this is essentially restricted to stationary initial conditions \cite{Gubinelli2018Energy}.

Here we extend the pathwise approach described above, implemented in the language of paracontrolled distributions, to develop a solution theory for the KPZ equation on $\RR$ for a fairly general set of initial conditions. The additional difficulty compared to the equation in bounded volume is that here we have to work in weighted function spaces, and the weights do not mix well with the nonlinearity: Roughly speaking, if the growth of $\partial_x h$ can be controlled with the weight $z$, then we need the stronger weight $z^2$ to control $(\partial_x h)^2$ and this prevents us from setting up a Picard iteration. That is the reason why most works on singular SPDEs deal with equations in finite volume. First steps to overcome this restriction were taken by Hairer and Labb\'e~\cite{pam, multiplicative} in their study of the linear rough heat equation and the linear parabolic Anderson model on the whole Euclidean space. For non-linear equations, a priori estimates are a natural and powerful tool and they were very successful in the study of the $\Phi^4_d$ equations in the work by Mourrat and Weber \cite{phi4, Mourrat2017Dynamic}, by Gubinelli and Hofmanova \cite{hofmanova_phi}, and by Barashkov and Gubinelli \cite{Barashkov2018}, all relying on the damping induced by the term $-\phi^3$. But such estimates depend strongly on the structure of the equation and this prevents the development of a general solution theory for singular SPDEs in infinite volume. In particular the KPZ equation is neither linear nor does it have a damping term, and therefore we need a different approach to control its solution in infinite volume.

A construction of pathwise solutions to the KPZ equation on the real line can probably be obtained starting from Hairer-Labb\'e's regularity structure based solution $w$ to the stochastic heat equation~\cite{multiplicative}. It should be possible to adapt the arguments from the proof of~\cite[Theorem~5.1]{Cannizzaro2017Malliavin} to deduce a strong maximum principle for the (regularity structure version of the) stochastic heat equation and thus the strict positivity of $w$. Then a chain rule in regularity structures should show, similarly as in~\cite[Section~4.3]{kpz}, that $h = \log w$ solves the KPZ equation in the regularity structure sense. The problem with that approach is that it gives no control at all for the growth of $h$ at infinity, and therefore there is no hope to get uniqueness -- after all there is non-uniqueness even for classical solutions to the heat equation if we allow for too much growth.
In this work we derive a priori $L^{\infty}$ estimates with linear growth for the solution to the KPZ equation on $\RR$, which give rise to an existence and uniqueness result and thus a complete theory. To derive these bounds we use the comparison principle and the link between KPZ equation and stochastic heat equation through the Cole-Hopf transform. There are several known comparison results for the stochastic heat equation, mostly based on probabilistic techniques and often relying on the specific probabilistic properties of the Gaussian noise. But although this leads to sharp bounds for the growth at infinity, see for example the work by Conus, Joseph and Khoshnevisan \cite{Conus}, there seem to be only qualitative results for the decay at infinity, a notable result being the strict positivity of the solution proven by Mueller \cite{Mueller} (and the regularity structure version of that result by Cannizzaro, Friz and Gassiat~\cite{Cannizzaro2017Malliavin} that we mentioned above). The analytic approach to SPDEs we follow allows us to use the full power of classical comparison results, providing us with quantitative and effective, if not sharp, upper and lower bounds, by restricting ourselves to a suitable class of strictly positive initial conditions. These estimates are very weak in terms of regularity but sufficient to avoid singularities when applying the Cole-Hopf transform, and this allows us to lift the a priori bounds to paracontrolled topologies and to prove uniqueness. Our solution is locally $1/2{-}\ve$ H\"older continuous in space and $1/4{-}\ve$ in time, and it has linear growth at infinity whereas its \lqt paracontrolled derivative\rqt may have sub-exponential growth.

As an application of our results, we give two alternative formulations of the KPZ equation by linking it with the random directed polymer measure and with a variational problem. The random directed polymer measure formally has the density
$$
\frac{d\QQ}{d\PP} = \frac{1}{Z}\exp\left( \int_0^T \xi(T-t, W_t) \ dt \right)
$$ 
with respect to the Wiener measure $\PP$ on $C([0, T])$, where $Z$ is a normalisation constant. On the torus this measure was constructed by Delarue and Diel \cite{singular}, see also \cite{cannizzaro}, who observed that $\QQ$ formally solves the SDE with distributional drift
$$
dX_t = \partial_x h(T-t, X_t) dt + dW_t,
$$
and then proceeded to construct a unique martingale solution $X$ by solving the Kolmogorov backward equation using rough path integrals. On the real line Alberts, Khanin and Quastel \cite{AlbertsQuastelPolymer} gave a different, probabilistic construction of $\QQ$, based on Kolmogorov's extension theorem, but they did not establish the link to the SDE beyond formal calculations. Here we combine the approaches of~\cite{singular} and~\cite{AlbertsQuastelPolymer}, which have existed independently so far, and give a rigorous explanation of the above SDE. We give a path-by-path construction of the random directed polymer measure that does not depend on the statistical properties of the white noise, but only on the ``model'' associated to it. We also show that the KPZ equation can be interpreted as the value function of the stochastic control problem
\[
	h(t,x) = \sup_v \EE_x\Big[ \oh(\gamma^v_t) - \int_0^t \xi(t-s, \gamma^v_s) \ ds - \frac12 \int_0^t v_s^2 \ ds \Big| \xi \Big],
\]
where under $\PP_x$
\[
	\gamma^v_t = x + \int_0^t v_s \ ds + W_t
\]
for a Brownian motion $W$ that is independent of $\xi$, a representation that was previously derived in \cite{kpz} for the KPZ equation on the torus.

%

\subsection{Structure}
In the first section we introduce techniques from paracontrolled calculus for SPDEs in a weighted setting, cf. \cite{singular_GIP, kpz, jorg}. Among them are the commutation and product estimates from Lemmata \ref{lem:parabolic_paraproduct_estimates} and \ref{lem:paraproduct_estimates}, as well as tailor-made Schauder estimates for the weighted setting, e.g. Lemma \ref{lem:interpolation_schauder}. The existence and uniqueness of solutions to the KPZ equation follows from a comparison result, Lemma \ref{lem:lower_estimate}. The lower estimate guarantees that the Cole-Hopf solution is the unique paracontrolled solution to the KPZ equation and that the latter depends continuously on the parameters of the equation (Theorem \ref{thm:existence_uniqueness_KPZ}). When considering uniqueness of solutions to PDEs on the whole space it is important to add some weight assumptions on the initial conditions. In this work we assume roughly linear growth of the initial condition (for the precise statement see Assumption \ref{assu:initial_condition} and Table \ref{table:kpz}). This is imposed upon us by the Cole-Hopf transform and the weighted Schauder estimates. These conditions suffice to start the equation in the invariant measure, the two sided Brownian motion \cite{imamura, borodin}.

Section \ref{sect:polymer_measure} addresses the random directed polymer measure $\QQ$. We prove sub-exponential moment estimates (Lemma \ref{lem:expntl_delta_momts}) and we show that the polymer measure is absolutely continuous with respect to a reference measure $\PP^U$ which we refer to as the partial Girsanov transform (cf. \cite[Section 7]{kpz}) and which is in turn singular with respect to the Wiener measure.
We conclude with two characterisations of the solution $h$ to the KPZ equation. The first, via the Feynman-Kac formula (Remark \ref{rem:characterisation_h_feyn_kac}), states that the solution $h$ is the free energy associated to the measure $\QQ$. The second is a variational representation \`a la Boué-Dupuis (Theorem \ref{thm:variational_rep}), cf. \cite{Boue_Dupuis, ustunel}.
The rest of this work is dedicated to technical, yet crucial, results. In particular (Section \ref{sect:abstract_solutions}) we prove a solution theorem for linear SPDEs, which applies to all linear equations studied in \cite{multiplicative, kpz, jorg}.

\begin{remark}
	Our approach uses the Cole-Hopf transform in several crucial steps. But we expect that the transform can be entirely avoided by making stronger use of the variational formulation of the KPZ equation, as soon as we can prove the following conjecture: Let, with the notation of Section~\ref{sec:prelim} below, $X \in C \CC^{-1/2-\varepsilon}_{p(\varepsilon)}$ for all $\varepsilon>0$, let $(\partial_t - \hh \Delta_x) Q = \partial_x X$, and let $Q \reso X \in C \CC^{-\varepsilon}_{p(\varepsilon)}$ for all $\varepsilon >0$. Then we conjecture that the paracontrolled solution $u$ to
	\[
		\partial_t u = \hh \Delta_x u + \partial_x (u X) + f, \ \ u(0) =0,
	\]
	grows sublinearly in the $L^\infty$ norm in space, provided that $f$ does. That is, for some $\delta < 1$
	\[
		\sup_{t \in [0,T]} \sup_{x \in \RR} \frac{|f(t,x)|}{1+|x|^\delta} < \infty \ \ \Rightarrow \ \ \sup_{t \in [0,T]} \sup_{x \in \RR} \frac{|u(t,x)|}{1+|x|^\delta} < \infty.
	\]
	While this conjecture seems very plausible, we are at the moment not able to prove or disprove it, and we leave it for future work.
\end{remark}

%

\section{Preliminaries}\label{sec:prelim}

\subsection{Fourier Transform}

We review basic knowledge and notations regarding the Fourier transform. We define the space of Schwartz functions $\mathcal{S}  \! \left( \RR \right)$ as the space of smooth and rapidly decaying functions. The dual space $\mathcal{S'} \! \left( \RR\right)$ is the space of tempered distributions. Let $\varphi \in \mathcal{S} \! \left( \RR \right) \! ,$ then we define for all $\xi \in \RR$: $$ \hat{\varphi} (\xi) = \FF \varphi  \, (\xi) = \int_{\RR} \! \varphi(x) e^{-i x \xi } dx$$ and for $\varphi \in \mathcal{S'} \! \left( \RR \right)$ we define the Fourier transform in the sense of distributions: $$\ba{ \hat{\varphi} , \psi} =  \ba{ \FF \varphi , \psi}  = \ba{\varphi , \FF \psi}, \ \forall \psi \in \mathcal{S} \! \left( \RR \right).$$
For $\varphi \in \mathcal{S}\! \left( \RR \right) \!$ the Fourier transform has the inverse $$\FF^{-1} \varphi \, (x) = \frac{1}{2 \pi}\int_{\RR} \varphi (\xi) e^{i x \xi} d\xi.$$ 
Since we will consider functions that have more than just polynomial growth at infinity, it is necessary that we go beyond the setting of tempered distributions and consider tempered ultra-distributions. This theory is presented in \cite{triebel} or \cite{Bjorck}. For a simple and hands-on introduction to all the tools we need we refer to \cite{jorg}.  

Consider the function  $$ \omega(x) = |x|^{\delta}, \ \delta \in (0,1)$$ with $\delta$ fixed once and for all. Using this weight, we define spaces of exponentially decaying Schwartz functions and their duals as follows.
\begin{definition}
   For $f \in \SR$ we define the seminorms
   \begin{align*}
       p_{\alpha, \lambda}(f) &  = \sup_{x \in \RR} e^{\lambda \omega(x)} |\partial^{\alpha} f(x)| \\
       \pi_{\alpha, \lambda}(f) &  = \sup_{x \in \RR} e^{\lambda \omega(x)} |\partial^{\alpha} \FF (f) (x)|
   \end{align*}
   for $\lambda > 0$ and $\alpha \in \NN^d$, and the associated locally convex space
   $$
   \SO = \bg{f \in \SR \big\vert  \ p_{\alpha, \lambda}(f) < \infty , \pi_{\alpha, \lambda}(f) < \infty \, \forall \, \lambda > 0, \alpha \in \NN_0^d},
   $$
   and we denote by $\SDO$ its dual, which we call the space of \emph{tempered ultra-distributions}. 
\end{definition}
 
We have the inclusions $$\SO \subsetneq \SR \subsetneq \mathcal{S'} \! \left(\RR \right) \subsetneq \SDO.$$ Finally, we  can define the Fourier transform on the space of tempered ultra-distributions just as before: For $f \in \SDO$ we set
\begin{align*}
    \langle \FF f, \varphi \rangle = \langle f, \FF \varphi \rangle , \ \ \langle \FF^{-1} f, \varphi \rangle & = \langle f, \FF^{-1} \varphi \rangle.
\end{align*}
We have introduced the Fourier Transform for exponentially decaying functions so that we can extend the Littlewood -- Paley theory also to weighted functions. We now fix the weights that we are allowed to use.

\begin{definition}
    We denote by $\boldsymbol{\rho}(\omega)$ the set of all measurable, strictly positive functions $z: \RR \to (0, \infty)$ such that for some $\lambda > 0$ and uniformly over all $x,y \in \RR$
    $$
        z(x)^{-1} \lesssim z(y)^{-1} e^{\lambda \omega (x-y) }.
    $$
    If this bound holds true we say that $z$ is $\omega-$moderate.
\end{definition}

The need for considering $\omega-$moderate weights can be explained by the following calculation: to estimate the convolution $\varphi * f$ we can compute
$$
    \bv{\frac{\varphi * f (x)}{z(x)}} \lesssim \langle \vert \varphi(x - \cdot) e^{\lambda \omega(x- \cdot)} \vert ,  \vert f(\cdot) / z(\cdot)\vert \rangle.
$$
Now we can intuitively bound the last term by some weighted norm of $f$ assuming that $\varphi$ is fixed and rapidly decaying.
\begin{definition}\label{def:weights}
In this work we shall consider the following two families of polynomial (resp. exponential) weights that lie in $\boldsymbol{\rho}(\omega)$,
\begin{align*}
    p(a)(x) & = (1 + |x|)^a, & \ a> 0, \\
    e(l+t) (x) & = \exp ((l + t) |x|^{\delta}), & \ l \in \RR, \ t \ge 0 , \ \delta \in (0, 1).
\end{align*}
\end{definition}
Note that we distinguish the parameters $t$ and $l$ because later we will consider time dependent weights.

\subsection{Littlewood--Paley Theory}

In this section we review the construction of weighted H\"older--Besov spaces. For a comprehensive introduction to Littlewood--Paley theory we refer to \cite{bah}. For a treatment of weighted spaces we also refer to \cite{jorg, Martin2018}. Following their constructions we fix a dyadic partition of unity generated by two smooth functions $\rho_{-1}$ and $\rho,$ that belong to $\SO$ and are supported in a ball around the origin $\mathcal{B}$ and and annulus around the origin $\mathcal{A}$, respectively. We then define $\rho_j(x) = \rho(2^{-j} x), j \ge 0$. Now we define the Littlewood--Paley blocks: for  $\varphi \in \SDO$ and $j \ge {-}1$ let $\Delta_j \varphi =  \FF^{-1} \left( \rho_j \widehat{\varphi} \right).$ We will use the following notation for paraproducts: $$S_i f = \sum\limits_{j \le i-1} \Delta_j f, \ \ f \para g =\sum\limits_{i} S_{i-1}\Delta_i g, \ \ f \reso g = \sum\limits_{|i - j | \le 1} \Delta_i f \Delta_j g. $$
\begin{definition}[Hölder--Besov spaces] For any $\alpha \in \RR$ and weight function $z \in \boldsymbol{\rho}(\omega)$ we define the space: $$ \CC^{\alpha}_z = \bg{ \left. \varphi \in \SDO \ \right| \  \norm{ \varphi }_{\CC^{\alpha}_z} = \norm{ 2^{\alpha j} \norm{ \Delta_j \varphi / z }_{L^{\infty}} }_{\ell^{\infty}} < \infty }.$$
We denote with $\mathcal{C}^{\alpha}$ for the space $\CC^{\alpha}_z$ with weight $z = 1$ and use the norm
	$$
	\norm{f}_{\infty , z} = \sup_{x \in \RR} \left\vert \frac{f(x)}{z(x)} \right\vert.
	$$
\end{definition}

The following result is of central importance in Littlewood--Paley theory. The classical proof can be found for example in \cite[Lemma 2.1]{bah}. Here we will just discuss the proof in the weighted case. The main difference is that the inequality does not hold uniformly over all scaling factors $\lambda$. Instead we have to assume that $\lambda$ is bounded away from zero.

\begin{proposition}[Bernstein inequality]\label{prop:bernstein_inequality}
	Let $\mathcal{B}$ be a ball about the origin and let $k \in \NN$ and $z \in \boldsymbol{\rho}(\omega)$ a weight function. Then for any $\lambda \ge \lambda_0 > 0$ and $f \in L^{\infty}_z$ we have:
	 If $\mathrm{supp}(\FF f) \subset \lambda \mathcal{B},$ then
		\begin{align*}
			\max_{| \mu | = k} \norm{\partial^{\mu} f}_{\infty ,z} \lesssim_{\lambda_0, k} \lambda^{k}\norm{f}_{\infty, z}.
		\end{align*}
\end{proposition}

\begin{proof}
Choose a compactly supported function $\psi \in \SO$ with $\psi = 1$ on $\mathcal{B}$  and set $\psi_{\lambda} (\cdot) = \psi(\lambda^{-1} \cdot)$. Then
	\begin{align*}
		\partial^{\mu}f = \partial^{\mu} \FF^{-1}\left( \psi_{\lambda} \FF f \right) = (2 \pi)^{\frac{d}{2}}f * \partial^{\mu} \FF^{-1}\psi_{\lambda}.
	\end{align*}
	Now it is immediate to see that $\FF^{-1}\psi_{\lambda}(x) = \lambda^{d}\FF^{-1}\psi (\lambda x)$, and hence $$ \partial^{\mu} \FF^{-1}\psi_{\lambda} (x) = \lambda^{d + |\mu|} \partial^{\mu}\FF^{-1}\psi (\lambda x).$$ Moreover, since $z \in \boldsymbol{\rho}(\omega)$ there exists $\nu > 0$ such that 
	$$
	\frac{1}{z(x)} \lesssim \frac{e^{\nu \omega(x-y)}}{z(y)} \le \frac{e^{c \omega(\lambda(x-y))}}{z(y)},
	$$
	where in the last step we used that $\omega(x) = |x|^\delta$, $\lambda \ge \lambda_0$ and $c = \nu/\lambda_0^{\delta}$. So eventually we can estimate:
	\begin{align*}
		\norm{\partial^{\mu}f}_{\infty , z} & \lesssim_{\lambda_0}  \bigg\| \int_{\RR} \frac{e^{ c \omega(\lambda(x-y))}}{z(y)}\bv{f(y) \lambda^{d + |\mu|} \partial^{\mu}\FF^{-1}\psi (\lambda (x-y))} dy \bigg\|_{\infty} \\
		& \lesssim \lambda^{k}  \norm{f}_{\infty, z} \big\| e^{ c \omega(\lambda( \cdot ))}\partial^{\mu}\FF^{-1}\psi (\lambda (\cdot)) \lambda^{d}  \big\|_{L^1}\\
		& \lesssim_{\psi, \delta, \theta, z} \norm{f}_{\infty, z} \lambda^{k}
	\end{align*}
	where in the last step we changed variables and used the assumption that $\psi$ is in $\SO$ and the growth assumptions on $z$ to conclude that the second norm is finite uniformly over all $\lambda \ge \lambda_0$.
\end{proof}

A classical consequence of this result is the following.
\begin{corollary}\label{cor:classical_besov_characterisation}
    If $f \in \CC^{\alpha}_z$ then $\partial^{\mu}f \in \CC^{\alpha - |\mu|}_z$ with $$\norm{\partial^{\mu}f}_{\CC^{\alpha - |\mu|}_z} \lesssim \norm{f}_{\CC^{\alpha}_z}$$ uniformly over all distributions $f.$
\end{corollary}

Moreover we can also deduce the characterisation of H\"older--Besov spaces.
\begin{corollary}[\cite{Martin2018}, Lemma~2.1.23]\label{cor:caracterisation_besov_holder}
	For any $\alpha \in (0, \infty) \setminus \mathbb{N}$ and $z \in \boldsymbol{\rho}(\omega)$ we find the equivalence between the following norms:
	$$
		\norm{ f }_{\CC^{\alpha}_z} \asymp \norm{ f}_{\infty, z} + \sum_{\substack{ k \in \mathbb{N}^d \\ |k| = \lfloor \alpha \rfloor } } \left( \| \partial^k f \|_{\infty, z} + \sup_{x} \sup_{|x - y| \le 1}  \frac{|\partial^k f(x) - \partial^k f(y)|}{z(x)|x - y|^{\alpha - \lfloor \alpha \rfloor}} \right)
	$$
	where with $|k|$ we denote the $\ell^1$ norm: $|k| = \sum_{i = 1}^d k^i$.
\end{corollary}

We conclude this section with a first set of paraproduct estimates. For this reason we define the ``commutator'' $C(f,g,h) = (f \para g) \reso h - f (g \reso h).$

\begin{lemma}\label{lem:paraproduct_estimates} Consider $f \in \mathcal{C}^{\alpha}_{ z_1}$, $h \in \mathcal{C}^{\beta}_{ z_2}$ and $g \in \mathcal{C}^{\gamma}_{ z_3}$ and let us write $z = z_1 \cdot z_2,$ where $z_i \in \boldsymbol{\rho}(\omega).$ Then
\begin{align*}
	\norm{f \para g}_{\beta, z} & \lesssim \norm{f}_{\infty, z_1} \norm{g}_{\beta, z_2}, \\
	\norm{f \para g}_{\alpha + \beta, z} & \lesssim \norm{f}_{\alpha, z_1} \norm{g}_{\beta, z_2} & \text{ if } \alpha < 0, \\
	\norm{f \reso g}_{\alpha + \beta, z} & \lesssim  \norm{f}_{\alpha, z_1} \norm{g}_{\beta, z_2} & \text{ if } \alpha + \beta > 0  .
\end{align*}
We also have that for $z = z_1 \cdot z_2 \cdot z_3$:
\begin{align*}
    \| f \para (g \para h) {-} (fg)\para h \|_{\alpha + \beta, z }& \lesssim \| f \|_{\alpha, z_1} \| g \|_{\alpha, z_2} \|h \|_{\beta, z_3} & \text{ if } \alpha > 0, \beta \in \RR, \\ 
	\| C(f,g,h) \|_{\beta + \gamma, z}  & \lesssim \norm{f}_{\alpha, z_1}\norm{g}_{\beta, z_2}\norm{h}_{\gamma, z_3} & \text{ if } \alpha{+}\beta{+}\gamma > 0 \text{ and } \beta{+}\gamma \neq 0.
\end{align*}
\end{lemma}
\begin{proof}
	The first three estimates are shown in~\cite[Lemma~4.2]{jorg}. The estimate for the commutator $C$ is from \cite[Lemma~4.4]{jorg}.
\end{proof}

\subsection{Time Dependence}
Throughout this work we mostly use an arbitrary but finite time horizon $T> 0$ which will be fixed from now on. When we change the time horizon we will explicitly state it. We define $\LLL = \partial_t - \hh \Delta$ and the associated semigroup: $$P_t f (x)= \int\limits_{\RR}\frac{1}{(2 \pi t)^{\frac{d}{2}}} e^{-\frac{|x-y|^2}{2t}} f(y) dy.$$
In this section the aim is to encode the following information:
\begin{enumerate}
    \item Time dependent weights.
    \item Parabolic space-time regularity.
    \item Blow-up at time $t = 0.$
\end{enumerate}
Here as before we follow the notation of \cite{jorg}. For an arbitrary horizon $T_r \ge 0$ we denote by $$X = (X(s))_{s \in [0, T_r]}$$ an increasing sequence of Banach spaces. A typical example could be the sequence $X(s) = \CC^{\alpha}_{e(l + s) }.$ In fact we will use only two kinds of time dependent weights, which we will refer to through the following abuse of notation:
\begin{align*}
	e(l+t)p(a)& = (e(l + t)(\cdot) p(a)(\cdot) )_{t \in [0, T]}, \\
	e(l+t) &= (e(l + t)(\cdot))_{t \in [0, T]},
\end{align*}
where the $t$ on the left-hand side is only a formal way of representing time dependence. In applications it will always be clear whether $t$ is fixed or whether we are considering a time-dependent weight.
Now we define the following space of functions for given $\beta \ge 0$:
\begin{align*}
    \MM^{\beta} X ([T_{\ell}, T_r]) = \big\{ f & :([T_{\ell}, T_r]) \to \mathcal{S}'_{\omega} \text{ such that } \\
	& t\mapsto t^{\beta}f(t) \text{ is continuous from }  [T_{\ell}, T_r] \text{ to } X(T_r) \text{ and } \\
    & \norm{f}_{\mathcal{M}^{\beta}(X)} =  \sup_{ T_\ell < t \le T_r} \norm{t^{\beta}f(t)}_{X(t)} < \infty \big\}
\end{align*}
We also write $C X ([T_{\ell}, T_r])$ for the space $\MM^{0} X ([T_{\ell}, T_r])$. Similarly we can define H\"older continuity through the following norm for $\alpha \in (0,1)$:
$$
    \norm{f}_{C^{\alpha}X ([T_{\ell}, T_r])} = \sup_{T_\ell < t \le T_r} \norm{f(t)}_{X(t)} + \sup_{T_\ell < s,t \le T_r} \frac{\norm{f(t) - f(s)}_{X(t \vee s)}}{|t-s|^{\alpha}}.
$$
So eventually we find the parabolically scaled H\"older spaces with respect to a possibly time dependent $z$ and of regularity $\alpha \in (0,2)$ with a blow-up of order $\beta$ in zero and for parameters $0 \le T_\ell \le T_{r}$: 
\begin{align*}
 \LL^{\beta, \alpha}_{z}([T_{\ell}, T_r]) = \big\{ f& : [T_{\ell},T_r] \to \mathcal{S'}_{ \! \! \omega} \text{, such that} \\
 & \norm{f}_{\LL^{\beta,\alpha}_{z}} = \norm{t \mapsto t^{\beta}f(t)}_{C^{\frac{\alpha}{2}}L^{\infty}_z ([T_{\ell}, T_r])} + \norm{f}_{\MM^{\beta} \CC^{\alpha}_{z}([T_{\ell}, T_r])} <{+}\infty \big\}.
\end{align*}
In general we omit from writing the dependence on the time interval $[0, T]$. Finally, we will write $\LL^{\alpha}_z$ for $\LL^{0, \alpha}_z.$

With these definitions at hand we are ready to go on with our theory. First, we introduce parabolically scaled paraproducts. Let $\varphi: \RR \to \RR_{\ge 0}$ be a smooth function with compact support and total mass $1$ which is \textit{non-predictive}, that is: $$\mathrm{supp}(\varphi) \subset [0, + \infty).$$
Then for any continuous function $f: \RR_{\ge 0} \to X$ (here $X$ is any Banach space) and $i\ge 0$ we define the operator $$Q_i f (t) = \int\limits_{\RR} 2^{2i}\varphi(2^{2i}(t-s))f(s \vee 0) ds = \int\limits_{- \infty}^t 2^{2i}\varphi(2^{2i}(t-s))f(s \vee 0) ds.$$

\begin{remark}
	As in~\cite{kpz,jorg} we silently identify $Q_i f$ with $Q_i \mathds{1}_{t > 0} f$ if $f$ has a blow-up in zero, i.e. if $f \in \MM^{\beta} X$ for $\beta > 0$.
\end{remark}

We have suggestively called $\varphi$ non-predictive because thanks to the condition on its support, $Q_i f(t)$ depends only on $f\big\vert_{[0,t]}$.
Now we can introduce the parabolically scaled paraproduct $$f \ppara g = \sum_{i} (S_{i-1} Q_i f) \Delta_i g.$$ With this definition at hand we obtain a second set of paraproduct estimates. 

\begin{lemma}\label{lem:parabolic_paraproduct_estimates}
Consider $\alpha \in \RR, \gamma <0, \beta \ge 0$. Choose two, possibly time dependent, weights $z_i: \RR_{\ge 0 } \to \boldsymbol{\rho}(\omega),$ for $i = 1,2$ such that $z_i$ is pointwise increasing in time and write $z(t) = z_1(t)  z_2(t)$. Then
\begin{align*}
	t^{\beta}\| f \ppara g(t) \|_{\CC^{\alpha}_{z(t)}} & \lesssim \norm{f}_{\MM^{\beta}L^{\infty}_{z_1}([0,t])} \norm{g(t)}_{\CC^{\alpha}_{z_2(t)}}, \\
 	t^{\beta} \|f \ppara g(t) \|_{\CC^{\alpha {+} \gamma}_{z(t)}} &\lesssim \norm{f}_{ \MM^{\beta}\CC^{\gamma}_{z_1}([0,t])} \norm{g(t)}_{\CC^{\alpha}_{z_2(t)}}.
\end{align*}
Moreover, for $\alpha \in (0,2)$ we find the following estimate
\begin{align*}
	\norm{f \ppara g}_{\LL^{\beta,\alpha}_{z}} \lesssim \norm{f}_{\LL^{\beta, \delta}_{z_1}}(\norm{g}_{C\CC^{\alpha}_{z_2}} + \norm{ \LLL g }_{C\CC^{\alpha - 2}_{z_2}})
\end{align*}
for any $\delta > 0.$ Finally, we also have the following commutation results:
\begin{align*}
	& t^{\beta}\|(\LLL (f \ppara g) -  f \ppara (\LLL g) )(t)\|_{\mathcal{C}^{\alpha + \gamma - 2}_{z(t)}}  \lesssim \norm{f}_{\LL^{\beta,\gamma}_{z_1}([0,t])}\norm{g(t)}_{\mathcal{C}^{\alpha}_{z_2(t)}}, \\
	& t^{\beta}\|(f \ppara g - f \para g)(t)\|_{\CC^{\alpha + \gamma}_{z(t)}} \lesssim \norm{f}_{\LL^{\beta, \gamma}_{z_1}([0,t])} \norm{g(t)}_{\CC^{\alpha}_{z_2(t)}}.
\end{align*}
\end{lemma}

\begin{proof}
	These estimates are shown in~\cite[Lemmas~4.7 - 4.9]{jorg}.
\end{proof}

Now we pass to a result regarding derivatives in $\LL^{\beta, \alpha}_z$ spaces.
\begin{lemma}\label{lem:dervtv_interp}
Consider a parameter $\alpha \in (0,1)$ and a weight $z : \RR_{\ge 0} \to \boldsymbol{\rho}(\omega)$ which is pointwise increasing in time. Then
$$
	\| \partial_x f \|_{\LL^{\beta, \alpha}_z} \lesssim \| f \|_{\LL^{\beta, \alpha {+} 1}_z }.
$$
\end{lemma}
\begin{proof}
	Since $\| \partial_xf(t)\|_{\CC^{\alpha}_{z(t)}} \lesssim \|f(t) \|_{\CC^{\alpha + 1}_{z(t)}}$ we can easily control the spatial regularity. Let us concentrate on the time regularity. We estimate for $t \ge s$:
	$$
	\| \partial_xf(t)t^{\beta} - \partial_xf(s)s^{\beta} \|_{\infty, z(t)} \le \sum_{j = {-}1}^{\infty} \| \partial_x [\Delta_j f(t)t^{\beta} - \Delta_j f(s)s^{\beta} ] \|_{\infty, z(t)}
	$$
	Now let us fix a $j_0$ such that $2^{-j_0} \le |t-s|^{1/2} <2^{- j_0 +1}$. We will use different estimates on small scales and on large scales. Indeed, an application of Bernstein's inequality (Proposition \ref{prop:bernstein_inequality}) gives for the large scales
	\begin{align*}
	\sum_{j = -1}^{j_0} \|\partial_x [\Delta_j f(t)t^{\beta} & {-} \Delta_jf(s)s^{\beta}]\|_{\infty, z(t)} \\
	& \lesssim \sum_{j = -1}^{j_0} 2^j\| \Delta_j[f(t)t^{\beta} {-} f(s)s^{\beta}]\|_{\infty, z(t)} \lesssim 2^{j_0} |t-s|^{\frac{\alpha{+} 1}{2}} \|f\|_{\LL^{\beta,\alpha + 1}_z},
	\end{align*}
	while on small scales
	$$
	\sum_{j = j_0 +1}^{\infty} \| \Delta_j[ \partial_x f(t)t^{\beta} {-} \partial_x f(s)s^{\beta}]\|_{\infty, z(t)} \lesssim \sum_{j = j_0 +1}^{\infty} 2^{{-}j \alpha}\| f\|_{\LL^{\beta,\alpha}_z} \lesssim 2^{-j_0 \alpha} \| f \|_{\LL^{\beta, \alpha}_z}.
	$$
	Substituting $2^{-j_0} \simeq |t-s|^{1/2}$ delivers the required result.
\end{proof}

We conclude the preliminaries by stating some important estimates regarding the heat semigroup, commonly referred to as Schauder estimates. We write $V_{T_\ell}(f)(t) = \smallint_{T_\ell}^t P_{t{-}s}f_s ds.$

\begin{proposition}\label{prop:schauder_estimate} 
	Fix $\alpha \in (0,2)$ and $z \in \boldsymbol{\rho}(\omega)$.
	\begin{enumerate}
		\item For $\gamma \in \RR$ such that $\beta = (\alpha {+} \gamma)/2 \in [0,1)$ we find that:
		\begin{align*}
			\norm{P_{\cdot} f}_{\LL^{\beta, \alpha}_{z} } \lesssim \norm{f}_{\CC^{ - \gamma}_{z}}.
		\end{align*}
		\item If we fix also $a \ge 0$ such that $\alpha {+} 2 a/ \delta \in (0,2)$  and $\beta{+}a/\delta \in [0,1)$ we find that:
		\begin{align*}
			\norm{V_{T_\ell}(f)}_{\LL^{\beta, \alpha}_{e(l+t)}([T_{\ell}, T_r])} \lesssim_{T_h} \norm{f}_{\MM^{\beta}\CC^{\alpha {+} 2a/\delta {-} 2}_{e(l+t)p(a) } ([T_{\ell}, T_r])}
		\end{align*}
		uniformly over all $0 \le T_{\ell} < T_r \le T_h$.
	\end{enumerate}	 
\end{proposition}

\begin{proof}
These are the weighted analogues of the estimates of \cite{kpz} or \cite{singular_GIP} and can be found as well in \cite{jorg}. A slight difference is the dependence on the interval $[T_\ell, T_r].$ Note that
$V_{T_{\ell}}(f)(t) = V(f_{\cdot {+} T_{\ell}})(t{-}T_\ell).$ Thus the proof follows by proving:
$$
\| t\mapsto (t{+}\kappa)^{\beta} V(f)(t)\|_{\LL^{\alpha}_{e(l+t)}} \lesssim_{T_h} \|(t{+}\kappa)^{\beta}f_t\|_{C\CC^{\alpha{-}2{-}2a/\delta}_{e(l+t)p(a)}([T_{\ell}, T_{r}])}
$$
uniformly over $\kappa$ in $[0, T_h].$ This follows from the same calculations as in \cite[Lemma 6.6]{kpz}.
\end{proof}

In the previous result the role of $e(l+t)$ and $p(a)$ only comes into play with the parameter $ a / \delta.$ Although this seems a minor detail in the statement, it is actually the key point that allows us to solve linear singular SPDEs on the whole real line with exponential weights. This approach has been developed by Hairer and Labbé in \cite{pam, multiplicative} and it is also present in \cite{singular, Hairer2013}.

In the next result we show our last product estimate. Since our definition of the $\LL^{\alpha}$ spaces does not allow for negative $\alpha$ we state in the following theorem the classical result of Young integration with parabolic scaling.

\begin{lemma}[Young Integration]\label{lem:young_time_product}
	As before let $z_i, i = \,2$ be pointwise increasing, time-dependent weights. Conisder $ f \in \LL^{\beta,\alpha}_{z_1}$ and $g \in \LL^{\gamma}_{z_2}$ with $\beta \in [0, 1)$ and $\alpha, \gamma \in (0,2)$. If $\alpha {+} \gamma {-} 2 > 0$, we have
	$$ f \cdot \partial_t g \in \LLL \LL^{\beta, \alpha + \gamma}_{z_1 \cdot z_2}$$
	and the following two estimates hold true:
	\begin{align*}
		\| V(f \cdot \partial_t g) \|_{\LL^{\beta, \alpha + \gamma{-}\ve}_{z_1 \cdot z_2}([0, T_h])} & \lesssim_{T_h} \| f \|_{\LL^{\beta, \alpha}_{z_1}([0, T_h])} \| g \|_{\LL^{\gamma}_{z_2}([0, T_h])}, \\
	\|V(f\cdot \partial_t g)\|_{\LL^{\beta, \alpha{+}\gamma {-} 2a/\delta {-}\ve}_{e(l+t)}([T_{\ell}, T_r])} &\lesssim_{T_h} \|f\|_{\LL^{\beta, \alpha}_{e(l{+}t)}([T_{\ell}, T_r])} \|g\|_{\LL^{\gamma}_{p(a)}([T_{\ell}, T_r])}
	\end{align*}
	for any $\ve > 0$ and $0 \le T_\ell \le T_r \le T_h$.
\end{lemma}
\begin{proof}
The proof of this result is the content of Lemma \ref{lem:young_heat_kernel_convolution} and the preceding results.
\end{proof}

The next result shows how to interpolate between different $\LL^{\beta, \alpha}_{z}$ spaces.

\begin{lemma}\label{lem:interpolation_schauder}
	Fix some parameters $\alpha \in (0,2),\beta \in [0,1), \ \ve \in [0, \alpha) \cap [0, 2\beta]$ as well as a time-dependent point-wise increasing weight $z$ and $0 \le T_\ell \le T_r$. Then
	\begin{align*}
		\| f\|_{\LL^{\beta-\ve/2, \zeta}_z ([T_{\ell}, T_r])} \lesssim \| f\|_{\LL^{\beta, \alpha}_z ([T_{\ell}, T_r])}.
	\end{align*}
	for any $\zeta < \alpha{-}\ve.$
Finally, for $\alpha \in (0,2), \beta \in [0,1)$ and $\ve \in [0, \alpha)$
$$
\|f \|_{\LL^{\beta, \alpha {-} \ve}_{z}([T_\ell, T_r]) }  \lesssim \| T_\ell^{\beta}  f(T_\ell) \|_{\CC^{\alpha - \ve}_{ z(T_\ell)}} + (T_r{-}T_\ell)^{\ve / 2}\norm{f}_{\LL^{\beta, \alpha}_{z } ([T_\ell, T_r])  }
$$
\end{lemma}
\begin{proof}
This result is analogous to \cite[Lemma 3.10]{jorg}. We only discuss the first statement. Note that here we allow also for $\ve = 2\beta$. This is possible because, using the same arguments as in the proof of \cite[Lemma 6.8]{kpz} we obtain the following bound (uniformly over $\zeta \le \alpha{-}\ve$):
$$
\sup_{t \in [0, T]}\| f(t) \|_{\CC^{\alpha{-}2\beta}_{z(t)}} + \| f \|_{C^{\alpha/2{-}\ve} L^{\infty}_z} \lesssim \| f\|_{\LL^{\beta, \alpha}_z}.
$$
However, it is a priori not clear that $f: [0,T] \to \CC^{\zeta}_{z(T)}$ is a continuous function. Since we have H\"older continuity in $L^{\infty}_{z(T)}$ of $f$ and a uniform bound in $\CC^{\alpha{-}2\beta}_z$ we can conclude by interpolation, at the price of an arbitrarily small loss of regularity, which explains the strict inequality $\zeta < \alpha{-}\ve$.
\end{proof}

With these results we end our brief introduction to the theory of paracontrolled analysis and Schauder estimates.

\section{The Paracontrolled KPZ Equation}

Here we briefly review the notion of paracontrolled solutions to the KPZ equation first introduced in \cite{kpz} and \cite{singular_GIP}. For counting regularity we will use the index $\alpha$. We will use the index $a$ for counting the polynomial growth of the noise at infinity and we recall that $\delta \in (0,1)$ is used in our definition of ultra-distributions. We will work under the following standing assumptions on the parameters.
\begin{assumption}
$$
\frac{2}{5} < \alpha < \frac{1}{2}, \qquad 0 \le a/\delta < \frac{5\alpha{-}2}{6}
$$
\end{assumption}

In general $\alpha \in (1/3, 1/2)$ will be sufficient. We will need some tighter control only in Section \ref{sect:polymer_measure}. We can compile the following rule-of-thumb table.

\begin{table}[H]
\begin{center}
\begin{tabular}{ccc}
	$\alpha$ & $\delta$ & $a$ \\
	\hline
	$1/2-$ & $1-$ & $0+$
\end{tabular}
\\[0.3em]
\caption{Rule-of-thumb for the Parameters}
\end{center}
\end{table}
Let us introduce the extended data for the KPZ equation. We collect in the next table all the terms involved. Here and throughout this work we will use the notation $X^{(\cdot)} = \partial_xY^{(\cdot)}.$

\begin{table}[H]
\begin{tabular}{rlllcl}
\ & \ & Regularity & \multicolumn{3}{l}{Definition}  \\[0.1em]
\hline \\
\ldelim\{{7.7}{75pt}[$\YY(\theta, Y_0, c^{\RS{lr}}, c^{\scalebox{0.8}{\RS{{Lrl}{Rrl}}}})$]\ & $Y$ & $\LL^{\alpha}_{p(1+a)}$ & $\LLL Y $&$  =  $&$ \theta$\\[0.1em]
\ & $Y^{\scalebox{0.8}{\RS{lr}}}$&$\LL^{2\alpha}_{p(a)}$ &$\LLL Y^{\scalebox{0.8}{\RS{lr}}}$ &$ = $&$ \hh (\partial_xY)^2 - c^{\RS{lr}}$ \\[0.1em]
\ & $Y^{\scalebox{0.8}{\RS{rLrl}}}$&$\LL^{\alpha +1}_{p(a)}$ &$\LLL Y^{\scalebox{0.8}{\RS{rLrl}}} $ &$ = $&$ \partial_xY \partial_xY^{\scalebox{0.8}{\RS{lr}}}$ \\[0.1em]
\ & $Y^{\scalebox{0.8}{\RS{rLrLrl}}}$ & $\LL^{2\alpha +1}_{p(a)}$ &$\LLL Y^{\scalebox{0.8}{\RS{rLrLrl}}} $ &$ = $&$ \partial_xY^{\scalebox{0.8}{\RS{rLrl}}} \reso \partial_xY + c^{\scalebox{0.8}{\RS{{Lrl}{Rrl}}}}$ \\[0.1em]
\ & $Y^{\scalebox{0.8}{\RS{{Lrl}{Rrl}}}}$ & $\LL^{2\alpha + 1}_{p(a)}$ &$\LLL Y^{\scalebox{0.8}{\RS{{Lrl}{Rrl}}}} $ &$ = $&$ \hh (\partial_xY^{\scalebox{0.8}{\RS{lr}}})^2 - c^{\scalebox{0.8}{\RS{{Lrl}{Rrl}}}}$ \\[0.1em]
\ & $\partial_xY^{\scalebox{0.8}{\RS{r}}} \reso \partial_xY$ & $C \mathcal{C}^{2\alpha - 1}_{p(a)}$ & $ $ &$  $&$ $ \\[0.1em]
\ & $X$ & $ C \CC^{\alpha-1}_{p(a)}$ & $X$ & $=$ & $\partial_xY$ \\[0.1em]
\ & $Y^{\scalebox{0.8}{\RS{r}}}$ & $\LL^{\alpha + 1}_{p(a)}$ & $\LLL Y^{\scalebox{0.8}{\RS{r}}} $ &$ = $&$\partial_xY $ \\[0.1em]
\hline
\end{tabular}
\caption{Extended Data of the KPZ Equation}\label{table:kpz}
\end{table}

Here we assume that $\theta \in \LLL C^{\alpha /2} (\RR;C^{\infty}_b(\RR))$ is a (spatially) smooth noise and $C^{\infty}_b(\RR)$ is the space of bounded and infinitely differentiable functions with all derivatives bounded. The reason for assuming only distributional regularity of $\theta$ in the time variable is that we do not want to exclude spatial mollifications of the space-time white noise, which are convenient from a probabilistic point of view because they preserve the Markovian structure of the equation. We solve the equations for the elements in $\YY(\theta)$ by taking all initial conditions equal to zero, except $Y(0) = Y_0$ is assumed to be non-trivial. We are interested in starting the KPZ equation at its invariant measure, and for that purpose it is convenient to let $Y_0$ be of the form $$Y_0(x) = B(x) + Cx,$$ where $B$ is a two sided Brownian motion and $C \in \RR$ (cf. \cite[Section 1.4]{quastel_intro}). Note that we have added $X = \partial_x Y$ to the table because we assume that it has a better behaviour at infinity than $Y$. Indeed, while $Y$ may have superlinear growth at infinity, its derivative $X$ is started in the invariant measure for the rough Burgers equation, which has the growth of white noise on $\RR$, i.e. it grows less than any polynomial.

We now rigorously define the spaces of functions we will work with. For a finite collection $I$ of Banach spaces $Y_i$ we call \textit{product norm} on $\bigtimes_{i \in I} Y_i$ the norm: $\|\cdot \|_{\bigtimes Y_i} = \max_{i \in I} \| \pi_i (\cdot) \|_{Y_i}$ with $\pi_i$ being the projection on the $i-$th coordinate.

\begin{definition}\label{def:ykpz_space}
We shall call $\myi$ be the image of the map $\YY(\theta, Y_0, c^{\BA}, c^{\BD})$ in the space
$$ \LL^{\alpha}_{p(1+a)} \times \LL^{2 \alpha}_{p(a)} \times \LL^{\alpha +1}_{p(a)} \times \LL^{2\alpha +1}_{p(a)} \times \LL^{2\alpha +1}_{p(a)} \times C\mathcal{C}^{2\alpha-1}_{p(a)} \times C\CC^{\alpha - 1}_{p(a)}$$
as we let $(\theta, Y_0, c^{\BA}, c^{\BD})$ vary in $\LLL C^{\alpha /2} (\RR;C^{\infty}_b(\RR)) \times C^\infty_b(\RR) \times \RR \times \RR.$ We define $\mY_{kpz}$ as the closure of the image of $\YY(\theta, Y_0, c^{\BA}, c^{\BD})$ in the above space endowed with the product norm, which we will refer hereafter as $\| \cdot \|_{\mY_{kz}}$. To any $\YY \in \myk$ we associate a distribution $\xi = \LLL Y.$ 
\end{definition}

These tools are sufficient to define paracontrolled solutions to the KPZ equation.
\begin{definition}\label{def:solutions_to_kpz}
We say that $h$ is a paracontrolled solution to the KPZ equation (\ref{eqn:kpz}) with initial condition $\oh \in C(\RR, \RR)$ and with external data $\YY \in \mY_{kpz}$ if there exists an $\kappa \in \RR$ and $\beta^\sharp \in (0, 1), \ \bp \in (0, \frac{\alpha{+}1}{2} )$ such that $h$ is of the form:
$$h =  Y + Y^{\scalebox{0.8}{\RS{lr}}} + Y^{\scalebox{0.8}{\RS{rLrl}}} + h^P, $$ where $h^P$ is paracontrolled by $Y^{\scalebox{0.8}{\RS{r}}}$ in the sense that
$$
    \LL^{\bp, \alpha{+}1}_{e(\kappa)} \ni h^P = h' \ppara Y^{\scalebox{0.8}{\RS{r}}} + h^{\sharp},
$$
with $h'$ in $\LL^{\bp, \alpha}_{e(\kappa)}$ and $h^{\sharp}$ in $\LL^{\beta^\sharp, 2\alpha +1}_{e(\kappa)}$, and if the following conditions are satisfied:
\begin{equation}\label{eqn:h^P}
\begin{aligned}
    \LLL h^P = & \ \LLL(Y^{\BC} \! \!+ Y^{\BD} \!) + (X X^{\BB} {-} X \reso X^{\BB}) + X^{\BA}X^{\BB} + \hh (X^{\BB})^2 \\
    & \ +  (X + X^{\BA} + X^{\BB})\partial_xh^P + \hh (\partial_xh^P)^2, \\
     h^P(0) = & \ \oh - Y(0),
\end{aligned}
\end{equation}
and
\begin{equation}\label{eqn:h'}
    h' = X^{\BB} + \partial_xh^P.
\end{equation}
\end{definition}

\begin{remark}
Note that all the terms involved in the last equation are well defined. In particular, the product $X \cdot \partial_x h^P$ is well defined by applying the commutation results for paraproducts that we introduced in the preliminaries. The two crucial ingredients for this purpose are the paracontrolled nature of $h^P$ and the fact that the resonant product $\partial_xY^{\scalebox{0.8}{\RS{r}}} \reso \partial_xY$ is given \textit{a priori} in $\mY_{kpz}$. 

For smooth noises $\theta$ the definition amounts to $h$ satisfying the equation
$$
    \LLL h = \hh (\partial_xh)^2 - c^{\BA} + \theta.
$$
In this sense a paracontrolled solution $h$ to the KPZ equation with white noise forcing solves
\begin{align*}
	\LLL h = \hh (\partial_xh)^{\diamond 2} + \xi,
\end{align*}
where $(\partial_xh)^{\diamond 2} =  `` (\partial_xh)^{2} - \infty ",$ with $\infty$ being the limit $ \lim\limits_{n \to \infty } c^{\BA}_{n}$ as some smooth noise $\theta^{n}$ converges to $\xi$. In particular, $(\partial_xh)^{\diamond 2}$ is a continuous functional on the space of paracontrolled distributions.
\end{remark}
A similar argument holds for the RHE \eqref{eqn:she}. It is possible to define paracontrolled solutions to a renormalised version of the equation:
\begin{equation}\label{eqn:rhe_renormalised}
\begin{aligned}
    \LLL w  = w \diamond  \xi,  \ \  w(0)  = \mathfrak{w}_0
\end{aligned}
\end{equation}
with $w \diamond \xi =$\lqt$w (\xi - \infty )$\rqt with $\infty = \lim c^{\BA}.$

\begin{definition}\label{def:rhe_solutions}
We say that $w$ is a paracontrolled solution to the RHE equation (\ref{eqn:rhe_renormalised}) with initial condition $\mathfrak{w_0}$ of the form $\mathfrak{w}_0 = w_0 e^{Y(0) + Y^{\BA}(0) + Y^{\BB}(0)}$ with $w^0 \in \CC^{\beta}_{e(l)}$, for some $\beta \in (0, 2\alpha{+}1]$, and external data $\YY \in \mY_{kpz}$ if there exists a $\kappa \in \RR$ such that $w$ is of the form:
\[
w = w^P e^{Y {+} Y^{\BA} {+} Y^{\BB}}, \ \ \LL^{\bp, \alpha{+}1}_{e(\kappa)} \ni w^P = w' \ppara Y^{\BZ} + Y^{\sharp} 
\]
 for $w' = X^{\BB}w^P + \partial_x w^P \in \LL^{\bp, \alpha{-}\ve}_{e(\kappa)}$ and $w^{\sharp} \in \LL^{\hb, 2\alpha{+}1{-}\ve}_{e(\kappa)}$ with $\ve \in (6a/\delta, 3\alpha{-}1)$, where we define 
 \begin{equation}\label{eqn:defn_beta_prime_and_hat}
 \hat{\beta} = \frac{2\alpha{+}1{-}\beta}{2},  \qquad \bp = \frac{\alpha{+}1{-}\beta}{2} \vee 0,
 \end{equation}
 and such that $w^P$ solves the equation
\begin{equation}\label{eqn:w^p}
\begin{aligned}
    \LLL w^P = & \ \big[(X X^{\BB} {-} X \reso X^{\BB}) + \LLL(Y^{\BC} {+} Y^{\BD})  + X^{\BA} X^{\BB} + \hh(X^{\BB})^2\big]w^P \\
    & \ + [X {+} X^{\BA} {+} X^{\BB}]\partial_x w^P \\
    w^P(0) =& \  w_0.
\end{aligned}
\end{equation}
\end{definition}

The existence of global in space solutions to the RHE is already established in \cite{multiplicative}. In Section \ref{sect:abstract_solutions} we review their approach and prove an existence result for the paracontrolled setting (Proposition \ref{prop:existence_rhe}).

Now we briefly discuss how white noise can be lifted to extended data for the KPZ equation.

\begin{theorem}[Renormalisation]\label{thm:renormalisation}
Let $\xi$ be a white noise on $[0,T]\times\RR$, let $B$ be an independent two-sided Brownian motion on $\RR$, and let $C \in \RR$. Then for any $\alpha < 1/2$ and $a> 0$ (see Table \ref{table:kpz}), $(\xi, B + Cx)$ is almost surely associated to a $\mathbb{Y}(\xi, B + Cx) \in \mathcal{Y}_{kpz}$: There exists a sequence $(\xi^{n}, Y^n_0, c^{\BA}_{n}, c^{\BD}_{n})$ in $\smooth \times C^\infty_b \times \RR \times \RR$ such that almost surely $(\xi^{n}, Y^n_0) \to (\xi, B + Cx)$ in the sense of distributions and such that
\[
\YY^{n} = \YY(\xi^{n}, Y^n_0, c^{\BA}_{n}, c^{\BD}_{n}) \to \YY(\xi, B + Cx),
\]
where the convergence is in $L^p(\Omega; \mathcal{Y}_{kpz})$, for all $p \in [1, {+}\infty)$. Moreover, while $\xi^{n}$ and $Y^n_0$ are of course random processes, the constants $c^{\BA}_{n}, c^{\BD}_{n}$ can be chosen deterministic. Finally also the following asymmetric product converges:
	\[
	\partial_x Y^{\BZ, n} \reso \partial_x Y \to \partial_x Y^{\BZ} \reso \partial_x Y \text{ in } L^p(\Omega; C\CC^{2\alpha{-}1}_{p(a)}).
	\]
\end{theorem}

\begin{proof}
	Let $\tilde{\xi}^n(\psi) = \xi|_{[-n,n]}(\sum_{k \in \mathbb{Z}} \psi(\cdot,2 k n + \cdot))$ and 
	\[\tilde Y^n_0 (\psi) = (B+Cx)|_{[-n,n]}\bigg(\sum_{k \in \mathbb{Z}} \psi(\cdot,2 k n + \cdot)\bigg)\]
	be the (spatial) $2n$-periodization of $\xi$ and $B + Cx$, respectively. Let $\varphi \in C^\infty_c(\RR)$ be even and such that $\varphi(0) = 1$ and define
	\[
		\xi^n = \varphi(n^{-1} \partial_x) \tilde \xi^n = \FF^{-1} ( \varphi(n^{-1} \cdot) \FF \tilde \xi^n),\qquad Y^n_0 = \varphi(n^{-1} \partial_x) \tilde Y^n_0 = \FF^{-1} ( \varphi(n^{-1} \cdot) \FF \tilde Y^n_0)
	\]
	as the spatial regularization of $\tilde \xi^n$ respectively $\tilde Y^n_0$ through the Fourier multiplier $\varphi(n^{-1}\cdot)$. It is not hard to show that $(\xi^n, Y^n_0) \in  \LLL C^{\alpha/2}(\RR; C^\infty_b(\RR)) \times C^\infty_b(\RR)$. In Section~9 of~\cite{kpz} the construction of $\mathbb Y(\xi, B)$ is performed in the periodic case, and slightly adapting the arguments of that paper we also obtain the convergence in our setting (for $C=0$): it suffices to change the definition of $E= \mathbb Z \setminus \{0\}$ to $E = \RR$ and to replace $L^p(\mathbb T)$ by $L^p(\RR, p(a))$ for $a > 1/p$ in the computations following equation (84). Since $p$ can be arbitrarily large, $a>0$ can be as small as we need.
	
	It remains to treat the case $C\neq 0$. But adding $C x$ to $Y_0$ only results in changing $Y(t) \to Y(t) + C x$ (here we used that spatial convolution with the heat kernel leaves $Cx$ invariant, because it is a harmonic function). Also, the approximations to $Cx$ are  smooth uniformly in $n$, and therefore adding these additional terms does not change the regularities or divergences. The result regarding the asymmetric resonant product follows along the same lines: for clarity it is postponed to Lemma \ref{lem:renormalisation_for_measure}.
\end{proof}

\renewcommand{\TD}{n\mathbb{T}}

In view of this result let us fix any $\YY \in \mY_{kpz}$ and prove that the KPZ equation driven by $\YY$ has a solution. We also fix a sequence $(\theta^n, Y^n_0, c^{\BA}_n, c^{\BD}_n) \in  \smooth \times C^\infty_b \times \RR \times \RR$ such that $\YY^n := \YY(\theta^n, Y^n_0, c^{\BA}_n, c^{\BD}_n) \to \YY$ in $\mY_{kpz}$.
\begin{assumption}\label{assu:initial_condition}
	We assume that the initial condition $\oh$ for the KPZ equation is of the form:
	$$
	    \oh - Y(0) \in \CC^{\beta}_{p(\delta)} 
	$$
	for some $\beta \in (0, 2\alpha{+}1]$ and that there exists a sequence $\oh^n, Y^n(0)$ in $C^{\infty}_b(\RR)$ such that:
	$$
	\oh^n  - Y^n(0) \to \oh - Y(0) \text{ in } \CC^{\beta}_{p(\delta)}.
	$$
\end{assumption}

\begin{remark}
In particular, we can choose any initial condition $\bar{h}$ in the space $\CC^{\beta}_{p(\delta)}.$ Indeed, in Theorem \ref{thm:renormalisation} we can set $C=0$ in the initial condition for $Y$, and then $Y(0) \in \CC^\beta_{p(\delta)}$. It is canonical to assume a bit of H\"older regularity for the initial condition of the KPZ equation, this is already needed for the equation on the torus~\cite{hai:solving_kpz, kpz}. The constraint on the growth is not so natural, in the case of smooth $\xi$ we would expect that at least subquadratic growth is sufficient. But with our methods sublinear growth is the best we can hope for, because we need $e^{\bar h}$ to be a tempered ultra-distribution.
\end{remark}

For the smooth data $\YY^n$ and initial condition $\oh^n$ we can solve the KPZ equation.

\begin{proposition}\label{prop:exist_smooth_KPZ}
For $\YY^n$ and $\oh^n$ as above there exists a unique paracontrolled solution $h^n$ to the KPZ equation as in Definition \ref{def:solutions_to_kpz}, with $h, \ h'^{, n}$ and $h^{\sharp, n}$ in $C^{\infty}_b(\RR)$ as well as in $\LL^{\alpha{+}1}, \ \LL^{\alpha}, \ \LL^{2\alpha{+}1}$, respectively.
\end{proposition}
\begin{proof}
This is a classical application of the Schauder estimates, cf. \cite[Section 4]{kpz}. Global existence in time follows from a partial Cole-Hopf transform, since $v^n = e^{h^n-Y^n}$ solves the linear equation
	\[
		\LLL v^n =  v^n \left(\frac12(\partial_x Y^n)^2 {-} c^{\BA}_n\right) + \partial_x v^n \partial_x Y^n
	\]
	with continuous-in-time and smooth-in-space data.
\end{proof}

\subsection{Existence}

First, we will prove an priori estimate for the smooth solutions $h^n(t,x)$ to the KPZ equation, similar to the one from \cite[Corollary 7.4]{kpz}. Let us consider a constant $M>0$ such that
$$
\sup_n  \| \YY^n\|_{\mY_{kpz}} {+}  \sup_n \| \oh^n - Y^n(0)\|_{\CC^{\beta}_{p(\delta)}}  \le M.
$$

\begin{lemma}\label{lem:lower_estimate}
We have uniformly over $n \in \NN$, $t \in [0,T]$ and $x \in \RR$
\begin{equation}\label{eqn:lower_a_priori_estimate_kpz}
    h^{P,n} (t,x) = [h^n {-} Y^n {-} Y^{\BA \!, n} {-} Y^{\BB,n}] (t,x)  \gtrsim_{M} -(1 {+} |x|)^{\delta}.
\end{equation}
\end{lemma}

\begin{proof}
	Recall that the function $h^{P,n}$ solves
\begin{equation*}
\begin{aligned}
    \LLL h^{P,n} = & \ \LLL(Y^{\BC,n} \! \!+ Y^{\BD,n} \!) + (X^n X^{\BB, n} {-} X^n \reso X^{\BB,n}) + X^{\BA,n}X^{\BB,n} + \hh (X^{\BB,n})^2 \\
    & \ +  (X^n + X^{\BA,n} + X^{\BB,n})\partial_xh^{P,n} + \hh (\partial_xh^{P,n})^2, \\
     h^{P,n}(0) = & \ \oh^n - Y^n(0).
\end{aligned}
\end{equation*}
    By comparison, e.g. \cite[Lemma 2.3]{lieb}, we see that $h^{P,n} \ge - v^n$ where the latter solves the following equation:
    \begin{equation*}
        \begin{aligned}
            \LLL v^n = & \ -[\LLL(Y^{\BC, n} \! {+} Y^{\BD, n} \!) + (X^n X^{\BB, n} {-} X^n \reso X^{\BB, n}) + X^{\BA, n}X^{\BB, n} + \hh (X^{\BB, n})^2] \\
    & \ +  (X^n {+} X^{\BA, n} {+} X^{\BB, n})\partial_xv^n, \\
    v^n(0) = & \ -[\oh^n - Y^n(0)].
        \end{aligned}
    \end{equation*}
    Now we find an upper bound for $v^n$. We consider the transformation $\tilde{u}^n = \exp (v^n)$ which solves the equation
    \begin{equation*}
        \begin{aligned}
            \LLL \tilde{u}^n = & \ -[\LLL(Y^{\BC, n} \! {+} Y^{\BD, n} \!) + (X^n X^{\BB, n} {-} X^n \reso X^{\BB, n}) + X^{\BA, n}X^{\BB, n} + \hh (X^{\BB, n})^2]\tilde{u}^n  \\
            & \ {+}  (X^n {+} X^{\BA, n} {+} X^{\BB, n})\partial_x\tilde{u}^n - \hh (\partial_x\tilde{u}^n)^2/\tilde{u}^n,
        \end{aligned}
    \end{equation*}
   with $\tilde{u}^n(0) = \exp( -[\oh^n {-} Y^n(0)])$. Again, by comparison it follows that $\tilde{u}^n \le u^n$ with the latter solving
    \begin{equation*}
        \begin{aligned}
            \LLL u^n = & \ -[\LLL(Y^{\BC, n} \! {+} Y^{\BD, n} \!) + (X^n X^{\BB, n} {-} X^n \reso X^{\BB, n}) + X^{\BA, n}X^{\BB, n} + \hh (X^{\BB, n})^2]u^n \\
    & \  +  (X^n {+} X^{\BA, n} {+} X^{\BB, n})\partial_x u^n,
        \end{aligned}
    \end{equation*}
with initial condition $u^n(0) =  \exp( -[\oh^n {-} Y^n(0)])$. Up to a sign this equation is just (\ref{eqn:w^p}). Proposition \ref{prop:existence_rhe} and Assumption \ref{assu:initial_condition} then imply that this equation admits a unique paracontrolled solution $u^n$ such that for a sufficiently large $\kappa$
    $$
    	\sup_{n}\| u^n \|_{\LL^{\beta^\prime, \alpha+1}_{e(\kappa)}} < {+} \infty.
    $$
Hence by Lemma \ref{lem:interpolation_schauder} $u^n$ is uniformly bounded in $C\CC^\zeta_{e(\kappa)}$ for some $\zeta >0$. Indeed, since $\beta>0$ it follows that $\bp < \frac{\alpha{+}1}{2}$. We can conclude by the monotonicity of the logarithm.
\end{proof}

\begin{remark}
	A lower bound for $h^n$ can be formally derived using the Feynman-Kac formula and Jensen's inequality as follows:
	\begin{align*}
		h^{n}(t,x) & = \log \EE\left[\exp\left(\int_0^t \theta^n(t-s, x+W_s) d s\right) \right] \\
		& \ge - \log \EE\left[\exp\left( - \int_0^t \theta^n(t-s, x+W_s) d s\right) \right],
	\end{align*}
	where $W$ is a Brownian motion and we recall that $\theta^n$ is deterministic. We can then use the Feynman-Kac formula once more to derive an upper bound for the expectation on the right hand side, which leads to a lower bound for $h^n$. Since in general $\theta^n$ is only a distribution in time and also we are interested in bounds for $h^{P,n}$ and not for $h^n$, we argue through the comparison principle instead.
\end{remark}

\begin{remark}
	The previous proposition provides us with a pathwise quantitative lower bound for the solution $w$ to the rough heat equation \eqref{eqn:rhe_renormalised}. Lower bounds for the \emph{stochastic} heat equation are classical by now. Mueller \cite{Mueller} proved that the solutions are strictly positive even when started in a nonnegative, nonzero initial condition (see also \cite[Theorem~5.1]{Cannizzaro2017Malliavin} for a pathwise version of this result), while in \cite{Conus} there are tight estimates regarding an upper bound for the solution. These results already relied on comparison principles, but only with respect to the initial condition. While the lower bound of~\cite{Cannizzaro2017Malliavin} is also pathwise, it is, along with the other quoted lower bounds bounds, only qualitative and gives no quantitative control. The price we pay for our result is that we restrict ourselves to strictly positive initial conditions, which satisfy Assumption \ref{assu:initial_condition}.
\end{remark}

Now we show that the sequence $h^n$ converges to some $h$. In the following lemma we collect the results regarding the rough heat equation.

\begin{lemma}\label{lem:convergence_of_RHE}
Let us consider $w^{P,n} = e^{h^n {-} Y^n {-} Y^{\BA \!, n} {-} Y^{\BB \!, n}}$. There exists a $\kappa \ge 0$ such that:
\[
    w^{P,n} \to w^P = we^{{-} Y{-} Y^{\BA} {-} Y^{\BB}} \text{ in } \LL^{\bp, \alpha + 1}_{ e(\kappa)},
\]
where $w$ solves the rough heat equation \eqref{eqn:rhe_renormalised} on the entire space with initial condition $w_0 = e^{\oh}$.
\end{lemma}
\begin{proof}
The initial condition $w^{n}(0) = e^{\oh^n}$ is of the form $w^n_0 e^{Y^n(0)}$ with $w^n_0$ converging to $w_0$ in $\CC^{\beta}_{e(l)},$ for some $l \in \RR$. Indeed this follows from Assumption \ref{assu:initial_condition} and Lemma \ref{lem:exp}, since we know that
\[
	\lim_{n} \| (\oh^n {-} Y^n(0)) - (\oh {-} Y(0))\|_{\CC^{\beta}_{p(\delta)}} =0.
\]
Thus, the first result is a consequence of Proposition \ref{prop:existence_rhe}.
\end{proof}

This lemma and the previous lower bound allow us to deduce the convergence of $h^n$ by exploiting the \textit{continuity} of the Cole-Hopf transform. \textit{A priori} it is not clear why taking the logarithm is a continuous operation, since it has a singularity in zero.

\begin{proposition}\label{prop:convergence_of_hn}
There exists $\kappa \ge 0$ such that: \begin{equation}\label{partial_convergence}
h^{P,n} = h^n {-} Y^n {-} Y^{\BA \!,n} {-} Y^{\BB \! {,} n} \longrightarrow h {-} Y {-} Y^{\BA} {-} Y^{\BB} \stackrel{\text{def}}{=} h^P \text{ in } \LL^{\bp, \alpha +1}_{e(\kappa)},
\end{equation}

Moreover $h = \log(w)$, where $w$ is the solution to the rough heat equation with initial condition $e^{\oh}$ and $h^P = \log(w^P)$. In addition, we have a sub-linear bound for $h^P$:
\begin{equation}\label{eqn:lower_polynomial_bound_h}
\sup_{t \in [0, T]} \| h^P(t,\cdot) \|_{L^{\infty}_{p(\delta)}} < {+} \infty.
\end{equation}
\end{proposition}

\begin{proof}
First, we use the results from Lemma \ref{lem:lower_estimate}, so that we can find a $C > 0$ and an $r \le 0$ such that 
\[
    w^{P, n}(t,x) = \exp( h^n {-} Y^n {-} Y^{\BA \!, n} {-} Y^{\BB \!, n} )(t,x) \ge C e(r|x|^\delta)
\]
In view of this and Lemma \ref{lem:convergence_of_RHE}, we can apply Lemma \ref{lem:log}, which guarantees that up to choosing a larger $\kappa$ 
\[
    \log (w^{P,n}) \longrightarrow \log (w^P) = h^P \text{ in } \LL^{\bp, \alpha {+}1}_{e(\kappa)}.
\]

Finally, the lower bound in $L^{\infty}_{p(\delta)}$ for $h^P$ follows from Lemma \ref{lem:lower_estimate}. The upper bound follows from the monotonicity of the logarithm and the Cole-Hopf transform.
\end{proof}

We have found a function $h$ which is a candidate for being a solution to the KPZ equation on the whole real line. We have shown that $h$ is of the form $$h = Y {+} Y^{\BA} {+} Y^{\BB} {+} h^P$$ with $h^P \in \LL^{\bp, \alpha + 1}_{e(\kappa)}.$ Now we want to prove that $h^P$ is of the form 
\[
h^P = h' \ppara Y^{\scalebox{0.8}{\RS{r}}} + h^{\sharp}
\]
with $h' \in \LL^{\bp, \alpha}_{e(\kappa)}$ and $h^{\sharp} \in \LL^{\beta^\sharp, 2\alpha + 1}_{e(\kappa)}$ for some $\beta^\sharp \in (0,1)$. We observe that $h^n$ is already paracontrolled, since we have started with a paracontrolled solution. This allows us to control the derivative term.
 
\begin{lemma}[Convergence of the Derivative Term]\label{lem:convergence_of_derivative}
There exists a $\kappa \ge 0$ such that
\[
h'^{, n} \longrightarrow h' \stackrel{def}{=} X^{\BB} + \partial_x h^P \text{ in } \LL^{\bp, \alpha}_{e(\kappa)}.
\]
\end{lemma}
\begin{proof}
Equation (\ref{eqn:h'}) from Definition \ref{def:solutions_to_kpz} and the fact that $h^n$ is a solution to the KPZ equation tell us that $$h'^{, n} =\partial_x (Y^{\scalebox{0.8}{\RS{rLrl}},n} + h^{P,n}).$$
Now both therms on the right-hand side of this equation converge in $\LL^{\bp, \alpha}_{e(\kappa)}$ for an appropriate $\kappa \ge 0$. Indeed, it follows from Proposition \ref{prop:convergence_of_hn} that $h^{P,n} {+} Y^{\BB \!, n}$ converges in $\LL^{\bp, \alpha + 1}_{e(\kappa)}$. By Lemma \ref{lem:dervtv_interp} this is enough to obtain convergence in $\LL^{\bp, \alpha}_{e(\kappa)}$ of the spatial derivative.
\end{proof}

Now we consider the rest term $h^{\sharp}.$ Here we use a different argument.

\begin{lemma}[Convergence of the Rest Term]\label{lem:convergence_of_sharp}
 There exists a $\kappa \ge 0$ such that the sequence $h^{\sharp, n}$ converges to a function $h^{\sharp}$ in $\LL^{\beta^\sharp, 2 \alpha +1}_{e(\kappa)}$, for $\beta^\sharp = \bp \vee (1 {-}\beta)$. Moreover $h^{\sharp}$ satisfies:
 $$
 	h^{\sharp} = h^P - h' \ppara Y^{\scalebox{0.8}{\RS{r}}},
 $$
as well as the equation:
\begin{equation}\label{eqn:h^sharp}
\begin{aligned}
	\LLL h^{\sharp} = Z(\YY, h^P, h') + X \reso \partial_x h^{\sharp}, \ \ h^{\sharp}(0) = \oh - Y(0),
\end{aligned}
\end{equation}
where we define $Z$ as
\begin{align*}
	Z(\YY, h^P, h')  = & \ \LLL (Y^{\BC}{+}Y^{\BD}) {+} X \para X^{\BB} {+} X^{\BA}X^{\BB} {+} \hh(X^{\BB})^2  {+} \hh (\partial_x h^P)^2 \\
	& + \ (X^{\BA} {+} X^{\BB})\partial_x h^P {+} X \para \partial_x h^P {+} X \reso \partial_x(h' \ppara Y^{\scalebox{0.8}{\RS{r}}})  \\
	& +  \bq{ h' \para \LLL(Y^{\scalebox{0.8}{\RS{r}}}){-} \LLL(h' \ppara Y^{\scalebox{0.8}{\RS{r}}})}.
\end{align*}
\end{lemma}
\begin{proof}
Since $h^n$ is a paracontrolled solution to the KPZ equation we know from Definition~\ref{def:solutions_to_kpz} that $h^{\sharp, n}$ satisfies the equation 
\begin{align*}
	\LLL h^{\sharp, n} & = Z(\YY^n, h^{P,n}, h'^{,n}) + X^n \reso h^{\sharp, n}, \ \ h^{\sharp, n}(0) = \oh^n - Y^n(0).
\end{align*}
Now, $\oh^n - Y^n(0)$ converges to $\oh - Y(0)$ in $\mathcal{C}^{\beta}_{p(\delta)}$ and $\YY^n$ converges to $\YY$ in $\mY_{kpz}$. Moreover from the previous results we know that $h^{P,n}$ converges to $h^P$ in $\LL^{\bp, \alpha{+}1}_{e(\kappa)}$ as well as that  $h'^{,n}$ converges to $h'^{,n}$ in $\LL^{\bp, \alpha}_{e(\kappa)}$. At this point, since $\beta >0$ (upon choosing a $\zeta^\prime$ small enough), we can conclude by the continuous dependence on the parameters from Proposition \ref{prop:existence_sharp}: 
$$h^{\sharp, n} \to h^{\sharp} \text{ in } \LL^{\beta^{\sharp}, 2 \alpha + 1}_{e(\kappa)},$$
up to taking a larger $\kappa,$ where $h^{\sharp}$ is the solution to equation (\ref{eqn:h^sharp}). This proves the result.
\end{proof}

The two lemmata above suffice to show that $h$ is a paracontrolled solution to the KPZ equation. We collect all the information about $h$ in the following theorem.

\begin{proposition}\label{prop:paracontrolled_solution_KPZ}
For any $\YY \in \mY_{kpz}$ and initial condition $\oh$ satisfying Assumption \ref{assu:initial_condition} the function $h$ constructed in Proposition \ref{prop:convergence_of_hn} is a paracontrolled solution to the KPZ equation as in Definition \ref{def:solutions_to_kpz}.	
\end{proposition}

\begin{proof}
    That $h$ has the correct structure follows from Proposition \ref{prop:convergence_of_hn}, Lemma \ref{lem:convergence_of_derivative} and Lemma \ref{lem:convergence_of_sharp}. In addition, $h^P$ solves Equation \eqref{eqn:h^P}, since $h^{\sharp}$ solves Equation \eqref{eqn:h^sharp}. 
\end{proof}

\subsection{Uniqueness}
It is  a rule-of-thumb that in order to obtain uniqueness for PDEs on the entire space some growth assumptions are needed in order to avoid solutions that do not have physical meaning. We will work under the assumption of sublinear growth in $L^{\infty}.$ This is mainly due to the fact that we work within the framework of the Cole-Hopf transform. First, we show that the exponential map preserves the paracontrolled structure of a solution.

\begin{lemma}
	Let $\hb, \bp$ be defined as in Definition \ref{def:rhe_solutions} for some $\beta >0$. Consider a function $h^P \in \LL^{\bp, \alpha +1}_{e(l)}$ that is paracontrolled, in the sense that $h^P = h' \ppara Y^{\scalebox{0.8}{\RS{r}}} + h^{\sharp}$ with $h' \in \LL^{\bp, \alpha}_{e(l)}$ and $h^{\sharp} \in \LL^{\beta^\sharp, 2\alpha + 1}_{e(l)}$, with $\beta^\sharp = \bp \vee (1 {-} \beta)$ and suppose that $h^\sharp$ solves Equation \eqref{eqn:h^sharp} for some initial condition. Suppose moreover that $\| h^P \|_{\infty, p(\delta) } < {+}\infty.$ Then the exponential $w^P = \exp (h^P)$ satisfies:
    $$
    w^P = w' \ppara Y^{\scalebox{0.8}{\RS{r}}} + w^{\sharp}
    $$
    with $ w' = w^P h' \in \LL^{\bar{\beta}^\prime, \alpha}_{e(\kappa)} $ and $w^{\sharp} \in \LL^{\bar{\beta}^{\sharp}, 2\alpha +1}_{e(\kappa)} $ for an appropriate $\kappa \ge 0$ and some $\bar{\beta}^{\prime}, \bar{\beta}^{\sharp}\in (0,1)$.
\end{lemma}

\begin{proof}
	It follows from the growth assumptions on $h^P$ as well as from Lemma \ref{lem:exp_2} that $w^P$ lies in $\LL^{\bp, \alpha{+}1}_{e(\kappa)}$ and $w^P h'$ lies in $\LL^{\bar{\beta}^\prime, \alpha}_{e(\kappa)}$ for some $\kappa$ large enough and $\bar{\beta}^\prime > (\bp{-}1/2)\vee 0 {+} \bp$.
    We still need to show that $$\LLL (w^P {-} w' \ppara Y^{\scalebox{0.8}{\RS{r}}}) \in \MM^{\bar{\beta}^\sharp} \CC^{2\alpha -1}_{e(\kappa)} + \LLL (\LL^{\bar{\beta}^\sharp, 2\alpha +1}_{e(\kappa)}).$$ Indeed 
    \begin{align*}
        \LLL & (w^P {-} w' \ppara Y^{\scalebox{0.8}{\RS{r}}}) = w^P \LLL (h'\ppara Y^{\scalebox{0.8}{\RS{r}}} {+} h^{\sharp}) - \hh w^P (\partial_x h^P)^2 -  w' \para X + \MM^{\bar{\beta}^\prime}\CC^{2\alpha -1}_{e(\kappa)} \\
        & = w^P \bigg( h' \para X {-} \big[ h' \para \LLL(Y^{\BZ}){-} \LLL(h' \ppara Y^{\BZ}) \big] \\
        & \ \ \ \ {+}  \left[ Z(\YY, h^P, h^\prime) {-}\hh (\partial_x h^P)^2 {+} X\reso \partial_x h^\sharp \right] \bigg) -  w' \para X + \MM^{\bar{\beta}^\prime}\CC^{2\alpha -1}_{e(\kappa)} \\
        & = w^P \left[ Z(\YY, h^P, h^\prime) {-}\hh (\partial_x h^P)^2{-} \big[ h' \para \LLL(Y^{\BZ}){-} \LLL(h' \ppara Y^{\BZ}) \big] {+} X \reso \partial_x h^\sharp \right] {+} \MM^{\bar{\beta}^\prime}\CC^{2\alpha -1}_{e(\kappa)}
    \end{align*}
    where we have applied the paraproduct estimates from Lemmata \ref{lem:parabolic_paraproduct_estimates} and \ref{lem:paraproduct_estimates}. We now consider one term at a time. Let us start with the product $w^P (X\reso \partial_xh^\sharp)$. We work with $\beta$ small enough, so that we always have a non-trivial blow-up. The opposite case is simpler. Since $X\reso \partial_xh^\sharp$ has regularity $3\alpha{-}1$. Applying Lemma \ref{lem:interpolation_schauder} we thus find:
    \begin{align*}
    	w^P (X\reso \partial_xh^\sharp) \in \MM^{\beta^{\sharp}} \CC^{\ve}_{e(\kappa)}
    \end{align*}
     for $\ve \in (0, \beta)$.
     
    Next we treat the term $w^P \big[ Z(\YY, h^P, h^\prime) {-}\hh (\partial_x h^P)^2 {-} [ h' \para \LLL(Y^{\BZ}){-} \LLL(h' \ppara Y^{\BZ}) ]\big]$. We find:
\begin{align*}
	w^P \big[ Z(\YY, & h^P, h^\prime) {-}\hh (\partial_x h^P)^2 {-} [ h' \para \LLL(Y^{\BZ}){-} \LLL(h' \ppara Y^{\BZ}) ] \big] = w^P \LLL(Y^{\BC}{+}Y^{\BD}) \\
	& {+} w^P\big[ X \para X^{\BB} {+} X^{\BA}X^{\BB} {+} \hh(X^{\BB})^2 {+} (X^{\BA} {+} X^{\BB})\partial_x h^P {+} X \para \partial_x h^P {+} X \reso \partial_x(h' \ppara Y^{\scalebox{0.8}{\RS{r}}})  \big].
\end{align*}	    
    Here applying Lemma \ref{lem:young_time_product} we have that the term on the first row lies in $\LLL (\LL^{(\bp{-}1/2)\vee 0, 2\alpha {+}1}_{e(\kappa)})$. The terms on the second row on the other hand lie in $\MM^{\bar{\beta}^\prime}\CC^{2\alpha{-}1}$. Hence, we can conclude that for some $\bar{\beta}^\sharp\in (0,1)$ the statement of the Lemma is true.
\end{proof}

\begin{theorem}\label{thm:existence_uniqueness_KPZ}
	For any initial condition $\bar{h}$ satisfying Assumption \ref{assu:initial_condition}, there exists a unique paracontrolled solution to the KPZ equation in the sense of Definition \ref{def:solutions_to_kpz} for $\bp$ as in Definition \ref{def:rhe_solutions} and $\beta^\sharp$ as in Lemma \ref{lem:convergence_of_sharp}, under the condition that
	$$
	\| h^P\|_{\infty, p(\delta)} < {+} \infty.
	$$
	In addition, for $\YY_i \in \mY_{kpz}$ and $\oh_i$ satisfying Assumption \ref{assu:initial_condition} for $i = 1, 2$, with
	$$
	\| \YY_i\|_{\mY_{kpz}}, \ \ \sup_{n} \| \oh^n_i{-} Y^n_i(0)\|_{\CC^{\beta}_{p(\delta)}} \le M
	$$
we can estimate for some $\kappa = \kappa(M)$ large enough:
\begin{align*}
\| h_1^P{-} h_2^P\|_{\LL^{\bp, \alpha{+}1}_{e(\kappa)}} &+ \|h'_1{-} h'_2 \|_{\LL^{\bp, \alpha}_{e(\kappa)}} + \|h^{\sharp}_1 {-} h^{\sharp}_2\|_{\LL^{\beta^\sharp, 2\alpha{+}1}_{e(\kappa)}} \\
& \lesssim_M \| \YY_1{-} \YY_{2}\|_{\mY_{kpz}} + \| \oh_1{-} Y_1(0) {-}(\oh_2{-}Y_2 (0)) \|_{\CC^{\beta}_{p(\delta)}}.
\end{align*}
	Finally, the function $w = \exp (h) $ is the solution to the RHE in the sense of Definition \ref{def:rhe_solutions}.
\end{theorem}

\begin{proof}
	The existence of solution satisfying the required bound follows from Proposition \ref{prop:convergence_of_hn}. Let us prove uniqueness. Hence suppose that $h$ is a solution to the KPZ with $\| h^P\|_{\infty, p(\delta)} < {+} \infty.$ From the previous result we deduce that indeed $w^P = \exp (h^P)$ is paracontrolled. We need to show that it solves the rough heat equation~\eqref{eqn:w^p}: since this equation has a unique paracontrolled solution our result will follow.
	First, we apply the chain rule to see that $$\partial_t w^P = w^P \cdot \partial_t h^P, \ \ \partial_x^2 w^P = w^P \cdot (\partial_x^2 h^P + (\partial_x h^P)^2)$$ with all products classically well-defined. Thus $w^P$ solves:
\begin{align*}
	(\partial_t - \hh \partial_x^2)w^P & = w^P \big[ (X + X^{\BA} + X^{\BB}) \diamond \partial_x h^P \\
	& + \LLL(Y^{\BC} \! \!+ Y^{\BD} \!) + (X X^{\BB} - X \reso X^{\BB}) + X^{\BA}X^{\BB} + \hh (X^{\BB})^2 \big]
\end{align*}
where we have marked the product that needs the paracontrolled structure of $h^P$ to be well defined with the diamond symbol. In view of the fact that $\partial_x w^P = w^P \partial_x h^P$ and by considering smooth approximations we see that:
$$
	w^P(X \diamond \partial_x h^P) = X \diamond \partial_x w^P
$$
and thus $w^P$ solves Equation \eqref{eqn:w^p}. Hence the local Lipschitz dependence on the parameters follows from same property of the exponential and logarithmic map from Lemmata \ref{lem:exp}, \ref{lem:log} as well as of the solution to the RHE (Proposition \ref{prop:existence_rhe}) and the solution to the Sharp equation (Proposition \ref{prop:existence_sharp}).
\end{proof}

\section{Polymer Measure}\label{sect:polymer_measure}
In this section we build the random directed polymer measure associated to white noise and study its link to the solution $h$ of the KPZ equation. From this point onwards we will use arrows to denote time inversion with respect to the time horizon $T$, i.e. we write $\ola{f}$ for the function $\ola{f}(t) = f(T{-}t)$.

\subsection{An Informal Calculation}

Let us consider a formal solution to the SDE

\begin{equation}\label{eqn:SDE_with_Dh}
\begin{aligned}
d\gamma_t = \partial_x \ola{h}(t,\gamma_t)dt + dW_t,  \ \ \gamma_0 = x_0,
\end{aligned}
\end{equation}
where $W$ is a Brownian motion started in $x_0 \in \RR$. There are two issues with this SDE. The first is that $\partial_xh$ is only a distribution. The approaches developed by Delarue and Diel \cite{singular} as well as Cannizzaro and Chouk \cite{cannizzaro} have tackled this aspect successfully. The second issue is that in our setting $\partial_xh$ is of exponential growth, so a priori we would expect that the solution $\gamma$ could explode in finite time. What gives us hope is that the exponential growth of $\partial_x h$ is mostly due to our approach through the Cole-Hopf transform.

We exploit the random directed polymer measure associated to white noise to build a weak solution to the SDE \eqref{eqn:SDE_with_Dh}, avoiding the use of $h,$ and replacing it instead with elements of $\YY$ up to a rest term $Y^R$ (cf. \cite[Section 7]{kpz}). In the remainder of this preamble we present a formal calculation that explains the approach.

Consider the Wiener measure $\mathbb{P}_{x_0}$ started in $x_0$ on the space $C([0,T]; \RR)$. Denote with $\gamma$ the coordinate process on $C([0,T]; \RR)$. We define the measure $\mathbb{Q}_{x_0}$ given by the Radon-Nikodym derivative:
\begin{equation}\label{eqn:naive_kpz_polymer_measure}
\frac{d\mathbb{Q}_{x_0}}{d\mathbb{P}_{x_0}} = \exp\left( \int_0^T \partial_x\olh(s,\gamma_s)d\gamma_s - \hh \int_0^T|\partial_x\olh|^2(s,\gamma_s)ds \right).
\end{equation}
where $h$ is a solution to KPZ for a (spatially) smooth noise $\theta \in \smooth$, with extended data $\YY(\theta)$ and with initial condition $\oh.$ By Girsanov's theorem under this measure the coordinate process is a weak solution to the SDE \eqref{eqn:SDE_with_Dh}. We can formally apply the Itô formula to $\ola{h}$:
\begin{align*}
	\olh(t,\gamma_t) - \olh(0,x_0) = -\int_0^t \! (\hh|\partial_x\olh|^2 + (\ola{\theta} {-} c^{\BA}))(s,\gamma_s)ds + \int_0^t \! \partial_x\olh(s, \gamma_s) d\gamma_s.
\end{align*}

Now define the random directed polymer measure $\mathbb{\widetilde{Q}}_{x_0}$ by:
\begin{equation}\label{eqn:naive_polymer_measure}
\frac{d\mathbb{\widetilde{Q}}_{x_0}}{d \mathbb{P}_{x_0}} = C\exp\left( \int_0^T (\ola{\theta}(s, \gamma_s) {-} c^{\BA}) ds \right),
\end{equation}

where $C>0$ is a normalizing constant. Note that unless $\oh = 0$ the polymer measure \textit{is not} exactly the measure that solves the SDE (\ref{eqn:SDE_with_Dh}). Indeed:
$$
    \frac{d\mathbb{Q}_{x_0}}{d\mathbb{P}_{x_0}} = \exp\left( \olh(T,\gamma_T) - \olh(0,x_0) + \int_0^T (\ola{\theta}(s, \gamma_s) {-} c^{\BA}) ds \right).
$$
Since $(\partial_t + \hh \Delta )(\ola{Y} + \ola{Y}^{\BA})= -(\ola{\theta} {-} c^{\BA}) - \hh |\partial_x\ola{Y}|^2$ and writing $U = \ola{Y} + \ola{Y}^{\BA}$ we can apply Itô's formula to $U(t,\gamma_t)$ and write:
\begin{align*}
	\int_0^t (\ola{\theta}(s, \gamma_s) {-} c^{\BA}) ds  = \int_0^t \partial_xU(s,\gamma_s)d\gamma_s - \int_0^t\hh |\partial_xU|^2(s, \gamma_s)ds + \text{ Rest},
\end{align*}
where the rest is
\begin{align*}
	 \text{ Rest} = \int_0^t( \hh |\ola{X}^{\BA}|^2 {+} \ola{X}\ola{X}^{\BA} )(s, \gamma_s) ds + U \big\vert^{(0,x_0)}_{(t, \gamma_t)}.
\end{align*}
Finally, defining $Y^R$ as the solution to:
 \begin{equation}\label{eqn:Y^R}
 (\partial_t - \hh \Delta)Y^R = \hh|X^{\BA}|^2 + X X^{\BA} + (X {+} X^{\BA}) \partial_x Y^R, \ \ Y^R(0) = 0,
 \end{equation}
the rest term can be rewritten as
\begin{align*}
	\text{ Rest } = \int_0^t\ola{X}^R(s, \gamma_s) (d\gamma_s - \partial_xUds ) + \bq{U + \ola{Y}^R}^{(0,x_0)}_{(t, \gamma_t)}. 
\end{align*}

In this way we can do the change of measure in two steps: 
\begin{enumerate}
	\item We build the singular measure $\mathbb{P}^{U}_{x_0}$ with:
$$ \frac{d\mathbb{P}^{U}_{x_0}}{d\mathbb{P}_{x_0}} = \exp \left(\int_0^t \partial_xU(s,\gamma_s)d\gamma_s - \hh \int_0^t|\partial_xU|^2(s, \gamma_s)ds \right)$$
	\item Under $\mathbb{P}^U$ the process $W_t = \gamma_t - \int_0^t\partial_xU(s,\gamma_s)ds$ is a B. M. started in $x_0$, so that we can obtain the measures $\mathbb{Q}_{x_0}$ and $\widetilde{\QQ}_{x_0}$ as absolutely continuous perturbations by setting
\begin{equation}\label{eqn:definition_radon_nikodym_derivative_with_h}
 \frac{d\mathbb{Q}_{x_0}}{d\mathbb{P}^U_{x_0}} = \exp\left(\int_0^T\ola{X}^R(s, \gamma_s) dW_s + \bq{\olh - U - \ola{Y}^R}^{(T, \gamma_T)}_{(0,x_0)}\right)
\end{equation}
\begin{equation}\label{eqn:definition_radon_nikodym_polymer_measure_NO_h}
	\frac{d \widetilde{\mathbb{Q}}_{x_0}}{d \mathbb{P}^U_{x_0}} = \frac{\exp \left(\int_0^T\ola{X}^R(s, \gamma_s) dW_s + \bq{U + \ola{Y}^R}^{(0,x_0)}_{(T, \gamma_T)}\right) }{\mathbb{E}_{\mathbb{P}^U_{x_0}} \bq{\exp \left( \int_0^T\ola{X}^R(s, \gamma_s) dW_s + \bq{U + \ola{Y}^R}^{(0,x_0)}_{(T, \gamma_T)} \right)} }.
\end{equation}
\end{enumerate}

This approach goes back to~\cite[Section 7]{kpz} for the equation on the torus, but due to the weighted spaces in which we have to work it becomes much more complicated in our setting and actually we cannot directly make sense of (\ref{eqn:definition_radon_nikodym_derivative_with_h}, \ref{eqn:definition_radon_nikodym_polymer_measure_NO_h}). In the next paragraphs we shall show how to rigorously carry out the analysis and how to construct the measures $\mathbb{P}^U_{x_0}, \wQ$ and $ \widetilde{\QQ}_{x_0}$ via the partial Girsanov transform that we just illustrated.

\subsection{A Paracontrolled Approach}\label{subsectn:paracontrolled_SDE}

In order to construct the measure $\PP^U_{x_0}$ we prove the existence of martingale solutions to the associated SDE:
\begin{equation}\label{eqn:SDE_para}
\begin{aligned}
	d\gamma_t = \partial_xU(t, \gamma_t)dt + dW_t, \ \ \gamma_0 = x_0.
\end{aligned}
\end{equation}

The essential tool for solving the martingale problem is to solve the backward Kolmogorv equation
\begin{equation}\label{eqn:kolmogorov_para}
	(\partial_t + \hh \Delta + \partial_xU\partial_x)\varphi_{\tau}  = f , \ \ \varphi_{\tau}(\tau) = \varphi^0, \ \ t \in [0,\tau],
\end{equation}
for $\tau \in [0,T]$ and a sufficiently large class of forcings $f$ and terminal conditions $\varphi^0$. 

\begin{remark}
This approach to SDEs with singular drift was established in the work by Delarue and Diel \cite{singular} who used rough path integrals (inspired by \cite{hai:solving_kpz}) to solve the Kolmogorov equation. In contrast to our setting, the assumptions on the weight in \cite{singular} do not allow linear growth for $Y.$ This is only a technical issue, but overcoming it would result in several lengthy calculations. Thus we prefer to follow Cannizzaro and Chouk \cite{cannizzaro} who formulated the approach of Delarue and Diel in the paracontrolled framework and thereby also extended it to higher dimensions. This suits our setting better and allows the reader to have a complete overview and a better understanding of the techniques at work.
\end{remark}

The first results concern the existence of solutions to the Kolmogorov equation.

\begin{proposition}\label{prop:existence_kolmogorov_eqn}
Fix any $l \in \RR, \YY \in \mY_{kpz}, \tau \in [0, T]$ as well as an initial condition $\varphi^0 \in \CC^{2\alpha + 1}_{e(l)}$ and a forcing $f \in C\CC^{2\alpha -1}_{e(l)}([0, T]).$ In this setting Equation (\ref{eqn:kolmogorov_para}) has a unique paracontrolled solution $\varphi_{\tau}$.
Moreover for any $M > 0$, if we denote by $\varphi_{\tau}^1$ and $\varphi_{\tau}^2$ the respective solutions to the equation for two different external data $\YY_1$ and $\YY_2,$ initial condition $\varphi^0_1$ and $\varphi^0_2$ and forcings $f_1$ and $f_2$ such that 
$$\| \YY_i \|_{\mY_{kpz}}, \ \|\varphi^0_i\|_{\CC^{2\alpha{+}1}_{e(l)}}, \ \| f_i\|_{C\CC^{2\alpha{-}1}_{e(l)}([0, T])} \le M,
$$
we find that for some $\kappa = \kappa(l,T)$ and any $\ve \in (6a/\delta{+}1{-}2\alpha, 3\alpha{-}1):$
\begin{align*}
	\sup_{ \tau \in [0, T]} \|\varphi_{\tau}^1 {-} \varphi_{\tau}^2\|_{\LL^{\alpha{+}1{-}\ve}_{e(\kappa)} } \lesssim_M & \| \varphi^0_1 {-} \varphi^0_2 \|_{\CC^{2\alpha{+}1}_{e(l)}} \\
	& + \| f_1 {-} f_2 \|_{C\CC^{2\alpha{-}1}_{e(l)}([0, T])} + \|\YY_1 {-} \YY_2\|_{\mY_{kpz}}.
\end{align*}
\end{proposition}
\begin{proof}
This result is a consequence of Theorem \ref{thm:solution_abstract}. By time reversal it suffices to solve the equation
$$
(\partial_t -\hh \Delta - \partial_x(X_{T{-}\tau} {+} X^{\BA}_{T{-}\tau})\partial_x)\ola{\varphi}_{\tau}  = \ola{f} , \ \ \ola{\varphi}_{0}(\tau) = \varphi^0, \ \ t \in [0,\tau],
$$
where we write $\ola{g}(t) = g(\tau{-}t)$ and the terms $X_{t{-}\tau}, X^{\BA}_{T{-}\tau}$ belong to the data $\YY_{T{-}\tau}$ as constructed in Proposition \ref{prop:rescale_tranlate_external_data}. We thus know that $\YY_{T{-}\tau} \in \myc$ (see Definition~\ref{def:singular-Y}) for $b = 2a$ (recall that $a$ is the polynomial growth coefficient of our data) and some $\zeta > 1/2{-}\alpha.$ Then we apply Theorem~\ref{thm:solution_abstract} with the coefficients chosen as follows: $F(\YY_{T{-}\tau})(u) = \partial_x u $ and 
\[
R(\YY_{T{-}\tau}, f)(u) = {-}\ola{f} + X^{\BA}_{T{-}\tau} \partial_x u + X_{T{-}\tau} \para \partial_x u,
\]
where the parameter $f$ lives in the space $\mX = C\CC^{2\alpha{-}1}([0,T]).$ An application of the Schauder estimates and the estimates for paraproducts shows that $R$ and $F$ satisfy the requirements of Assumption \ref{assu:parameters_eqn}. Thus we find a solution $\ola{\varphi}_{\tau} \in  \LL^{\alpha{+}1{-}\ve}_{e(\kappa)}$ for any $\ve \in (6a/\delta{+}2\zeta, 3\alpha{-}1)$, where both the parameter $\kappa$ and the estimates on the norm of the solution can be chosen uniformly over $\tau$ as a consequence of the estimates from Theorem \ref{thm:solution_abstract}, since (cf. Proposition \ref{prop:rescale_tranlate_external_data}):
$$
\sup_{t \in [0, T]} \| \YY_{T{-}\tau} \|_{\myc} < {+}\infty.
$$
\end{proof}

In the following we will show how to use the existence of solutions to the PDE to find unique martingale solutions to the martingale problem \eqref{eqn:SDE_para}. For technical reasons in order to construct the polymer measure we will need a slightly more complicated version of the space $\mY_{kpz},$ in which we add as a requirement the convergence of an asymmetric product. This convergence is guaranteed in the case of space-time white noise by the result of Theorem \ref{thm:renormalisation}.

\begin{definition}
For $\YY^n, \YY \in \myk$ we say that $\YY^n \to \YY$ in $\myp$ if the convergence holds in $\myk$ and in addition the following asymmetric resonant product converges: 
$$
\partial_x Y^{\BZ, n} \reso \partial_x Y \to \partial_x Y^{\BZ} \reso \partial_x Y \text{ in } C\CC^{2\alpha{-}1}_{p(a)}.
$$
Similarly, we say that $\YY$ belongs to $\myp$ if there exists a sequence $\YY^n \in \myi$ such that $\YY^n \to \YY$ in $\myp.$
\end{definition}

\begin{proposition}\label{prop:exist_SDE_para}
	For any $x_0 \in \mathbb{R}$ and $\YY \in \myp$ there exists a unique probability measure $\mathbb{P}^U_{x_0}$  on $C([0,T]; \RR)$ such that under this measure the coordinate process $(\gamma_t)_{t \in [0, T]}$ satisfies:
	\begin{enumerate}
		\item $\mathbb{P}^U_{x_0}(\gamma_0 = x_0) = 1,$
		\item for any $\tau \le T, l \in \RR$ and  for any $f$ in $C L^{\infty}_{e(l)}([0, \tau])$ and $\varphi^{0}$ in $\CC^{2\alpha{+}1}_{e(l)}$ the paracontrolled solution $\varphi(t,x)$ to  Equation (\ref{eqn:kolmogorov_para}) on the interval $[0,\tau]$ satisfies that $$\varphi(t, \gamma_t) - \int_0^tf(s, \gamma_s)ds, \qquad t \in [0, \tau]$$ is a square integrable martingale under $\mathbb{P}^U_{x_0},$ with respect to the canonical filtration.
		\item $\gamma$ is a.s. $\zeta-$H\"older continuous for any $\zeta < 1/2.$
	\end{enumerate}
\end{proposition}

We split the proof of this proposition in two lemmata, which are interesting in themselves. In the first one we derive a priori estimates for the exponential moments of a solution to the SDE.
\begin{lemma}\label{lem:expntl_delta_momts}
    Consider any $M, l\ge0$. Fix any $x_0 \in \RR$ and $\YY \in \myi$ which has norm $\| \YY \|_{\myk} \le M$. There exists a constant $C= C(M, l,T)>0$ such that the  strong solution $\gamma_t$ to the SDE (\ref{eqn:SDE_para}) with the uniformly Lipschitz drift $\partial_x U$ satisfies:
    $$
    \sup_{0 \le t \le T} \mathbb{E}_{x_0}\bq{ e^{l |\gamma_t|^{\delta}}} \le Ce^{C|x_0|^{\delta}}.
    $$
\end{lemma}
\begin{proof}
Fix a terminal condition $\varphi^0$ such that $\varphi^{0}(x) = e^{l |x|^{\delta}}$ for $|x| > 1$ and is smooth and bounded for $|x|\le 1.$ For any $\tau \in [0,T]$ and $\YY \in \myk$ with $\| \YY \|_{\myk} \le M$ it follows from Proposition \ref{prop:existence_kolmogorov_eqn} that there exists a constant $C=C(M, l, T)$ such that the solution $\varphi_{\tau}$  to Equation (\ref{eqn:kolmogorov_para}) with forcing $f = 0$ and terminal condition $\varphi^0$ satisfies $ |\varphi_{\tau}(t,x)| \le Ce^{C|x|^{\delta}}.$ From the It\^o formula and the fact that we chose a bounded noise we know that $\varphi_{\tau}(t, \gamma_t)$ is a true martingale. Hence:
$$
\sup_{\tau \in [0,T]} \mathbb{E}_{x_0}\bq{e^{l|\gamma_{\tau}|^{\delta}} } \simeq \sup_{\tau \in [0, T]} \mathbb{E}_{x_0} \bq{\varphi_{\tau}(\tau, \gamma_{\tau})} = \sup_{\tau \in [0,T]} \varphi_{\tau}(0, \gamma_0) \le Ce^{C|x_0|^{\delta}}.
$$
\end{proof}

With this result at hand we can prove tightness and convergence for the laws of the solutions associated to the SDE.

\begin{lemma}\label{lem:weak_convergence_partial_SDE}
	Consider a sequence $\YY^n$ in $\myi$ such that  $\YY^n \to \YY$ in $\myp$. Let $W$ be a Brownian motion and $\gamma^n$ the strong solutions to the SDE (\ref{eqn:SDE_para}) driven by the smooth and bounded drift $\partial_xU^n.$
Then there exists a measure $\PP^U_{x_0}$ on $C([0, T]; \RR^2)$ such that, denoting with $(\gamma, \ol{W})$ the canonical process on this space:
$$
	(\gamma^{n}, W) \Rightarrow (\gamma, \ol{W})
$$
in the sense of weak convergence of measures. The process $\gamma$ is the unique martingale solution to the martingale problem of Proposition \ref{prop:exist_SDE_para} and it is $\zeta-$H\"older continuous for any $\zeta < 1/2$.
In addition, for any process $H^{n}$ adapted to the filtration $(\FF^n_t) = \sigma(\gamma^n_s | s \le t)$ if the three processes $(\gamma^{n}, H^{n}, W)$ jointly converge to $(\gamma, H, \ol{W}),$ then also:
$$
(\gamma^{n}, H^{n}, W, \smallint \! H^{n} d W) \Rightarrow (\gamma, H, \ol{W}, \smallint \! H d\ol{W})
$$
in the sense of weak convergence on $C([0, T]; \RR^4).$

\end{lemma}

\begin{proof}
We articulate the proof of this lemma as follows. First, we show tightness for the law of $\gamma^n$. Then we show that the weak limits of $\gamma^n$ are the unique martingale solution to the given martingale problem. From this we deduce the first weak convergence result. Finally, we address the issue with the stochastic integral. Let us denote with $(\mathcal{F}^n_t)_{t \in [0, T]}$ the filtration generated by $\gamma^n$ and note that $W$ is an $\mathcal{F}^n-$Brownian motion. Also, let $(\mathcal{F}_t)_{t \in [0, T]}$ be the canonical filtration, i.e. the one generated by the coordinate process $(\gamma_t)_{t \in [0, T]}$.

\textit{Step 1.} For any $\tau$ and $h$  such that $0\le \tau \le \tau {+} h \le T$ we consider the solution $\varphi^n_{\tau{+}h}$ to the Kolmogorov PDE (\ref{eqn:kolmogorov_para}) with $f = 0$ and $\varphi^n_{\tau+h}(\tau{+}h, x) = x$ driven by the external data $\YY^n$. Then we see that
\begin{align*}
	\gamma^n_{\tau{+}h} {-} \gamma^n_{\tau} = & \varphi^{n}_{\tau+h}(\tau {+}h, \gamma^n_{\tau{+}h}) {-} \varphi^{n}_{\tau+h}(\tau, \gamma^n_{\tau}) + \varphi_{\tau+h}^{n}(\tau, \gamma^n_{\tau}) {-} \varphi_{\tau+h}^{n}(\tau {+}h, \gamma^n_{\tau}) \\
	= & \int_{\tau}^{\tau{+}h }\partial_x\varphi^{n}_{\tau+h}(s, \gamma^n_s) dW_s + \varphi^{n}_{\tau+h}(\tau, \gamma^n_{\tau}) {-} \varphi^{n}_{\tau+h}(\tau {+}h, \gamma^n_{\tau}).
\end{align*}
Hence, we get from the Burkholder-Davis-Gundy inequality for any $p\ge 1$
\begin{align*}
	\EE_{x_0} \big[ |\gamma^n_{\tau+h}-\gamma^n_{\tau}|^{2p} \big] \lesssim & \EE \bigg[ \bigg\vert \int_\tau^{\tau{+}h} |\partial_x\varphi^n_{\tau{+}h}(s, \gamma^n_{s})|^2 ds \bigg\vert^{p} \bigg] + \EE_{x_0} \big[|\varphi^n_{\tau{+}h}(\tau, \gamma^n_\tau) {-}\varphi^n_{\tau{+}h}(\tau{+}h, \gamma^n_\tau))|^{2p}\big].
\end{align*}
Now we can apply our uniform bound
$$
\sup_n \sup_{0\le \tau \le \tau+h \le T}\|\varphi^{n}_{\tau+h}\|_{\LL^{2\alpha{+}1{-}\ve}_{e(\kappa)}} < {+}\infty
$$ for some $\kappa$ large enough together with the result of the previous lemma to see that the second term is uniformly bounded by $h^{p(2\alpha{+}1{-}\ve)}$, whereas the first term can be estimated via:
\[
	\EE \bigg[ \bigg\vert \int_\tau^{\tau{+}h} |\partial_x\varphi^n_{\tau{+}h}(s, \gamma^n_{s})|^2 ds \bigg\vert^{p} \bigg] \lesssim  h^{p{-}1} \EE \bigg[  \int_\tau^{\tau{+}h} |\partial_x\varphi^n_{\tau{+}h}(s, \gamma^n_{s})|^{2p} ds \bigg] \lesssim h^{p}.
\]
Hence, we eventually find:
$$
	\EE_{x_0} \big[ |\gamma^n_{\tau+h}-\gamma^n_{\tau}|^{2p} \big] \lesssim h^{p} {+} h^{p(2\alpha{+}1{-}\ve)} \lesssim h^{p},
$$
and thus an application of Kolmogorov's continuity criterion gives for all $\zeta < 1/2$ and all $p > 1/(2\zeta)$:
$$
\sup_{n} \EE_{x_0} \bigg[ \ \sup_{0 \le s \le t} \frac{|\gamma^n_t {-} \gamma^n_s|^{2p}}{|t-s|^{2p (\zeta {-}1/(2p))}} \ \bigg] < {+} \infty.
$$
This is enough to ensure the tightness of the sequence $B^n$ and also of the couple $(B^n, W)$, as well as the H\"older continuity of the limiting process.

\textit{Step 2.} Next we prove that all weak limit points of $\gamma^n$ are solutions to the martingale problem of Proposition \ref{prop:exist_SDE_para}. Uniqueness of such solutions can then be proven as in \cite[Proof of Theorem 8]{singular}. Fix $f$ and $\varphi^0$ as required. Then the solutions $\varphi^n$ to Equation \eqref{eqn:kolmogorov_para} with smooth noise $\theta^n$ converge to the solution $\varphi$ in $\LL^{\alpha{+}1{-}\ve}_{e(\kappa)}$ for $\ve \in (6a/\delta{+}1{-}2\alpha, 3\alpha{-}1)$. For fixed $n$ we have that
$$
M^n_t = \varphi^n(t, \gamma^n_t) - \int_0^t f(s, \gamma^n_s) ds = \int_0^t \partial_x \varphi^n(s, \gamma^n_s) dW_s
$$
is a martingale with respect to the filtration $(\mathcal{F}^n_t)_{t \in [0, T]}$ (since $W$ is an $\mathcal{F}^n-$Brownian motion) and satisfies, due to our exponential bound from Lemma~\ref{lem:expntl_delta_momts}:
$$
\sup_{n} \EE \bigg[ \sup_{0 \le t \le T_0} |M^n_t|^2 \bigg] <{+} \infty.
$$
Hence the sequence is uniformly integrable and together with Lemma \ref{lem:expntl_delta_momts} and the Skorohod embedding theorem this guarantees that up to taking a subsequence
$$
M^n_t \to M_t = \varphi(t, \gamma_t) - \int_0^t f(s, \gamma_s) ds = \int_0^t \partial_x \varphi(s, \gamma_s) d\ol{W}_s
$$
almost surely and in $L^1$ and it follows the latter is a martingale with respect to the canonical filtration $(\mathcal{F}_{t})_{t \in [0, T]}$.

\textit{Step 3.} By tightness we can show that along a subsequence $(\gamma^{n_k}, W) \Rightarrow (\gamma, \ol{W}).$ If we can prove that the joint law of $(\gamma, \ol{W})$ is uniquely defined, the joint weak convergence follows. If the drift were a smooth function we could observe that $\ol{W}_t = \gamma_t - \int_0^t \partial_x U(s, \gamma_s) \ ds$, with the right hand-side being a measurable function of the process $\gamma$. In the rough setting one has to be more careful, since it is not clear how the last term is defined. We will show that for a sequence of measurable functions $F_n$ it is possible to write $(\gamma_t, \ol{W}_t) = \lim_n (\gamma_t, F_n(\gamma))$. Indeed for $n \in \NN$ one can solve the equation
$$
(\partial_t {+}\hh \Delta {+} (X {+} X^{\BA}) \partial_x) \varphi^n = X^n{+} X^{\BA, n}, \ \ \varphi^n(T) = 0.
$$
We can subtract the term of lowest regularity to find $\varphi^n = Y^{\BZ, n} + \psi^n$ with $\psi^n$ solving:
$$
(\partial_t {+}\hh \Delta {+} (X {+} X^{\BA}) \partial_x) \psi^n = X^{\BA, n}  {+} (X^{\BA} {+}X) X^{\BZ, n}, \ \ \psi^n(T) = 0.
$$
Since $\YY^n \to \YY$ in $\myp$ the resonant product $X \reso X^{\BZ, n}$ converges to $X \reso X^{\BZ}$ in $C\CC^{2\alpha{-}1}_{p(a)}$. Thus along the same lines of Theorem \ref{thm:solution_abstract} it is possible to find a paracontrolled solution to the previous equation with the structure:
$$
\psi^n = (\partial_x \psi^{n} + X^{\BZ, n}) \ppara Y^{\BZ} + \psi^{\sharp, n}
$$
with $\psi^n \in \LL^{\alpha{+}1{-}\ve}_{e(l+t)}$, $\psi^{\sharp, n} \in \LL^{2\alpha{+}1{-}\ve}_{e(l+t)}$ for any $l > 0$ and $\ve \in (6a/\delta{+}1{-}2\alpha, 3\alpha{-}1)$. Moreover, because the resonant product converges to the right limit, as $n \to \infty$ the above solutions $\psi^n$ converge in $\LL^{\alpha{+}1{-}\ve}_{e(l{+}t)}$ to the solution $\psi$ of
$$
(\partial_t {+}\hh \Delta {+} (X {+} X^{\BA}) \partial_x) \psi = X^{\BA}  {+} (X^{\BA} {+}X) X^{\BZ}, \ \ \psi(0) = 0.
$$
Similarly the solutions $\varphi^{n,n}$ to the equation 
$$
(\partial_t {+}\hh \Delta {+} (X^n {+} X^{\BA, n}) \partial_x) \varphi^{n,n} = X^n{+} X^{\BA, n}, \ \ \varphi^{n,n}(0) = 0.
$$
exhibits the same structure $\varphi^{n,n} = Y^{\BZ, n} + \psi^{n,n}$ and by the Lipschitz dependence on the parameters in Theorem \ref{thm:solution_abstract} we have that $\psi^{n,n}$ converges to $\psi$ in $\LL^{\alpha{+}1{-}\ve}_{e(l+t)}$. Along the subsequence $(n_k)_k$ along which $(\gamma^{n_k}, W)$ converges to $(\gamma, \ol{W})$ we can now apply \cite[Theorem 2.2]{Kurtz} to see that
\begin{align*}
\widetilde{A}_t^{n_k} &= \int_0^{t}( \ola{X}^{n_k} {+} \ola{X}^{\BA, {n_k}})(s, \gamma^{n_k}_s) ds \\ 
& = \ola{\varphi}^{{n_k},{n_k}}(t, \gamma^{n_k}_t) {-} \ola{\varphi}^{{n_k},{n_k}}(0, x_0) {-} \int_0^t \partial_x \varphi^{{n_k},{n_k}}(\gamma^{n_k}_s) dW_s
\end{align*}
converges to
$$
A_t : = \ola{\varphi}(t, \gamma_t) {-} \ola{\varphi}(0, x_0) - \int_0^t \partial_x \varphi(\gamma_s) d\ol{W}_s.
$$
By using $\varphi^n$ instead of $\varphi^{{n_k},{n_k}}$ we find that:
$$
A_t = \lim_n A^n_t := \lim_n \int_0^t (\ola{X}^n{+}\ola{X}^{\BA, n})(s, \gamma_s) ds
$$
Now since 
$$(\gamma, \ol{W}) = (\gamma, \gamma{-}x_0{-} A) = \lim_n (\gamma, \gamma {-}x_0 {-}A^n) = \lim_n (\gamma, F_n(\gamma))$$
we find the required uniqueness of the law.

\textit{Step 4.} Finally, as we already noted, the convergence of the stochastic integrals along a subsequence is a consequence of \cite[Theorem 2.2]{Kurtz}
\end{proof}

\subsection{Polymer Measure}
Our next aim is to construct the ``full'' polymer measure $\mathbb{Q}_{x_0}$. In principle we would like to apply the fomulas~\eqref{eqn:definition_radon_nikodym_derivative_with_h} or~\eqref{eqn:definition_radon_nikodym_polymer_measure_NO_h} for the explicit Radon-Nikodym derivative with respect to $\bb{P}_{x_0}^U$. Unfortunately we do not have sufficient control of the growth of $X^R$, so we need to argue differently. Even for \lqt smooth\rqt noises $\theta$ in $\LLL C^{\alpha /2} (\RR;C^{\infty}_b(\RR))$ Equation \eqref{eqn:naive_polymer_measure} does not make sense, since we lack smoothness in time. Thus we follow the calculations at the beginning of this section and define the continuum polymer measure in the following way.
\begin{definition}\label{def:smooth_polymer_measure}
For $\YY$ in $\myi$ and consider the solution $e^{h}$ to the RHE for an initial condition $e^{\oh} = e^{Y(0)}\cdot w_0,$ with $w_0\ge 0, w_0 \in \CC^{\beta}_{e(l)}$, for some $l \in \RR, \beta >0$, we define:
$$
\frac{d \wQ(\YY)}{d \PP_{x_0}} = \exp\bigg( \int_0^T (\hh | \ola{X}|^2 {-} c^{\BA})(s, \gamma_s) ds + \big[ \ola{Y}{-} \olh \big]^{(0, x_0)}_{(T, \gamma_T)} \bigg) \cdot \frac{d \PP^{\ola{X}}_{x_0}}{d \PP_{x_0}},
$$
where $\PP^{\ola{X}}_{x_0}$ is the measure under which the coordinate process $\gamma$ solves $d \gamma = \ola{X}(\gamma) {+} dW$ for a Brownian motion $W$ started in $x_0$ and where $\PP_{x_0}$ is the Wiener measure.
\end{definition}

Although the above notation suggests that we use the solution $h$ to the KPZ equation associated to a smooth noise, this is really just notation (which we chose because it fits the Gibbs measure formalism). Actually the construction of the continuum random polymer measure does not depend on the existence of the solution $h$: we only need to understand the solution $w = e^h$ to the RHE with some strictly positive initial condition.

The construction of the polymer measure we review in the following is already known from a work by Alberts, Khanin and Quastel \cite{AlbertsQuastelPolymer}. We implement their strategy in our pathwise setting and we link it with the approach of Delarue and Diel \cite{singular}. The idea is to show convergence of the finite dimensional distributions by controlling the density of the transition function with respect to the Lebesgue measure. Eventually a tightness result guarantees that the limiting measure is supported in the space of continuous functions.

\begin{lemma}\label{lem:convergence_polymer_measure}
Fix $x_0 \in \RR, \YY \in \myk$ and $\oh$ such that $w_0 = e^{\oh{-}Y(0)} \in \CC^{\beta}_{e(l)}$ for some $l \in \RR, \beta>0$. Then we can define a measure $\wQ = \wQ(\YY)$ on $\RR^{[0,T]}$ through its finite dimensional distributions as follows. For any $n \in \NN$ and $0= t_0 \le t_1 < \ldots < t_n \le T$ we define:
\begin{align}\label{eqn:finite_dimensional_distributions}
\wQ(\gamma_{t_1} = dx_1, \ldots, \gamma_{t_n} = dx_n)= \prod_{i = 1}^{n}  \ol{Z}(t_{i{-}1}, x_{i-1}; t_i, x_i) dx_i,
\end{align}
where $y \mapsto \ol{Z}(s,x;t,y)$ is the probability density 
$$\ol{Z}(s,x;t,y)= e^{{-}\olh(s, x )}Z(s,x;t,y)e^{\olh(t,y)},$$
where $\ola{f}(r) = f(T{-}r)$ and $Z(s, x;t,y )$ solves:
\begin{align}\label{eqn:density_pde_ploymer_measure}
(\partial_s + \hh \Delta_x )Z + \ola{\xi} \diamond Z & = 0, \qquad s \in [0,t], \qquad Z(t,x; t, y) = \delta(x{-}y).
\end{align}

If $\YY \in \myi$ this measure coincides with the one from Definition \ref{def:smooth_polymer_measure}.
Moreover, if $\YY^m \to \YY$ in $\mY_{kpz}$ and $w_0^m \to w_0$ in $\CC^{\beta}_{e(l)}$ the measures converge weakly, i.e.:
$$
\wQ(\YY^m) \Rightarrow \wQ(\YY).
$$
Finally, for any $M, l \ge 0$ there exists a constant $C=C(M,l,T)>0$ such that:
\begin{equation}\label{eqn:exponential_moments_polymer_measure}
\sup_{0 \le t\le T} \EE_{\wQ} \Big[ e^{l|\gamma_t|^{\delta}}\Big] \le C e^{{-}\olh(0, x_0)}e^{C|x_0|^{\delta}}
\end{equation}
uniformly over $\YY, w_0$ such that $\| \YY \|_{\mY_{kpz}},  \| w_0\|_{\CC^{2\alpha{+}1}_{e(l)}} \le M$.
\end{lemma}

\begin{remark}
The existence of solutions to equation \eqref{eqn:density_pde_ploymer_measure} requires the study of the equation with initial condition in H\"older-Besov spaces with integrability index $p \in [1, \infty)$ in order to tame the singularity of the Dirac delta. Although it would be straightforward to generalize Theorem \ref{thm:solution_abstract} in this direction  (cf. \cite{kpz}), we refrain from doing so in this work and fix $p = \infty$. Indeed, we are mainly interested in the KPZ Equation, where the nonlinearity does not allow initial conditions with integrability $p\neq \infty$. The existence of solutions to \eqref{eqn:density_pde_ploymer_measure} is one of the main results in \cite{pam} and we do not need to repeat its proof. Hairer and Labbé show the existence of functions $Z(t,x;s,y)$ which lie locally in space and time in  $\LL^{1/2-}_p$. The most important feature - and the only one we will use - of these functions is that they are fundamental solutions to the RHE. For simplicity, let us indicate with $*$ the contraction along a variable, so that $f(*_1, \ldots, *_n)h(*_1, \ldots, *_n) = \int dy_1 \cdots dy_n f(y_1, \ldots, y_n) h(y_1, \ldots, y_n).$ With this notation, the function  $\varphi(s,x) = Z(s,x; t,*)g(*)$ solves 
\begin{equation}\label{eqn:RHE_backwards_for_polymer}
(\partial_s + \hh \Delta_x )\varphi + \ola{\xi} \diamond \varphi = 0, \qquad  s \in [0,t], \qquad \varphi(t,x) = g(x).
\end{equation}
This equation can be solved also in our setting, as long as $g \in \CC^{\beta}_{e(l)}$ for $\beta>2\alpha{-}1$, following Corollary \ref{cor:backwards_RHE} and the preceding discussion. 
\end{remark}

\begin{proof}[Proof of Lemma \ref{lem:convergence_polymer_measure}]
The property highlighted in the previous remark suffices to show that $\ol{Z}(s,x;t,y)$ is a probability distribution for any $s,t,x,y$ and satisfies the Chapman-Kolmogorv equations. Indeed for the first property it suffices to observe that since $e^{\olh}$ solves Equation \eqref{eqn:RHE_backwards_for_polymer} with terminal condition $e^{\oh}$ and thus
$$e^{\olh(s,x)} = Z(s,x; t,*)e^{\olh(t, *)}.$$
The Chapman-Kolmogorov equations are satisfied, since for $0 < s < r < t \le T$:
$$
\ol{Z}(s,x; r,*) \ol{Z}(r, *; t, y) = e^{-\olh(s,x)}Z(s, x; r,*)Z(r,*; t, y)e^{\olh(t,y)} = \ol{Z}(s,x; t,y).
$$
In particular, the $\ol{Z}$ form a consistent family of probability distributions and hence the measure $\wQ$ is well-defined.
Let us consider the case of spatially smooth noise. We show that the one-dimensional distributions described by $\ol{Z}$ coincide with those of the measure $\wQ$ as from Definition \ref{def:smooth_polymer_measure}: the general case follows similarly. Fix a Lipschitz function $g$ and $0<t\le T$, then $\EE_{\wQ}[g(\gamma_t)] = u(0, x_0)$ with
\begin{align*}
  u(s,x) = e^{(\ola{Y}{-}\olh)(s,x)} \EE_{\wQ} \bigg[ \exp \bigg( \int_s^t (\hh|\ola{X}|^2{-}c^{\BA})(r, \gamma_r) \bigg) \cdot g(\gamma_t)e^{{-}(\ola{Y}{-}\olh)(t, \gamma_t)} \big\vert X_s = x \bigg].
\end{align*}
Then the Feynman-Kac formula guarantees that $u = e^{(\ola{Y}{-}\olh)} w$ with $w$ solving
\begin{align*}
(\partial_s {+}\Delta_x {+} \ola{X}\partial_x)w = -\big( \hh |\ola{X}|^2{-}c^{\BA} \big), \qquad s \in [0, t], \qquad w(t, x) = g(x)e^{{-}(\ola{Y}{-}\olh)(t,x)}
\end{align*}
so that a simple calculation shows that $e^{\ola{Y}(s,x)}w(s,x) = Z(s,x; t,*)e^{\olh(t,*)}g(*)$. Hence the claim follows.
Let us pass to proving the convergence of the finite-dimensional distributions. We check the convergence of $$\EE_{\wQ^m}[g_1(\gamma_{t_1})\cdots g_n(\gamma_{t_n})]$$ for all globally bounded and Lipschitz functions $g$, with $\wQ^m = \wQ(\YY^m)$ and $0 < t_1< \ldots <t_n \le T$. Then
\begin{align*}
\EE_{\wQ^m}[ & g_1(\gamma_{t_1})   \cdots  g_n(\gamma_{t_n})] = \\
 & = \ol{Z}^m(0, x_0; t_{1}, *_1)g_1(*_1) \prod_{i = 1}^{n{-}1} \ol{Z}^m(t_{i}, *_{i}; t_{i{+}1}, *_{i{+}1})g_i(*_{i{+}1})  \\
& = e^{{-}\olh(0, x_0)} Z^m(0, x_0; t_{1}, *_1)g_1(*_1) \bigg[ \prod_{i = 1}^{n{-}1} Z^m(t_{i}, *_{i}; t_{i{+}1}, *_{i{+}1})g_i(*_{i{+}1}) \bigg] e^{\olh^m(t_n, *_n)}.
\end{align*}
The last term in the product $Z^m(s, x; t_{n}, *_{n})e^{\olh^m(t_n, *_n)}g_{n}(*_{n})$ solves in $(s,x)$ Equation \eqref{eqn:RHE_backwards_for_polymer} on $[0, t_n]$ with terminal condition 
$$e^{\olh^m(t_n, x)}g_{n}(x) = e^{(\ola{Y}^m{+}\ola{Y}^{\BA, m} {+} \ola{Y}^{\BB,m})(t_{n}, x)} g_{n} (x) \ola{w}^{P,m}(t_n, x)
$$
where we used the structure of a solution $e^h$ to the RHE. Since $g_n \ola{w}^{P,m}$ converges to $g_n \ola{w}^P$ in $\CC^{\alpha{+}1{-}\ve}_{e(\kappa)}$ for any $\ve \in (6a/\delta{+}1{-}2\alpha, 3\alpha{-}1)$ (see Proposition \ref{prop:existence_rhe}), Corollary \ref{cor:backwards_RHE} guarantees that $Z^m(s, x; t_{n}, *_{n})e^{\olh^m(t_n, *_n)}g_{n}(*_{n})$ converges to $Z(s, x; t_{n}, *_{n})e^{\olh(t_n, *_n)}g_{n}(*_{n})$, the latter solving Equation \eqref{eqn:RHE_backwards_for_polymer} on $[0, t_n]$ with terminal condition $e^{\olh(t_n, x)}g_{n}(x).$ Note that the convergence holds in a space with an explosion at time $s = t_n.$ Since we are interested in the value of the solution only at the time $t_{n-1} < t_n$ this does not play a role. In particular
\begin{equation*}
R^m_n(x) = e^{-(\ola{Y}^m{+}\ola{Y}^{\BA, m} {+} \ola{Y}^{\BB,m})(t_{n-1}, x)} Z^m(t_{n-1}, x; t_{n}, *_{n})e^{\olh(t_n, *_n)}g_{n}(*_{n})
\end{equation*}
converges to some $R_n$ in $\CC^{\alpha{+}1{-}\ve}_{e(\kappa)}$, up to taking a possibly larger $\kappa$.
Now we pass to the second-to-last term. Again
\begin{align*}
Z^m(s, x; t_{n{-}1}, *_{n{-}1})g_{n{-}1}(*_{n{-}1})e^{(\ola{Y}^m{+}\ola{Y}^{\BA, m} {+} \ola{Y}^{\BB,m})(t_{n-1}, *_{n-1})} R^m_n(*_{n-1})
\end{align*}
solves in $(s,x)$ Equation \eqref{eqn:RHE_backwards_for_polymer}. Since $g_{n{-}1}(x)R^m_n(x)$ converges $\CC^{\alpha{+}1-\ve}_{e(\kappa)}$ this solution converges once more via Corollary \ref{cor:backwards_RHE} to 
$$Z(s, x; t_{n{-}1}, *_{n{-}1})g_{n{-}1}(*_{n{-}1})e^{(\ola{Y}{+}\ola{Y}^{\BA} {+} \ola{Y}^{\BB})(t_{n-1}, *_{n-1})} R_n(*_{n-1}).$$
Iterating this procedure $n$ times we deduce the convergence of the finite-dimensional distributions.

For the exponential bound let us choose a smooth function $\varphi$ such that $\varphi(x) =  \exp(l|x|^{\delta})$ for $|x| > 1$ and $\varphi$ is smooth and bounded for $|x|\le 1$. Then:
\begin{align*}
\EE_{\wQ} \Big[ e^{l|\gamma_t|^{\delta}}\Big] \simeq \EE_{\wQ}\Big[\varphi(\gamma_t) \Big] = \ol{Z}(0,x_0; t,*_1) \varphi(*_1) \le C(M,l,T) e^{{-}\olh(0, x_0)} e^{C(M, l,T) |x_0|^{\delta}},
\end{align*}
where in the last step as before we used the bounds from Corollary \ref{cor:backwards_RHE} and the fact that $\beta>0$.
\end{proof}

Now we show that the polymer measure is supported on the space of continuous functions.

\begin{lemma}\label{lem:tightness_polymer}
There exists a value $a_{crit} > 0$ such that for $\YY^n \in \myi, \YY \in \myk$ such that $\YY^n \to \YY$ in $\myk$ for some $a \le a_{crit}$ ($a$ being the growth parameter in $\myk$ from Definition \ref{def:ykpz_space}) and for a sequence of initial conditions
$$
 e^{\oh^n {-}Y^n(0)} = w_0^n \to w_0 = e^{\oh{-}Y(0)} \text{ in } \CC^{2\alpha{+}1}_{e(l)},
$$
the sequence of measures $\wQ^n$ is tight in $C([0,T])$. Moreover any accumulation point has paths which are almost surely $\zeta-$H\"older continuous, for any $\zeta < 1/2.$
\end{lemma}

\begin{proof}
We want to use the Kolmogorov criterion. For this reason we fix $q > 1$ and $s \le t$ with $|t{-}s| \le 1$ and will prove that
$$
\EE_{\wQ(\YY)} \big[ |\gamma_t {-}\gamma_s|^{2q} \big] \lesssim_M |t{-}s|^q
$$
uniformly over all $\YY \in \myi$ and $w_0 = e^{\oh{-}Y(0)} \in \CC^{\beta}_{e(l)}$ such that $\|\YY\|_{\myk}, \|w_0\|_{\CC^{\beta}_{e(l)}} \le M$. Given such an estimate the tightness of the sequence as well as the H\"older continuity of the limit points follow by an application of Kolomogorov's continuity criterion. To find this estimate fix $\YY \in \myk$ and $w_0  = e^{\oh{-}Y(0)} \in \CC^{\beta}_{e(l)}$ and let us rewrite the expectation through the densities:
\begin{align*}
\ol{Z}(0, x_0; s, *_1)|*_1 {-}*_2|^{2q}\ol{Z}(s,*_1; t, *_2)e^{\olh(t, *_2)}
\end{align*}
where $e^h$ is the solution to the RHE with initial condition $e^{\oh}$ and external data $\YY$.
Let us proceed one integration variable at a time: we consider $x_1$ fixed and estimate
\begin{equation}\label{eqn:technical:moment_estimate}
|x_1 {-} *_2|^{2q} Z(s,x_1; t, *_2)e^{\olh(t, *_2)}.
\end{equation}
First, we shift $x_1$ to zero. For this purpose we introduce the notation $g^{x_1}(x) = g(x{+}x_1)$ and for $\YY = Y(\theta, Y(0), c^{\BA}, c^{\BD})$ in $\myi$ we write $\YY^{x_1} = \YY(\theta^{x_1}, Y(0)^{x_1}, c^{\BA}, c^{\BD})$, where the latter is obtained by shifting all the extended data by $x_1$. Hence we find the identity:
$$
Z(\YY)(s,x{+}x_1; t, y) = Z(\YY^{x_1})(s,x; t,y{-}x_1),
$$
then we rewrite the term under consideration as 
$$\varphi^{x_1}(r, x) = Z(\YY^{x_1})(r,x; t, *_2)|*_2|^{2q}e^{\olh^{x_1}(t, *_2)},$$
and we aim at estimating $\varphi^{x_1}(s, 0)$ uniformly over $x_1.$ Note that $\varphi^{x_1}$ solves Equation \eqref{eqn:RHE_backwards_for_polymer}:
$$
(\partial_r {+} \hh \Delta_x) \varphi^{x_1} = - \ola{\theta}^{x_1} \diamond \varphi^{x_1}, \qquad r \in [0, t], \qquad \varphi^{x_1}(t,x) = |x|^{2q}e^{\olh^{x_1}(t, x)}
$$
in the sense of Corollary \ref{cor:backwards_RHE}. Now we exploit the parabolic scaling of the equation. Let us define $\lambda = \sqrt{|t{-}s|}$ and write $g^{x_1}_{\lambda}(r,x) = \varphi^{x_1}(s{+} \lambda^2r, \lambda x)$ for $r \in [0, 1]$, so that
$$
(\partial_r {+} \hh \Delta_x) g^{x_1}_{\lambda} = - \lambda^2 \ola{\theta}^{x_1}_{s, \lambda} \diamond g^{x_1}_{\lambda}, \qquad r \in [0, 1], \qquad g_{\lambda}^{x_1}(1,x) = \lambda^{2q}|x|^{2q}e^{\olh^{x_1}(t, \lambda x)}
$$
where formally the term $\ola{\theta}^{x_1}_{s, \lambda}$ should be understood as $\ola{\theta}^{x_1}_{s, \lambda}(r,x) = \theta(T{-}s{-}\lambda^2r, x_1{+}\lambda x).$ Rigorously this means that $g^{x_1}_{\lambda}$ solves Equation \eqref{eqn:RHE_backwards_for_polymer} on $[0, 1]$ in the sense of Corollary \ref{cor:backwards_RHE} associated to the extended data $\YY^{x_1}_{T{-}t, \lambda}$(see Definition \ref{def:singular-Y} and Proposition \ref{prop:rescale_tranlate_external_data}). Since $e^h$ solve the RHE it is of the form 
$$
e^{\olh^{x_1}(t, x)} = e^{(Y^{x_1} {+}Y^{\BA, x_1}{+}Y^{\BB, x_1})(T{-}t, x)}w^{P, x_1}(T{-}t,x)
$$
 with $w^P \in \LL^{\bp, \alpha{+}1{-}\ve}_{e(\kappa)}$ for some $\kappa$ sufficiently large, so that with Lemma \ref{lem:interpolation_schauder}:
 $$\sup_{t}\|w^{P, x_1}(t, \cdot)\|_{L^{\infty}_{e(\kappa)}} \le C e^{C|x_1|^\delta}$$
 for some $C(M)> 0.$ Hence by comparison $g^{x_1}_{\lambda} \le \psi^{x_1}_{\lambda},$ the latter being the solution to
\begin{align*}
(\partial_r {+} \hh \Delta_x) \psi^{x_1}_{\lambda} &= - \lambda^2 \ola{\theta}^{x_1}_{s, \lambda} \diamond \psi^{x_1}_{\lambda}, \qquad r \in [0, 1], \\
 \psi_{\lambda}^{x_1}(1,x) &= \lambda^{2q}|x|^{2q}Ce^{\kappa|x|^{\delta} + C |x_1|^{\delta}} e^{(Y^{x_1} {+}Y^{\BA, x_1}{+}Y^{\BB, x_1})(T{-}t, \lambda x)}.
\end{align*}
Following the results from Corollary \ref{cor:backwards_RHE} the solution $\psi^{x_1}_{\lambda}$ is of the form:
\[
\psi^{x_1}_{\lambda}(0,0) = \lambda^{2q} e^{C|x_1|^{\delta}}e^{(Y^{x_1} {+}Y^{\BA, x_1}{+}Y^{\BB, x_1})(T{-}s, 0)}\psi^{P, x_1}_{\lambda}(0,0),
\]
where we can estimate the norm of the last term:
\[
\| \psi^{P, x_1}_{\lambda}\|_{\LL^{\alpha{+}1{-}\ve}_{e(l)}} \lesssim e^{C T \big( 1 {+} \| \YY^{x_1}_{T{-}t, \lambda} \|_{\mY_{kpz}^{\varrho, b}} \big)^{q_1}}
\]
for $\varrho \in (1/2{-}\alpha, \alpha]$ and $b = 2a$ and some $C(M)>0.$ Now it follows from the definition of $\YY^{x_1}$, Definition \ref{def:singular-Y} (in particular from the fact that the norm $\| \cdot \|_{\mY_{kpz}^{\varrho, b}}$ does not depend on $Y$, but only on $X$ - otherwise we would obtain linear growth in $x_1$) and Proposition \ref{prop:rescale_tranlate_external_data} that:
$$
\sup_{x_1} \sup_{t \in [0, T)} \sup_{\lambda \in (0, 1]} \frac{1}{1{+}|x_1|^a}\| \YY^{x_1}_{T{-}t, \lambda} \|_{\mY_{kpz}^{\varrho, b}} \lesssim_M 1.
$$
Now choose $a_{crit}$ so that $a_{crit} q_1 = \delta:$ for $a\le a_{crit}$ we can conclude that up to choosing a larger $C(M)>0$:
\begin{align*}
\EE_{\wQ^n} \big[ |\gamma_t {-} \gamma_s|^{2q} \big] & \lesssim_M \lambda^{2q} \ol{Z}(0, x_0; s, *_1)e^{(Y {+}Y^{\BA}{+}Y^{\BB})(T{-}s, *_1)} e^{C|*_1|^{\delta}} \\
& \lesssim_M |t{-}s|^{q} e^{-\olh(0, x_0)}e^{(Y {+}Y^{\BA}{+}Y^{\BB})(T{-}s, x_0)}e^{\widetilde{C}|x_0|^{\delta}} \lesssim_{M, x_0} |t{-}s|^{q}.
\end{align*}
This concludes the proof.
\end{proof}

We collect the two previous results in the following proposition.

\begin{proposition}\label{prop:existence_polymer_measure}
For any $x_0, l \in \RR$ and $\YY$ which lies in $\mY_{kpz}$ for $a\le a_{crit}$ (see Lemma \ref{lem:tightness_polymer}) and $e^{\oh {-}Y(0)} \in \CC^{\beta}_{e(l)}$ there exists a measure $\wQ$ on $C([0,T]; \RR)$ such that for $\YY^n \in \myi$ converging $\YY^n \to \YY$ in $\myk$ and initial conditions $e^{\oh^n{-}Y^n(0)} \to e^{\oh{-}Y(0)}$ in $\CC^{\beta}_{e(l)}$, the polymer measures converge weakly:
$$
\wQ(\YY^n) \Rightarrow \wQ(\YY) \text{  in  } C([0,T]).
$$
In addition, under $\wQ$ the sample paths are a.s. $\zeta-$H\"older continuous for any $\zeta < 1/2.$
\end{proposition}

Now we show that the measure $\wQ$ we just built has a density with respect to the singular Girsanov transform $\PP^U_{x_0}$, i.e. while we are not able to construct the measure using Equation \eqref{eqn:definition_radon_nikodym_derivative_with_h} from our formal discussion above, the equation holds a posteriori. This equation is useful because it describes the singular part of the polymer measure in terms of the solution $Y^R$ to the linear equation  \eqref{eqn:Y^R}, and therefore it is not necessary to understand the solution to the KPZ equation or the RHE in order to study the polymer measure.

Recall that $Y^R$ was defined as the solution to
\begin{equation*}
   (\partial_t - \hh \Delta)Y^R = \hh|X^{\BA}|^2 + X X^{\BA} + (X {+} X^{\BA}) \partial_x Y^R, \ \ Y^R(0) = 0,
\end{equation*}
Indeed we can find a paracontrolled solution $Y^R$ of the form:
$$
Y^R = Y^{\BB} {+} Y^P, \ \ Y^P = Y' \ppara Y^{\scalebox{0.8}{\RS{r}}} {+} Y^{\sharp}, \ \ Y' = X^{\BB} {+} \partial_x Y^P.
$$
The equation for $Y^P$ can be solved with calculations similar to the ones leading to Proposition \ref{prop:existence_rhe}. Eventually we find a solution $Y^P \in \LL^{\alpha{+}1}_{e(\kappa)}$ for $\kappa$ large enough.
\begin{proposition}\label{prop:polymer_measure_decomposition}
For any $\YY$ which lies in $\myp$ for $a < a_{crit}$ (see Lemma \ref{lem:tightness_polymer}), the measure $\wQ$ has a density with respect to the measure $\PP^U_{x_0}$ which is given by:
$$
\frac{d \wQ}{d \PP^U_{x_0}} = \exp\Big( \int_0^T \ola{X}^R(s, \gamma_s) d\ol{W}_s  + \big[ U {+}\ola{Y}^R {-}\olh \big]^{(0, x_0)}_{(T, \gamma_T)} \Big),
$$
where $\ol{W}$ is the Brownian motion started in $x_0$ from Lemma \ref{lem:weak_convergence_partial_SDE}.
\end{proposition}
\begin{proof}
Under the above hypothesis the existence of the measure $\wQ$ is guaranteed by Proposition \ref{prop:existence_polymer_measure}, while the existence of the measure $\PP^U_{x_0}$ follows from Proposition \ref{prop:exist_SDE_para}.
From the computation at the beginning of this section, which lead us to Equation \eqref{eqn:definition_radon_nikodym_derivative_with_h}, the above decomposition holds true at a smooth level. So for any $\YY \in \mY_{kpz}$ let us choose  $\YY^n \in \myi$ such that $\YY^n \to \YY$ in $\myk$. For any $M \in \NN$ let us fix a continuous cut-off functional $\eta_M$ on $C([0,T])$ such that $\eta_M(\gamma) = 1$ if $\|\gamma\|_{\infty} \le M,$ and $\eta_M(\gamma) = 0$ if $\| \gamma \|_{\infty} \ge M+1.$ For any continuous and bounded functional $f$ on $C([0,T])$ we have
\begin{align*}
& \EE_{\wQ^{n}} \big[ f(\gamma) \eta_M(\gamma)\big] \\ 
& = \EE_{\PP^{U^{n}}_{x_0}} \Bigg[ f(\gamma) \eta_{M}(\gamma) e^{\big( \int_0^T \ola{X}^{R, n}(s, \gamma_s) dW_s  + \big[ U^{n}{+}\ola{Y}^{R, n} {-} \olh^{n} \big]^{(0, x_0)}_{(T, \gamma_T)} \big)}\Bigg].
\end{align*}
Now the left-hand side converges by Proposition \ref{prop:existence_polymer_measure}, while the right-hand side converges by Lemma \ref{lem:weak_convergence_partial_SDE}. So we find that:
\begin{align*}
& \EE_{\wQ} \big[ f(\gamma) \eta_M(\gamma)\big] \\
& = \EE_{\PP^{U}_{x_0}} \Bigg[ f(\gamma) \eta_{M}(\gamma) e^{\big( \int_0^T \ola{X}^R(s, \gamma_s) d\ol{W}_s  + \big[ U{+}\ola{Y}^R {-} \olh \big]^{(0, x_0)}_{(T, \gamma_T)} \big)}\Bigg].
\end{align*}
Taking $f \equiv 1$ and sending $M \to \infty$, we obtain from Fatou's lemma that
$$
\exp \bigg( \int_0^T \ola{X}^R(s, \gamma_s) d\ol{W}_s  + \big[ U{+}\ola{Y}^R {-} \olh \big]^{(0, x_0)}_{(T, \gamma_T)} \bigg) \in L^1(\mathbb{P}^U_{x_0}).
$$
Thus we can pass to the limit over $M \to \infty$ and deduce the result by dominated convergence.
\end{proof}

\begin{remark}\label{rem:characterisation_h_feyn_kac}
We have discussed the construction of the measure $\wQ$. This in particular allows us to build the measure $\widetilde{\mathbb{Q}}_{x_0}$ from Equation \eqref{eqn:definition_radon_nikodym_polymer_measure_NO_h} by choosing $\oh = 0.$ In addition, using the fact that the Radon-Nikodym derivative integrates to $1$, we find a representation for the solution $h$ to the KPZ equation as follows:
\begin{align*}
 		[h {-} Y {-} Y^{\BA} {-} Y^R](T,x_0) = \log \mathbb{E}_{\mathbb{P}^{U}_{x_0}}\bq{e^{\int_0^T \ola{X}^R(s, \gamma_s) d\ol{W}_s + [\oh - Y(0)](\gamma_T)}}
\end{align*} 
\end{remark}

\subsection{Variational representation}

Here we show that we can solve the martingale problem \eqref{eqn:SDE_with_Dh} associated to the KPZ equation, and that the solution solves a stochastic control problem. The first step is to define martingale solutions in the paracontrolled setting. One main difference with respect to the definition of \cite{singular, cannizzaro} is that we do not directly solve the PDE associated to the martingale problem. This is because we cannot control the growth of the drift $\partial_xh$ at infinity sufficiently well. Instead, we solve the PDE to remove the singular part $\partial_x U$ of the drift, and then we add the regular part $\nu$ of the drift (which later will be a control) back by hand.

Following \cite[Section 7]{kpz} we will denote by $\mathfrak{pm}$ the set of progressively measurable processes on $[0,T]\times C([0,T]; \RR).$ By this we mean that $\nu \in \mathfrak{pm}$ if for any $0 \le t \le T$ the restriction of $\nu$ to times smaller than $t,$ 
$$ \nu\big\vert_{[0,t]  \times C([0,T]  ; \RR)} \text{ is } \mathcal{B}([0,t]) \otimes \mathcal{F}_t-\text{measurable},$$
where $\mathcal{F} = (\mathcal{F}_t)_{0 \le t \le T}$ is the canonical filtration on $C([0,T]; \RR).$

\begin{definition}\label{def:matingale_problem}
For an element  $\nu \in \mathfrak{pm}$ we say that a measure $\mathbb{P}$ on the filtered measurable space $(C([0,T]; \RR), (\mathcal{F}_t) )$ is a martingale solution to the SDE
\begin{equation}\label{eqn:SDE_for_control_problem}
\begin{aligned}
	d\gamma_t = (\partial_x U + \nu_t )(t, \gamma_t)dt + dW_t, \ \ \gamma_0  = x_0,
\end{aligned}
\end{equation}
if the following two conditions are satisfied for the coordinate process $(\gamma_t)$:
\begin{enumerate}
	\item $\mathbb{P}(\gamma_0 = x_0 ) = 1$.
	\item For any $l \in \RR, \tau \in [0, T]$ and for any $f$ in $C L^{\infty}_{e(l)}([0, \tau]), \ \varphi^0$ in $\CC^{2 \alpha {+} 1}_{e(l)},$ let $\varphi \in \LL^{\alpha {+}1{-}\ve}_{e(\kappa)}$ (see Equation \eqref{eqn:kolmogorov_para} and Proposition \ref{prop:existence_kolmogorov_eqn}) to the equation:
	\begin{equation}\label{eqn:kolm_control_defn}
	\begin{aligned}
		(\partial_t {+} \hh \Delta {+} \partial_x U\partial_x) \varphi  = f, \ \ \varphi(\tau) = \varphi^0.
	\end{aligned}
	\end{equation}
	Then the process $\varphi(t, \gamma_t) - \int_0^t [f(s, \gamma_s) {+} \partial_x\varphi(s,\gamma_s) \nu(s, \gamma)]  ds$ is a square integrable martingale on $[0,\tau]$ with respect to the filtration $(\mathcal{F}_t)$.
\end{enumerate}
\end{definition}

This allows to show that the polymer measure solves the SDE \eqref{eqn:SDE_with_Dh}.

\begin{proposition}\label{prop:martingale_solution}
Consider $\YY \in \myk$ and let $h$ be the solution to the KPZ equation from Theorem \ref{thm:existence_uniqueness_KPZ}. Under the assumptions of Proposition \ref{prop:existence_polymer_measure}, the therein constructed measure $\wQ(\YY)$ is a martingale solution to the SDE
$$
d\gamma_t = \partial_x \olh(s, \gamma_s) ds + dW_s, \ \ \gamma(0) = x_0,
$$
in the sense of the above definition, with control $\nu = \partial_x ( \olh {-}U)$.
\end{proposition}
\begin{proof}
	Consider $\tau, \varphi^0, f$ as in Definition~\ref{def:matingale_problem}. We need to prove that 
	$$
	M_t = \varphi(t,\gamma_t) - \int_0^t [f(s, \gamma_s) {+} \partial_x \varphi(s,\gamma_s) \partial_x(\olh {-} U)(s, \gamma_s)]ds,\qquad t \in [0,\tau],
	$$
	is a martingale with respect to the measure $\mathbb{Q}_{x_0}$. In fact, consider smooth data $\YY^n \in \myi$ such that $\YY^n \to \YY$ in $\myk.$ Then we can find solutions $\varphi^n$ to the PDE (\ref{eqn:kolm_control_defn}) with $\partial_x U$ replaced by $\partial_x U^n$. Since Proposition \ref{prop:existence_kolmogorov_eqn} guarantees that these solutions satisfy
	\[
	\varphi^n \to \varphi \text{ in } \LL^{\alpha {+} 1{-}\ve}_{e(\kappa)}
	\]
for a suitably chosen $\kappa \ge 0$ and $\ve \in (6a/\delta{+}1{-}2\alpha, 3\alpha{-}1)$, and since the uniform sub-exponential bound \eqref{eqn:exponential_moments_polymer_measure} holds true, the process
	\[
	M^n_t = \varphi^n(t,\gamma_t) - \int_0^t f(s, \gamma_s) {+} \partial_x\varphi^n(s,\gamma_s) \partial_x(\olh^n {-} U^n)(s, \gamma_s)ds, \qquad t \in [0,\tau],
	\]
	is a $\mathbb{Q}^n_{x_0}-$martingale such that for a suitable $C \ge 0$
	\[
	    \sup_n \EE_{\wQ^n} \Big[ \sup_{0\le t \le \tau} \big\vert M_t^n \big\vert^2 \Big] \lesssim \sup_n \sup_{0 \le t \le \tau} \EE_{\wQ^n} \Big[ e^{C|\gamma_t|^{\delta}} \Big] < {+}\infty,
	\]
	where $\QQ^n_{x_0}$  is the polymer measure associated to $\YY^n$, as in Proposition \ref{prop:existence_polymer_measure}. The same proposition guarantees that $\wQ^n \Rightarrow \wQ$. Hence the martingale property is preserved in the limit.
	
\end{proof}

We conclude this section on the polymer measure with a variational characterisation of the solution to the KPZ equation.

\begin{theorem}\label{thm:variational_rep}
Consider an extended data $\YY$ and an initial condition $\oh$ which satisfies Assumption \ref{assu:initial_condition}. Let $h$ be the paracontrolled solution to the KPZ equation from Theorem \ref{thm:existence_uniqueness_KPZ}. The following representation holds true:
\begin{equation*}
\begin{aligned}
	[h - &Y - Y^{\BA \! } - Y^{R}] (T,x_0) =  \\
	& = \! \! \sup_{\substack{\nu \in \mathfrak{pm} \\ \gamma \in \mathfrak{M}(\nu,x_0)}} \! \! \! \EE \!  \left[ \overline{h}(\gamma_T) - Y(0, \gamma_T) + \hh \int_0^T \Big(|\ola{X}^{R}|^2 - |\nu {-} \ola{X}^{R} |^2(s, \gamma_s) \Big) ds \right],
\end{aligned}
\end{equation*}
where the optimal control $\nu$ is 
\[
    \nu(s, \gamma) = \partial_x(\olh {-} U)(s,\gamma_s).
\]
\end{theorem}

\begin{proof}
We follow step by step the original proof of \cite[Theorem 7.3]{kpz}. First, let us define
\[
    h^R = h {-} Y {-} Y^{\BA} {-} Y^R = h^P - Y^{P},
\]
which is paracontrolled in the sense that $h^R = \partial_x h^R \ppara Y^{\scalebox{0.8}{\RS{r}}} + h^{R, \sharp}$  with $h^R \in \LL^{\alpha{+}1}_{e(\kappa)}$ and $h^{R, \sharp} \in \LL^{2\alpha+1}_{e(\kappa)}$ for an appropriate $\kappa \ge 0.$ In addition $h^R$ is a paracontrolled solution to the equation:
\begin{align*}
	\LLL h^R & =  \hh |X^R|^2 + (X + X^{\BA} + X^R)\partial_x h^R + \hh(\partial_x h^R)^2,  \\
	h^R(T) & =   \oh - Y(0),
\end{align*}
which by reversing time we can translate into
\begin{align*}
	(\partial_t {+} \hh \Delta {+} \partial_xU \partial_x)\ola{h}^R = \ {-}\hh |\ola{X}^R|^2 {-} \ola{X}^R \partial_x\ola{h}^R {-} \hh (\partial_x\ola{h}^R)^2, \ \  \ola{h}^R(T) = \oh {-} Y(0).
\end{align*}
This means that if we take a martingale solution $\gamma$ to the problem
\[
	d\gamma_t = (\partial_x U + \nu)dt + dW_t,
\]
we can use $h^R$ as a test function, according to Definition \ref{def:matingale_problem}. From this point onwards we can follow exactly the proof of \cite{kpz} to get to the conclusion that for any $\nu \in \mathfrak{pm}$ and $\gamma \in \mathfrak{M}(\nu, x_0)$ 
\begin{align*}
	\olh^R(0,x_0) & =  \EE \bq{ \oh(\gamma_T) - Y(0,\gamma_T) + \hh \int_0^T \Big( |\ola{X}^R|^2 - |\nu {-}  \ola{X}^R|^2(s, \gamma_s) \Big)ds } \\
	& \quad + \EE \bq{\int_0^T \frac{1}{2} |\tilde{\nu}(s, \gamma)|^{2}ds } \\
	& \quad \ge \sup_{\substack{\nu \in \mathfrak{pm} \\ \gamma \in \mathfrak{M}(\nu,x_0)}} \! \! \! \EE \bq{ \oh(\gamma_T) - Y(0,\gamma_T) + \hh \int_0^T \Big( |\ola{X}^R|^2 - |\nu {-}  \ola{X}^R|^2(s, \gamma_s) \Big)ds }
\end{align*}
with $\tilde{\nu} = \nu {-} \ola{X}^R {-} \partial_x \ola{h}^R = \nu {-} \partial_x(\ola{h} {-} \ola{Y} {-} \ola{Y}^{\BA})$, where in the last line we took the supremum on both sides in the line above and then forgot the term with $\tilde{\nu}$. For fixed $\nu$ equality 
\[
\olh^R(0,x_0) =  \EE \bq{ \oh(\gamma_T) - Y(0,\gamma_T) + \hh \int_0^T \Big( |\ola{X}^R|^2 - |\nu {-}  \ola{X}^R|^2(s, \gamma_s) \Big)ds }
\]
holds only if $\tilde{\nu} = 0.$ Thus the supremum is achieved in the polymer measure and equals $\olh^R(0,x_0)$.
\end{proof}

\section{Linear Paracontrolled Equations in Weighted Spaces}\label{sect:abstract_solutions}

\subsection{A Solution Theorem}

We consider an abstract paracontrolled equation of the form:
\begin{equation}\label{eqn:abstract}
	\LLL u = R(\YY, \nu)(u) +  [F(\YY)(u)]\para X +  X \reso \partial_x u , \ \ u(0) = u_0,
\end{equation}
for some functionals $R$ and $F$ which we will specify later, and where $\YY \in \mc{Y}_{kpz}$ and $\nu$ is simply an additional parameter living a Banach space $\mathcal{X},$ which we add to treat certain applications. At an intuitive level $R$ represents a smooth rest term, $\para \ $ is the irregular part of a product and $\reso$ is the ill-posed part of a product, the latter term being the one which requires a paracontrolled structure from the solution.
 
Actually it will be necessary to consider slightly more general $\YY$, allowing for an additional singularity: see Definition~\ref{def:singular-Y} for the definition of $\mY_{kpz}^{\zeta, b}$.

Now we introduce the Banach space of paracontrolled distributions that will contain the solution to Equation (\ref{eqn:abstract}). Consider $u_0 \in \CC^{\beta}_{e(l)}$ for some $l \in \RR$ and $\beta \in (2\alpha{-}1, 2\alpha{+}1]$ (recall the regularity parameter $\alpha$ from Table \ref{table:kpz} and the preceding discussion) as well as a parameter $\ve > 0$ and 
\begin{equation}\label{eqn:def_beta_hat_prime_solution_theorem}
\hb = \frac{2\alpha{+}1{-}\beta}{2}, \qquad \beta^\prime = \frac{\alpha{+}1{-}\beta}{2}\vee 0.
\end{equation}
Forthermore, fix a time horizon $T_h \ge 0$ and $\YY \in \mY^{\zeta, b}_{kpz}([0, T_h])$ for some values $\zeta, b \ge 0$ which we will specify later. The parameter $\ve$ represents a small gap between the regularity of the solution we prove and the expected maximal regularity and appears essentially to deal with the global spatial well-posedness. The parameter $\hb$ quantifies the time blow-up at $t = 0$ in the space $\CC^{2\alpha{+}1}.$ The parameter $\bp$ quantifies the blow-up in the space $\CC^{\alpha{+}1}$ (cf. Lemma \ref{lem:interpolation_schauder}).

Then we introduce a subset
\[
\mathcal{D}(\YY) \subset \LL^{\bp, \alpha{+}1{-}\ve}_{e(l+t)}([0, T_h]) \times \LL^{\bp, \alpha {-} \ve}_{e(l+t)p(a)}([0, T_h]) \times \LL^{\hb, 2\alpha{+}1{-}\ve}_{e(l+t)} \ni (u, u', u^{\sharp})([0, T_h]),
\]
which is defined by the property:
\[
	(u, u', u^{\sharp}) \in \mathcal{D}(\YY) \quad \Leftrightarrow \quad u = u' \ppara Y^{\scalebox{0.8}{\RS{r}}} + u^{\sharp}, \qquad u^{\sharp}(0) = u_0.
\]
Moreover, we endow $\mathcal{D}(\YY)$ with the product topology and a product norm. With some abuse of notation we write 
\[
\nnorm{u} = \|(u, u', u^{\sharp}) \|_{\mathcal{D}(\YY)}.
\]
Since we want to compare also paracontrolled distributions which are controlled by different enhanced data we introduce the quantity
\[
\nnorm{u_1: u_2} = \max \{ \|u'_{1} - u'_{2}\|_{\LL^{\bp, \alpha{-}\ve}_{e(l+t)p(a)}} , \ \| u^{\sharp}_1 - u^{\sharp}_2\|_{\LL^{\hb, 2\alpha{+}1{-}\ve}_{e(l+t)}} \}
\]
for $u_i \in \mathcal{D}(\YY_i)$. Since the proof of the solution theorem relies on a contraction argument on a small time interval we also introduce the notation $\mD^0_{S}(\YY)$ for $S \le T_h$ for the space $\mD(\YY)$ where we replaced the time horizon $T_h$ with $S.$ We also use the convention that $\mD^{0}_0(\YY) = \CC^{\beta}_{e(l)}$. This is motivated by the fact that a point in $\mD^0_S$ will later be an initial condition for the solution on $[S, S{+}\tau]$ for some small $\tau$. Then if we fix $0 \le T_{\ell} < T_r \le T_h$ and an initial condition $u_{T_{\ell}} \in \mD^0_{T_{\ell}}(\YY)$, we consider the space $\mD^{T_{\ell}}_{T_r}(\YY, u_{T_{\ell}})$ of all functions $u$ in $\mD^0_{T_r}$ such that 
\[
u\vert_{[0, T_{\ell}]} = u_{T_{\ell}}, \ \ u'\vert_{[0,T_{\ell}]} = u_{T_{\ell}}', \ \ u^{\sharp}\vert_{[0, T_{\ell}]} = u_{T_{\ell}}^{\sharp}.
\]
We endow this space with the product norm on:
\[
\LL^{\bp, \alpha{+}1{-}\ve}_{e(l+t)}([T_{\ell}, T_r]) \times \LL^{\bp, \alpha {-} \ve}_{e(l+t)p(a)} ([T_{\ell}, T_r]) \times \LL^{\hb, 2\alpha{+}1{-}\ve}_{e(l+t)} ([T_{\ell}, T_r]).
\]
If we do not fix any initial condition, we write $\mD^{T_{\ell}}_{T_r}(\YY).$ Furthermore we remind the notation $V_s$ for the integration operator:
\[
V_s(f)(t) = \int_s^t P_{t{-}h}f_h dh,
\]
where $(P_t)$ is the semigroup generated by $\tfrac12 \Delta$. Now we state the assumption on the coefficients of the equation. As a rule-of-thumb they must be Lipschitz dependent on $u$ and locally Lipschitz dependent on $\YY \in \mY_{kpz}^{\zeta, b}$. Recall that $\mc X$ is an arbitrary Banach space, and we write $\|\cdot\|_{\mc X}$ for its norm.

\begin{assumption}[On the parameters of the equation]\label{assu:parameters_eqn}
	Let $l \in \RR$, let $\varepsilon \in [0, 3\alpha{-}1)$, let $a>0$ and $\zeta \ge 0$ be such that $\ve{-}6a/\delta{-}2\zeta >0$, and let $b \le 2a$  and $\beta \in (2\alpha{-}1, 2\alpha{+}1]$ as well as $\hb, \bp$ as in Equation \eqref{eqn:def_beta_hat_prime_solution_theorem} and consider $0 \le T_{\ell} < T_r \le T_h$. Given $M > 0$ we consider $\YY_i$ in $\mY^{\zeta,b}_{kpz}([0, T_h])$ as well as $\nu_i$ in $\mathcal{X}$ and $u_{0}^i$ in $\CC^{\beta}_{e(l)}$ for $i = 1,2$ such that: 
	\[ \|\nu_i\|_{\mathcal{X}}, \ \| \YY_i\|_{\mY^{\zeta, b}_{kpz}}, \ \| u_{0}^i\|_{\CC^{\beta}_{e(l)}}  \le M \]
	and we require that there exists a $p\ge 1$ such that the following holds true:
	\begin{enumerate}
		\item There exists a $\gamma > 0$ such that for any $\YY, \nu$ and $T_\ell, T_r$ satisfying $0 \le T_{\ell}< T_r \le T_h$ we have that: 
		\[
			V_{T_{\ell}} \circ R(\YY, \nu): \LL^{\bp, \alpha{+}1{-}\ve}_{e(l+t)}([T_{\ell}, T_r]) \to \LL^{\bp, 2\alpha {+} 1 -\ve}_{e(l+t)}([T_{\ell}, T_r])
		\]
		is a Lipschitz function which satisfies:
		\begin{align*}
		\| V_{T_{\ell}}&(R(\YY, \ \nu)(u))\|_{\LL^{\bp, 2 \alpha {+} 1{-}\ve}_{e(l+t)}([T_{\ell}, T_r]) } \\
		&\lesssim (1{+}\| \YY \|_{\myc} {+} \| \nu \|_{\mathcal{X}})^p(T_r{-}T_{\ell})^{\gamma}(1{+}\| u\|_{\LL^{\bp, \alpha{+}1{-}\ve}_{e(l+t)}([T_{\ell}, T_r])}), \\
		\| V_{T_{\ell}}&(R(\YY, \ \nu)(u_1)) - V_{T_{\ell}}(R(\YY, \nu)(u_2)) \|_{\LL^{\bp, 2 \alpha {+} 1{-}\ve}_{e(l+t)}([T_{\ell}, T_r]) } \\
		  & \lesssim (1{+}\| \YY \|_{\myc} {+} \| \nu \|_{\mathcal{X}})^p (T_r{-}T_{\ell})^{\gamma}\| u_1 - u_2\|_{\LL^{\bp, \alpha{+}1{-}\ve}_{e(l+t)}([T_{\ell}, T_r])}, \\
		\| V_{T_{\ell}}& (R(\YY_1, \ \nu_1)(u_1)) {-} V_{T_{\ell}}(R(\YY_2, \nu_2)(u_2)) \|_{\LL^{\bp, 2 \alpha {+} 1{-}\ve}_{e(l+t)} ([T_{\ell}, T_r])} \\
		&  \lesssim_{M, \nnorm{ u_i }_{\mD^{T_{\ell}}_{T_r}} } \|\YY_1 {-} \YY_2\|_{\myc} + \|\nu_1 {-}\nu_2\|_{\mathcal{X}} +(T_r{-}T_{\ell})^{\gamma}\| u_1 - u_2\|_{\LL^{\bp, \alpha{+}1{-}\ve}_{e(l+t)}([T_{\ell}, T_r])}.
		\end{align*}
		\item The map $F(\YY): \LL^{\bp, \alpha{+}1{-}\ve}_{e(l+t)}([T_{\ell}, T_r]) \to \LL^{\bp, \alpha{-}\ve}_{e(l+t)p(a)}([T_{\ell}, T_r])$ is Lipschitz continuous (for fixed $\YY$) and satisfies:
		\begin{align*}
		\| F(\YY)(u) & \|_{\LL^{\bp, \alpha{-}\ve}_{e(l+t)p(a)}([T_{\ell}, T_r]) }  \lesssim (1{+}\|\YY \|_{\myc})^p(1+\|u \|_{\LL^{\bp, \alpha{+}1{-}\ve}_{e(l+t)}([T_{\ell}, T_r])}) \\
		\| F(\YY)(u_1) &{-} F(\YY)(u_2) \|_{\LL^{\bp, \alpha{-}\ve}_{e(l+t)p(a)}([T_{\ell}, T_r]) } \\
		& \lesssim (1{+}\|\YY \|_{\myc})^p \|u_1 {-} u_2 \|_{\LL^{\bp, \alpha{+}1{-}\ve}_{e(l+t)}([T_{\ell}, T_r])}, \\
		\| F(\YY_1)(u_1) & {-} F(\YY_2)(u_2) \|_{\LL^{\bp, \alpha{-}\ve}_{e(l+t)p(a)}([T_{\ell}, T_r]) }  \\
		& \lesssim_{M, \nnorm{ u_i }_{\mD^{T_{\ell}}_{T_r}} } \| \YY_1 {-} \YY_2 \|_{\myc} + \|u_1 {-} u_2 \|_{\LL^{\bp, \alpha{+}1{-}\ve}_{e(l+t)}([T_{\ell}, T_r])}.
		\end{align*}
	\end{enumerate}
\end{assumption}

\begin{remark}
These assumptions provide a minimal working environment that is sufficient for our needs, and they could of course be generalized. In every point, the first two inequalities are necessary for the fixed-point argument, and the last one to obtain the locally Lipschitz continuous dependence of the fixed point on the parameters.
\end{remark}

The central idea in the paracontrolled approach is that the ill-posed resonant product is well defined for paracontrolled distributions. This is the content of the next result.

\begin{lemma}[Paracontrolled Product]\label{lem:product_estimate}
Under the previous assumptions, for $u$ in $\mathcal{D}^{0}_{T_r}(\YY)$, and $\YY \in \myc$ the following product estimate holds:
\begin{align*}
	\|X \reso \partial_xu \|_{\MM^{\hb{+}\zeta}\CC^{2\alpha{-}1}_{e(l+t)p(3a)}([T_{\ell}, T_r])} & \lesssim (1 {+} \| \YY \|_{\myc})^2 \nnorm{u}_{\mD^{0}_{T_r}}.
\end{align*}
if we consider two different enhanced data $\YY_i \in \myc$ as well as initial conditions $u^i \in \mathcal{D}^{T_{\ell}}_{T_r}$ we can also bound
\begin{align*}
	\|X_1 \reso \partial_xu_1 & - X_2 \reso \partial_x u_2 \|_{\MM^{\hb{+}\zeta}\CC^{2\alpha{-}1}_{e(l+t)p(3a)}([T_{\ell}, T_r])} \lesssim_{M, \nnorm{ u_i }_{\mD^{0}_{T_r}} } \| \YY_1 - \YY_2 \|_{\myc} + \nnorm{u_1 : u_2}_{\mD^{0}_{T_r}}.
\end{align*}
\end{lemma}

\begin{proof}
Let us prove the first estimate and assume $T_r = T$. We define 
\[\widetilde{u}^{\sharp} = \partial_x u^{\sharp}{+}\partial_x u' \ppara Y^{\scalebox{0.8}{\RS{r}}} \in \LL^{\hb, 2\alpha{-}\ve}_{e(l{+}t)p(2a)},\] then we can compute:
\[
	X \reso  \partial_x u = u' \para (X \reso \partial_x Y^{\scalebox{0.8}{\RS{r}}}) + C(u', \partial_x Y^{\scalebox{0.8}{\RS{r}}}, X)+ X \reso C_2(u', \partial_x Y^{\scalebox{0.8}{\RS{r}}}) + X \reso \widetilde{u}^{\sharp}. 
\]
with $C_2(u', \partial_x Y^{\scalebox{0.8}{\RS{r}}}) = u' \ppara Y^{\scalebox{0.8}{\RS{r}}} {-} u' \para Y^{\scalebox{0.8}{\RS{r}}}$. Now we can estimate one at the time all these terms. For the first one we have:
\[
\| u' \para (X \reso \partial_x Y^{\scalebox{0.8}{\RS{r}}}) \|_{\MM^{\hb{+}\zeta}\CC^{2\alpha{-}1}_{e(l+t)p(3a)}} \lesssim \| u' \|_{\MM^{\hb}\CC^{\alpha{-}\ve}_{e(l+t)p(a)}} \|X \reso \partial_x Y^{\scalebox{0.8}{\RS{r}}} \|_{\MM^{\zeta}\CC^{2\alpha{-}1}_{p(b)}}.
\]
Similarly we can treat the other terms, by applying the commutation results of Lemma \ref{lem:paraproduct_estimates} to the second term and Lemma \ref{lem:parabolic_paraproduct_estimates} to the third term, as well as the resonant product estimate from Lemma~\ref{lem:paraproduct_estimates} for the last term. The second estimate follows similarly.
\end{proof}

In view of the previous lemma, we can rigorously make sense of the equation under consideration.
\begin{definition}\label{def:abstract_solutions}
Under Assumption \ref{assu:parameters_eqn} a function $u \in \mD(\YY)$ is said to be a solution to be a (paracontrolled) solution to Equation \eqref{eqn:abstract} if $u' = F(\YY)(u)$ and if the equation is satisfied with the last product $ X \reso \partial_x u$ defined in the sense of Lemma \ref{lem:product_estimate}.
\end{definition}

Now we can prove the existence and uniqueness of solutions to the equation.

\begin{theorem}\label{thm:solution_abstract}
Make Assumption \ref{assu:parameters_eqn} and let $\YY \in \myc$. Then there exists a unique solution to Equation \eqref{eqn:abstract} in the sense of Definition \ref{def:abstract_solutions}, and there is $q \ge 0$ such that
\[
\nnorm{u} \lesssim e^{C T_h \big( 1 {+} \| \YY \|_{\myc} + \| \nu \|_{\mathcal{X}} \big)^q}(1{+} \|\YY\|_{\myc}+ \| \nu \|_{\mathcal{X}})^{p{+}1}( 1{+} \| u_0 \|_{\CC^{\beta}_{e(l)}} ).
\]
Moreover, the solution depends locally Lipschitz continuously on the parameters of the equation: for two solutions $u_i$, $i=1,2$, associated to $\YY_i$ and parameters $\nu_i$ we find that
\begin{align*}
	\nnorm{u_1 : u_2} \lesssim_{M, T_h} \|u^1_0 - u^2_0 \|_{\CC^{\beta}_{e(l)}} + \| \YY_1 - \YY_2 \|_{\myc} +\|\nu_1 {-} \nu_2 \|_{\mathcal{X}}.
\end{align*}
\end{theorem} 

\begin{proof}
Fix $0 \le T_{\ell} < T_r \le T_h$ and $u_{T_{\ell}} \in \mD^{0}_{T_{\ell}}(\YY)$. Let us define the following map on $\mathcal{D}^{T_{\ell}}_{T_r}(\YY, u_{T_{\ell}})$. For $t \ge T_{\ell}$ we write: 
\begin{align*}
	\II (u)(t) &= P_{t{-}T_{\ell}} u_0(T_{\ell}) + V_{T_{\ell}}\big( R(\YY, \nu)(u) +F(\YY)(u) \para X +   X \reso \partial_x u \big)(t), \\
	\II' (u)(t) & = F(\YY)(u)(t), \\
	\II^{\sharp}(u)(t) & = \II(u)(t) - \II'(u) \ppara Y^{\scalebox{0.8}{\RS{r}}}(t),
\end{align*}
and $\II(u) = u_{T_\ell}$ on $[0, T_{\ell}].$ By induction, we assume that if $T_{\ell} >0$, then $u_{T_{\ell}}$ is a solution to the equation on $[0, T_{\ell}]$: in particular we will use that $u'_{T_{\ell}} = F(u_{T_{\ell}})$.

For the sake of brevity we will write $\nnorm{\cdot}$ for the norm in $\mD^{T_{\ell}}_{T_r}(\YY, u_{T_{\ell}})$ and $\nnorm{u_{T_{\ell}}}$ for the norm of $u_{T_{\ell}}$ in $\mathcal{D}^{0}_{T_{\ell}}(\YY)$. We show that $(\II, \II^\prime, \II^\sharp)$ maps $\mD^{T_{\ell}}_{T_r}(\YY, u_{T_{\ell}})$ into itself, similar arguments then show that $(\II, \II^\prime, \II^\sharp)^2$ is a contraction on the same space (the presence of the square will guarantee that also the derivative term is contractive). We proceed one term at a time.

\textit{Step 1.} Let us start with $\II$, for which Assumption~\ref{assu:parameters_eqn} yields 
\begin{align*}
	\|\II &(u)  \|_{\LL^{\bp, \alpha {+}1{-}\ve}_{e(l+t)}([T_{\ell}, T_r])}  \lesssim \|P_{\cdot{-}T_{\ell}}u_{T_{\ell}}(T_{\ell})\|_{\LL^{\bp, \alpha{+}1{-}\ve}_{e(l{+}t)}([T_{\ell}, T_r])} \\
	& + (1{+} \|\YY \|_{\myc}{+} \|\nu\|_{\mX} )^{p}(T_r{-}T_{\ell})^{\gamma}\big(1{+}\| u\|_{\LL^{\bp, \alpha{+}1{-}\ve}_{e(l+t)}([T_{\ell}, T_r])}\big) \\
	&+ \|V_{T_{\ell}}(F(\YY)(u)\para X) \|_{\LL^{\bp, \alpha {+}1{-}\ve}_{e(l+t)}([T_{\ell}, T_r])}  + \| V_{T_{\ell}}( X \reso \partial_x u_1)\|_{\LL^{\bp, \alpha {+}1{-}\ve}_{e(l+t)}([T_{\ell}, T_r])}.
\end{align*}
Regarding the first term: By the first estimate of the Proposition \ref{prop:schauder_estimate} we can bound it with $\hbar(\bp, T_\ell, \nnorm{u_{T_\ell}}),$ 
with $\hbar$ satisfying the bound
\[\hbar(\gamma, T_{\ell},  \nnorm{u_{T_\ell}} )\lesssim_{T_h}  \| u_0\|_{\CC^{\beta}_{e(l)}} 1_{\{ T_\ell =0\}}+ (T_\ell)^{-\gamma} \nnorm{u_{T_\ell}} 1_{\{T_\ell >0 \}}.\]

Regarding the second term, we use the second estimate of the same proposition to get:
\begin{align*}
	\|V_{T_{\ell}}(&F(\YY) (u) \para X) \|_{\LL^{\bp, \alpha {+}1{-}\ve}_{e(l+t)}([T_{\ell}, T_r])} \lesssim (T_r{-}T_{\ell})^{\gamma_1}\| F(\YY)(u)\para X \|_{\MM^{\bp}\CC^{\alpha {-}1}_{e(l+t)p(2a)} ([T_{\ell}, T_r])} \\
	& \lesssim (1{+}\|\YY\|_{\myc})^{p{+}1}(T_r{-}T_{\ell})^{\gamma_1}\| u \|_{\LL^{\bp, \alpha{+}1{-}\ve}_{e(l{+}t)}([T_{\ell}, T_r])}
\end{align*}
where the term $(T_r{-}T_{\ell})^{\gamma_1}$ is a consequence of the last estimate from Lemma \ref{lem:interpolation_schauder} with $\gamma_1 = (\ve {-} 4a/\delta)/2 > 0$. The last line follows from the third condition on $F$.
A similar estimate holds for the ill-posed product:
\begin{align*}
\| V_{T_{\ell}}( X \reso & \ \partial_x u) \|_{\LL^{\hb, 2\alpha {+}1{-}\ve}_{e(l+t)}([T_{\ell}, T_r])} \lesssim (T_r{-}T_{\ell})^{\gamma_2} \| V_{T_{\ell}}( X \reso \partial_x u) \|_{\LL^{\hb, 2\alpha {+}1{-}6a/\delta {-}2\zeta {-}\lambda}_{e(l+t)}([T_{\ell}, T_r])} \\
& \lesssim (T_r{-}T_{\ell})^{\gamma_2} \| V_{T_{\ell}}( X \reso \partial_x u) \|_{\LL^{\hb{+}\zeta, 2\alpha {+}1{-}6a/\delta }_{e(l+t)}([T_{\ell}, T_r])}\\
& \lesssim (T_r{-}T_{\ell})^{\gamma_2}\|X \reso \partial_x u \|_{\MM^{\hb{+} \zeta}\CC^{2\alpha{-}1}_{e(l+t)p(3a)} ([T_{\ell}, T_r])} \lesssim (T_r{-}T_{\ell})^{\gamma_2} (1 {+} \| \YY \|_{\myc})^2 \nnorm{u}_{\mD^{0}_{T_r}} \\
& \lesssim  (T_r{-}T_{\ell})^{\gamma_2} (1 {+} \| \YY \|_{\myc})^2( \nnorm{u_{T_\ell}} + \nnorm{u}),
\end{align*}
 where in the first step we used the last estimate from Lemma \ref{lem:interpolation_schauder} and we defined $\gamma_2 = (\ve{-}6a/\delta{-}2\zeta{-}\lambda)/2$ and $\lambda \in (0,\ve{-}6a/\delta{-}2\zeta).$ We chose to subtract an additional (arbitrarily small) regularity $\lambda$ in order to apply the second estimate of Lemma \ref{lem:interpolation_schauder} in the second step and thus gain a factor $\zeta$ in the time-explosion.
Hence, we eventually estimate:
\begin{align*}
	\| V_{T_{\ell}}( X \reso  \ \partial_x u) \|_{\LL^{\beta^\prime, \alpha {+}1{-}\ve}_{e(l+t)}([T_{\ell}, T_r])} \lesssim \| V_{T_{\ell}}( X \reso  \ \partial_x u) \|_{\LL^{\hb, 2\alpha {+}1{-}\ve}_{e(l+t)}([T_{\ell}, T_r])}
\end{align*}

From these estimates it follows that $\II$ maps $\mD_{T_r}^{T_\ell}(\YY, u_{T_\ell})$ into $\LL^{\hb, \alpha{+}1{-}\ve}_{e(l+t)}([T_\ell, T_r])$. Moreover similar calculations, based on the Lipschitz assumptions on the coefficients, show that $\II$ is a contraction for fixed initial condition $u(T_\ell) = u_{T_\ell}(T_\ell)$, provided that $(T_r{-}T_{\ell}) \lesssim (1 {+} \| \YY \|_{\myc})^{-(p+1)/(\gamma \wedge \gamma_2)}$.

\textit{Step 2.} We consider the paracontrolled remainder term. Since $u_{T_\ell}$ is a solution to the equation on $[0, T_{\ell}]$ we find that:
\begin{align*}
u^{\sharp} = P_{\cdot{-}T_{\ell}} (u^{\sharp}(T_{\ell})) + V_{T_{\ell}}\big[ R(\YY, \nu)(u)) +   X \reso \partial_x u  + C_4 & (F(\YY)(u), \partial_x X) \\
 & + C_3(F(\YY)(u), Y^{\scalebox{0.8}{\RS{r}}}) \big],
\end{align*}
where $C_4(u', \partial_x X) = u' \ppara \partial_x X {-} u' \para \partial_x X$. Proceeding as before we can estimate:
\begin{align*}
\|u^{\sharp}&\|_{\LL^{\hb, 2\alpha{+}1{-}\ve}_{e(l+t)}([T_\ell, T_r)} \lesssim (1{+}\| \YY\| {+} \|\nu \|_{\mX})^{p+1} \big[ 1 + \hbar(\hb, T_\ell, \nnorm{u_{T_\ell}} ) + (T_\ell{-}T_r)^{\gamma}\| u\|_{\LL^{\bp, \alpha{+}1{-}\ve}_{e(l+t)}([T_{\ell}, T_r])} \\
& + (T_r{-}T_{\ell})^{\gamma_2} \nnorm{u} \big] +  (T_{\ell}{-}T_r)^{\gamma_2}\| V_{T_{\ell}}( C_4(F(\YY)(u), \partial_x X))\|_{\LL^{\bp, 2\alpha{+}1{-}6a/\delta}_{e(l+t)}([T_{\ell}, T_r])} \\
& +  (T_{\ell}{-}T_r)^{\gamma_2}\| V_{T_{\ell}}( C_3(F(\YY)(u), Y^{\scalebox{0.8}{\RS{r}}}))\|_{\LL^{\bp, 2\alpha{+}1{-}6a/\delta}_{e(l+t)}([T_{\ell}, T_r])}
\end{align*}
where we used the same estimate as before for the rest term $R$ and the ill-posed product. Through the bounds from Lemma \ref{lem:parabolic_paraproduct_estimates} for the commutators $C_3, C_4$ we can then estimate the last two terms in the sum via
\begin{align*}
 \|\YY \|_{\myc} \| F(\YY) (u) \|_{\LL^{\bp, \alpha}_{e(l+t)}([T_\ell, T_r])} \lesssim (1{+}\|\YY\|_{\myk})^{p+1} \nnorm{u}_{\mD^0_{T_r}}
\end{align*}
where in the last step we used the estimate on $F$ from Assumption \ref{assu:parameters_eqn}. Eventually we find the following bound:
\begin{align*}
	\|u^{\sharp}\|_{\LL^{\hb, 2\alpha {+} 1 {-}\ve}_{e(l+t)}([T_{\ell}, T_r])} \lesssim (1{+}\| \YY\|_{\myc}{+} \|\nu \|_{\mathcal{X}})^{p{+}1}(1 + \nnorm{u_{T_\ell}} + (T_r {-} T_{\ell})^{\gamma \wedge \gamma_2} \nnorm{u}).
\end{align*}

\textit{Step 3.} Now we can conclude that for some $q$ large enough and 
\[
T^* = (T_r{-}T_{\ell}) \lesssim (1 {+} \|\YY\|_{\myc}{+} \|\nu \|_{\mc X})^{-q}
\] 
the map $(\II, \II^\prime, \II^\sharp)^2$ is a contraction on $\mD^{T_{\ell}}_{T_r}(\YY, u_{T_\ell})$ and thus it has a unique fixed point. Indeed, due to the presence of the square, and our assumptions on $F$, the derivative $\II^\prime$ inherits the contractive property from $\II$.
Since the length $T^*$ of the interval $[T_{\ell}, T_r]$ could be chosen independently of $u_{T_\ell}$, we can iterate this procedure and concatenate the fixed points to get a solution on $[0, T_h].$
Then the exponential bound follows immediately by observing that we need to iterate approximately $T_h(1 {+} \|\YY\|_{\myc}{+} \|\nu \|_{\mc X})^{q}$ times the inequality
\begin{align*}
\nnorm{\II(u)}_{T_{\ell}, T_r} &\le \overline{C} (1 {+} \|\YY\|_{\myc}{+} \|\nu \|_{\mc X})^{p{+}1}(1 {+} \hbar(T_{\ell},  \nnorm{\II(u)}_{0, T_{\ell}}) ) \\
&  \le C(T^*) (1 {+} \|\YY\|_{\myc}{+} \|\nu \|_{\mc X})^{p{+}1}(1 {+} \nnorm{ \II(u)}_{0, T_{\ell}} ).
\end{align*}
where in the last step we used that $T_\ell \ge T^*$ so that we have a good bound on $\hbar.$
The local Lipschitz dependence on the parameters follows along the same lines. 
\end{proof}

In the remainder of the section we will apply this to several concrete linear equations.

\subsection{Rough Heat Equation}

In this section we show how to solve Equation \eqref{eqn:w^p}, which we recall here:
\begin{align*}
\LLL w^P = & \ \big[(X X^{\BB} {-} X \reso X^{\BB}) + \LLL(Y^{\BC} {+} Y^{\BD})  + X^{\BA} X^{\BB} + \hh(X^{\BB})^2\big]w^P \\
    & \ + [X {+} X^{\BA} {+} X^{\BB}]\partial_x w^P, \\
    w^P(0) =& \  w_0.
\end{align*}
This equation can be written in the form of Equation (\ref{eqn:abstract}) with 
\begin{align*}
	R(\YY)(u) = &  \big[X \para X^{\BB} {+} \LLL(Y^{\BC} {+} Y^{\BD}) {+} X^{\BA} X^{\BB} {+} \hh(X^{\BB})^2\big]u \\
	& + X \para \partial_x u + [X^{\BA} {+} X^{\BB}]\partial_x u , \\
	F(\YY)(u) = & X^{\BB} u + \partial_x u.
\end{align*}

Our aim is clearly to apply Theorem  \ref{thm:solution_abstract}: For this reason we have to check the requirements from Assumption \ref{assu:parameters_eqn}. The first step is counting the regularities. Taking away the time derivative (which for technical reasons we treat differently) we find that:
\begin{align*}
\|R(\YY)(u){-}\partial_t(Y^{\BC}{+}Y^{\BD})u \|_{\MM^{\bp} \CC^{2\alpha{-}1}_{e(l+t)p(2a)}([T_\ell, T_r])} \lesssim \|\YY \|_{\myk}^2 \| u\|_{\LL^{\bp, \alpha{+}1{-}\ve}_{e(l+t)}([T_\ell, T_r])}
\end{align*}
so that Proposition \ref{prop:schauder_estimate} applied to this term and Lemma \ref{lem:young_time_product} guarantee:
\[
\|V_{T_\ell}[R(\YY)(u){-}\partial_t(Y^{\BC}{+}Y^{\BD})u]\|_{\LL^{\bp, \nu}_{e(l+t)}([T_\ell, T_r])} \lesssim (1+\|\YY \|_{\myk})^2 \| u\|_{\LL^{\bp, \alpha{+}1{-}\ve}_{e(l+t)}([T_\ell, T_r])}
\]
for any $\nu < 2\alpha{+}1{-}4a/\delta.$ Thus the third and fourth estimates from Lemma \ref{lem:interpolation_schauder} then provide the bound:
\[
\|V_{T_\ell}[R(\YY)(u){-}\partial_t(Y^{\BC}{+}Y^{\BD})u]\|_{\LL^{\bp, 2\alpha{+}1{-}\ve}_{e(l+t)}([T_\ell, T_r])} \lesssim (T_{\ell}{-}T_r)^\gamma (1+\|\YY \|_{\myk})^2 \| u\|_{\LL^{\bp, \alpha{+}1{-}\ve}_{e(l+t)}([T_\ell, T_r])}
\]
for $\gamma = (\ve {-}4a/\delta)/2$. Similar estimates hold for the product $\partial_t(Y^{\BC}{+}Y^{\BD})u$ in view of Lemma \ref{lem:young_time_product}.

Let us pass to $F$. We immediately find:
\begin{align*}
\| F(\YY, u)\|_{\LL^{\bp, \alpha{-}\ve}_{e(l+t)p(a)}(T_\ell, T_r)} \lesssim (1{+}\| \YY \|_{\myc}) \| u \|_{\LL^{\bp, \alpha {+} 1 {-} \ve}_{e(l{+}t)(T_l, T_r)}},
\end{align*}

since by Lemma \ref{lem:dervtv_interp} the derivative $\partial_x u$ is controlled by
\[
	\| \partial_x u \|_{\LL^{\bp, \alpha{-}\ve}_{e(l+t)}(T_\ell, T_r)} \lesssim \| u \|_{\LL^{\bp, \alpha {+}1{-}\ve}_{e(l+t)}(T_\ell, T_r)}.
\]
The bounds for the differences follow from the bilinearity of $R$ and $F$. Hence applying Theorem \ref{thm:solution_abstract} guarantees the following result.

\begin{proposition}\label{prop:existence_rhe}
	For $l, \ve, \zeta, b, a, \beta, \beta^\prime, \hb$ as in the requirements of Assuption \ref{assu:parameters_eqn}, Equation \eqref{eqn:w^p} admits a unique paracontrolled solution with local Lipschitz dependence upon the parameters. That is, for initial conditions $w_0^1, w_0^2$ and extended data $\YY_1, \YY_2$ which satisfy the requirements of Assumption~\ref{assu:parameters_eqn}, there exist respectively two unique solutions $w_1^P, w_2^P$ to the RHE, that satisfy:
	\[
	\nnorm{w^{P}_1 : w^{P}_2} \lesssim_M \|w^1_0 - w^2_0 \|_{\CC^{\beta}_{e(l)}} + \| \YY_1 - \YY_2 \|_{\myc}.
	\]
Moreover we can bound the norm of the solution in terms of the extended data as follows:
\[
\nnorm{w^P} \lesssim e^{C T_h \big( 1 {+} \| \YY \|_{\myc} \big)^q}( 1{+} \| w_0 \|_{\CC^{\beta}_{e(l)}} )
\]
 for some $q \ge 0$ large enough. In particular, if $\zeta = 0$ these are the unique solutions to the RHE in the sense of Definition \ref{def:rhe_solutions}.
\end{proposition}

We can also solve a time-reversed version of the RHE, Equation \eqref{eqn:RHE_backwards_for_polymer}. In particular, we are also interested in uniform estimates over parabolic scaling of the equation. Consider some $\YY \in \myk$ with $\xi = \LLL Y$ and let us write $\ola{f}_{s,\lambda}(t,x) = f(T{-}s{-}\lambda^2 t, \lambda x),$  as well as $f_{s, \lambda}(t,x) = f(s{+}\lambda^2t, \lambda x)$. The aim is to solve the equation:

\[
(\partial_t {+}\hh \Delta)w = - \lambda^2 \ola{\xi}_{s,\lambda} \diamond w, \qquad w(\tau) = \mathfrak{w}_0
\]
for some $\tau \in[0, \lambda^{-2}(T{-}s)]$, where formally $\lambda^2 \ola{\xi}_{s,\lambda} \diamond w = \lambda^2 (\ola{\xi}_{s,\lambda}{-}c^{\BA})w$. If $\mathfrak{w}_0$ is of the form:
\[
\mathfrak{w}_0(\cdot) = e^{(\ola{Y}_{s,\lambda} {+} \ola{Y}^{\BA}_{s,\lambda} {+} \ola{Y}^{\BB}_{s, \lambda})(\tau, \cdot)}w_0(\cdot)
\]
with $w_0 \in \CC^{\beta}_{e(l)}$ for some $\beta \in (2\alpha{-}1, 2\alpha{+}1], l \in \RR$ we consider solutions $w$ of the form:
\[
w = e^{\ola{Y}_{s, \lambda} {+} \ola{Y}^{\BA}_{s, \lambda} {+} \ola{Y}^{\BB}_{s, \lambda}} w^P,
\]
with the time-reversed $w^P_{rev}(t,x) = w^P(\tau{-}t, x)$ solving the equation
\begin{align*}
(\partial_t {-} & \hh \Delta) w^P_{rev} =  -\big[(X_{\mu, \lambda} X_{\mu, \lambda}^{\BB} {-} X_{\mu, \lambda} \reso X_{\mu, \lambda}^{\BB}) + \LLL(Y_{\mu, \lambda}^{\BC} {+} Y_{\mu, \lambda}^{\BD}) \\
& \ + X_{\mu, \lambda}^{\BA} X_{\mu, \lambda}^{\BB} + \hh(X_{\mu, \lambda}^{\BB})^2\big]w^P_{rev} + [X_{\mu, \lambda} {+} X_{\mu, \lambda}^{\BA} {+} X_{\mu, \lambda}^{\BB}]\partial_x w^P_{rev}, \\
   & \ \  w^P_{rev}(0) = \  w_0.
\end{align*}
with $\mu = T{-}s{-}\lambda^2\tau$. This is exactly the same equation as for the paracontrolled term of the RHE, up to translations and scaling. 
Following Definition \ref{def:ykpz_space} and  Proposition \ref{prop:rescale_tranlate_external_data}, translating the enhanced data by a factor $\mu$ and rescaling by a factor $\lambda$ gives rise to a valid element $\YY_{\mu, \lambda}$of $\myc$ for $b = 2a$ and any $\zeta \in (1/2{-}\alpha, \alpha]$. Hence the previous equation admits a paracontrolled solution in the sense of Proposition \ref{prop:existence_rhe}. We collect this information in the following result.
 
\begin{corollary}\label{cor:backwards_RHE}
For any $w_0 \in \CC^{\beta}_{e(l)},$ for $\beta \in (2\alpha{-}1, 2\alpha{+}1], l \in \RR$ and $\lambda \in [0, 1], \ \tau \in [0, \lambda^{-2}(T{-}s)],$ and $\YY \in \myk$ there exists a unique paracontrolled solution $w^P_{rev}$ to the previous equation. In particular the following bound holds for some $\kappa, q \ge 0$ and any $\ve \in (6a/\delta{+}1{-}2\alpha, 3\alpha{-}1)$:
\begin{align*}
\| w^P \|_{\LL^{\bp, \alpha{+}1{-}\ve}_{e(\kappa)}([0, \tau])} \lesssim e^{C\tau \big( 1 {+} \| \YY_{T{-}s{-}\lambda^2\tau, \lambda} \|_{\myc} \big)^q}( 1{+} \| w_0 \|_{\CC^{\beta}_{e(l)}} ).
\end{align*}
Moreover the solution depends continuously on the parameters in the following sense: for two different enhanced data $\YY_i$ and initial conditions $w_0^i$ such that
\[
\| \YY_i \|_{\myk}, \|w_0^i\|_{\CC^{\beta}_{e(l)}} \le M
\]
we can estimate
\begin{align*}
\|w_1^p{-}w_2^p \|_{\LL^{\bp, \alpha{+}1{-}\ve}_{e(l+t)}([0, \tau])} \lesssim_M \|w_0^1{-}w_0^2 \|_{\CC^{\beta}_{e(l)}}+\|(\YY_1)_{T{-}s{-}\lambda^2\tau, \lambda}{-}(\YY_2)_{T{-}s{-}\lambda^2\tau, \lambda}\|_{\myc} 
\end{align*}
\end{corollary}

\subsection{Sharp Equation}
Now we consider the ``Sharp" equation, that is Equation (\ref{eqn:h^sharp}):
\begin{equation}
	\begin{aligned}
		\LLL u  = Z(\YY, h^P, h') + X \reso \partial_x u, \ \ u(0) = u_0 \in \CC_{e(l)}^{\beta}.
	\end{aligned}
\end{equation}
We do not need to look for paracontrolled solutions for this equation: it falls in the range of Equation (\ref{eqn:abstract}) with $F = 0$, and therefore we expect that the solution has a trivial paracontrolled structure. For the set of parameters 
\[
\nu = (h^P, h') \in \mathcal{X} = \LL^{\bp, \alpha {+}1}_{e(l)} \times \LL^{\bp, \alpha}_{e(l)},
\]
(recall the definition of $\hb, \bp$ from Equation \eqref{eqn:def_beta_hat_prime_solution_theorem}) we define 
\[
	R(\YY, h^P, h')(u) = Z(\YY, h^P, h').
\]
We only need to check that $R$ satisfies the properties of Assumption \ref{assu:parameters_eqn}. Since $R$ is constant in $u$ this reduces to the first and third property. Due to the non-linearity in $h^P$ we work under the additional assumption that $\beta>0$. We see that:
\begin{align*}
V_{T_\ell}(Z(\YY, h^P, h')) = Y^{\BC}{+}Y^{\BD} {-} P_{\cdot{-}T_\ell}[(Y^{\BC}{+}Y^{\BD})(T_\ell)] {+} V_{T_\ell}(Q(\YY, h^P, h'))
\end{align*}
with $Q$ defined as
\begin{align*}
	Q(\YY, h^P, h') = & X \para X^{\BB} {+} X^{\BA}X^{\BB} {+} \hh(X^{\BB})^2  {+} \hh (\partial_x h^P)^2 +(X^{\BA} {+} X^{\BB})\partial_x h^P {+} X \para \partial_x h^P \\ 
	& +  X \reso \partial_x(h' \ppara Y^{\scalebox{0.8}{\RS{r}}}) +  \big[ h' \para \LLL(Y^{\scalebox{0.8}{\RS{r}}}){-} \LLL(h' \ppara Y^{\scalebox{0.8}{\RS{r}}})\big].
\end{align*}
Now define \begin{equation}\label{eqn:def_beta_sharp}\beta^\sharp = \bp \vee (1{-}\beta)\end{equation} and note that due to our assumptions on $\beta$ we have that $\beta^\sharp \in (0,1)$. In view of the regularity assumptions on the enhanced data (see Table \ref{table:kpz}) and the paraproduct estimates from Lemmata \ref{lem:paraproduct_estimates} and \ref{lem:parabolic_paraproduct_estimates} we see that:
\[
\| Q(\YY, h^P, h') \|_{\MM^{\beta^\sharp} \CC^{2\alpha{-}1}_{e(2l)}} \lesssim (1 {+} \|\YY\|_{\myk}{+} \| h^P\|_{\LL^{\bp, \alpha{+}1}_{e(l)}}+ \| h'\|_{\LL^{\bp, \alpha}_{e(l)}})^4
\]
so that an application of the Schauder estimates 	of Proposition \ref{prop:schauder_estimate} guarantees that:
\[
\|V_{T_\ell}(Z(\YY, h^P, h')) \|_{\LL^{\beta^\sharp, 2\alpha{+}1}_{e(2l)}([T_\ell, T_r])} \lesssim (1 {+} \|\YY\|_{\myk}{+} \| h^P\|_{\LL^{\bp, \alpha{+}1}_{e(l)}}+ \| h'\|_{\LL^{\bp, \alpha}_{e(l)}})^4.
\]
The local Lipschitz dependence on the parameters then follows similarly by multi-linearity.
Thus we can apply Theorem~\ref{thm:solution_abstract} and obtain the result below.

\begin{proposition}\label{prop:existence_sharp}
	For any $l \in \RR$ $u_0 \in \CC^{\beta}_{e(l)}$, for $\beta \in (0, 2\alpha{+}1]$, and $\YY \in \myk, (h^P, h') \in \mX $ there exists a $\kappa> l$ such that Equation \eqref{eqn:h^sharp} has a unique solution $h^{\sharp}$ in $\LL^{\beta^\sharp, 2\alpha{+}1}_{e(\kappa)}$, for $\beta^\sharp$ as in Equation \eqref{eqn:def_beta_sharp}. For two initial conditions $u_0^i$ and two extended data $\YY_i$ and parameters $h^P_i, h'_i$ where $i = 1,2$, which satisfy the requirements of Assumption \ref{assu:parameters_eqn}, there exist respectively unique solutions $h_1^{\sharp}, h_2^{\sharp}$ to the equation and they satisfy
	\[
	\| h_1^{\sharp} {-}  h_2^{\sharp}\|_{\LL^{\beta^\sharp, 2\alpha{+}1}_{e(\kappa)}} \lesssim_M \|u^1_0 - u^2_0 \|_{\CC^{\beta}_{e(l)}} + \| \YY_1 - \YY_2 \|_{\myk} + \|h'_1 {-} h'_2 \|_{\LL^{\bp, \alpha}_{e(l)}} + \| h^P_1 {-} h^P_2\|_{\LL^{\bp, \alpha{+}1}_{e(l)}}.
	\]
\end{proposition}
\begin{proof}
	Theorem \ref{thm:solution_abstract} yields that the paracontrolled solution to Equation (\ref{eqn:h^sharp}), which according to the theorem has only regularity $\alpha{+}1{-}\ve$, has a vanishing derivative since $F = 0$. Moreover $\YY \in \myc$ for $b = a$ and $\zeta = 0$. Thus applying one last time the Schauder estimates to the solution $h^{\sharp}$ give us the bounds in $\LL^{\beta^\sharp, 2\alpha{+}1}_{e(\kappa)}.$
\end{proof}

\appendix
\addtocontents{toc}{\protect\setcounter{tocdepth}{-1}}

\section{Exponential and Logarithm on Weighted H\"older Spaces}

Here we discuss the regularity of the exponential and logarithmic maps on weighted H\"older spaces.

\begin{lemma}\label{lem:exp}
Consider any $\alpha \in (0, 2) \setminus \{1\}$ and $R, \hat{l} \ge 0.$ Then there exists an $l = l(R) \ge 0$ such that the exponential function $\exp$ maps
\[
	\exp : \bigg\{ f \in \CC^{\alpha}_{e(\hat{l})}(\RR^d) \text{ s.t. } \norm{f}_{\infty, p(\delta)} \le R \bigg\} \longrightarrow \CC^{\alpha}_{e(l)}(\RR^d).
\]
Moreover the exponential map is locally Lipschitz continuous, i.e. for $f,g$ in the set above such that 
\[
    \| f\|_{\alpha, e(\hat{l}) },  \| g\|_{\alpha, e(\hat{l})}  \le M
\]
we can estimate:
\[
    \| \exp(f) - \exp(g) \|_{\alpha, e(l)} \lesssim_{M,R} \| f - g \|_{\alpha, e(\hat{l})}.
\]
\end{lemma}

\begin{proof}
	Due to our choice of $\alpha$ and the classical characterization of Besov spaces (see Corollary \ref{cor:classical_besov_characterisation}), we need to find a bound for the uniform norm of $\exp(f)$ and, if $\alpha > 1$, $\partial_x \exp(f)$ as well as for the $\alpha$ H\"older seminorm of $\exp(f)$ or $\partial_x \exp(f)$, according to whether $\alpha<1$ or $\alpha>1$ respectively. Regarding the first bound, since
	$
	    \norm{f}_{\infty, p(\delta) } \le R
	$
	it follows directly that
	\[
	    \exp(f(x) - R p(\delta)(x)) \le 1,
	\]
	implying that \[\norm{\exp(f)}_{\infty, e(R)} \le 1.\]
	Furthermore for any $\beta \in (0,1\wedge \alpha]$ we also have
	\[
	\sup_{|x - y| \le 1} \frac{|e^{f(x)} - e^{f(y)}|}{e(\tilde{l}+\hat{l})(x)|x-y|^{\beta}} \le  \sup_{|x - y| \le 1} \frac{e^{f(x) \vee f(y)}}{e(\tilde{l})(x)} \frac{|f(x) - f(y)|}{e(\hat{l})(x)|x-y|^{\beta}}\lesssim \norm{f}_{\beta, e(\hat{l})}
	\]
    whenever $\tilde{l} > R.$ Moreover $\exp$ is also locally Lipschitz continuous, since for any two functions $f$ and $g$ as in the statement of this lemma we can write:
    \[
        e^{f(u)} - e^{g(u)} = \int_0^1\exp \big(g(u) + t (f(u) {-} g(u)) \big) (f(u) {-} g(u)) dt
    \]
    and therefore, for an appropriate choice of $\tilde{l}$,
    \[
        \norm{\exp(f) - \exp(g)}_{\infty, e(\tilde{l} + \hat{l} )} \lesssim_M \| f - g \|_{\infty, e(\hat{l})}.
    \]
    The same integral remainder formula gives
    \begin{align*}
        |e^{f(x)} - e^{g(x)} - (e^{f(y)} - e^{g(y)})| \le & \bigg| \! \int_0^1 [e^{g(x) + u (f{-}g)(x)} - e^{g(y) +u(f{-}g)(y)}](f{-}g)(x) du \bigg| \\
        +&  \bigg| \! \int_0^1 e^{g(y) + u (f{-}g)(y)}[(f{-}g)(y) -  (f{-}g)(x)] du \bigg|,
    \end{align*}
    so that applying inequalities of the kind $|a \tilde{a} -b \tilde{b}| \le |a-b||\tilde{a}| + |\tilde{a} - \tilde{b}||b|$ along with an appropriate choice of $\tilde{l}$ leads to the bound: 
    \[
        \sup_{|x - y| \le 1} \frac{|e^{f(x)} - e^{g(x)} - (e^{f(y)} - e^{g(y)})| }{ e(\tilde{l}+\hat{l})(x)|x-y|^{\beta} } \lesssim_{M} \| f - g \|_{\beta, e(\hat{l})}.
    \]
Finally, if $\alpha \in (1,2),$ we can write $\partial_x e^f = e^f \partial_x f,$ so that via the previous calculations and through an application of paraproduct estimates we deduce that $\partial_x e^f $ lies in $\CC^{\alpha - 1}_{e(l)}$ for $l = \tilde{l}+ \hat{l},$ along with the local Lipschitz continuity.
\end{proof}

The same calculations also show that the result still holds if we introduce time dependence:
\begin{lemma}\label{lem:exp_2}
Consider any $\alpha \in (0, 2) \setminus \{ 1 \}, \gamma \in (0,1)$ and $R, \hat{l} \ge 0.$ Then there exists an $l = l(R) \ge 0$ depending on $R$ such that the exponential function maps:
\[
	\exp : \bigg\{ f \in \LL^{\gamma, \alpha}_{e(\hat{l})} \text{ s.t. } \sup_{0 \le t \le T} \norm{f(t)}_{\infty, p(\delta)} \le R \bigg\} \longrightarrow \LL^{\gamma, \alpha}_{e(l)}.
\]
Moreover this function is locally Lipschitz continuous i.e. for $f,g$ with 
\[
    \| f\|_{\LL^{\gamma, \alpha}_{e(\hat{l})}}, \| g \|_{\LL^{\gamma, \alpha}_{e(\hat{l})}} \le M
\]
we can estimate:
\[
    \| \exp(f) - \exp(g) \|_{\LL^{\gamma, \alpha}_{e(l)}} \lesssim_{M,R} \| f - g \|_{\LL^{\gamma, \alpha}_{e(\hat{l})}}.
\]
\end{lemma}

Now we pass to a dual statement, namely the continuity of the logarithm.

\begin{lemma}\label{lem:log}
For given $\alpha \in (0,2) \setminus \{ 1\}, \gamma \in (0, 1)$ and $ r, C, c, \hat{l} > 0$ there exists an $l = l(r) > 0$ depending on $r$ such that the logarithm maps 
\[
\log: \mathcal{A} = \Big\{ f \in \LL^{\gamma, \alpha}_{ e(\hat{l})} \ \big\vert \ c e^{{-}r|x|^\delta} \le f(t, x) \le C e^{r|x|^\delta} \Big\} \longrightarrow \LL^{\gamma, \alpha}_{ e(l)}.  
\]
Furthermore this map is locally Lipschitz continuous, i.e. for $f,g \in \mathcal{A}$ with 
\[
   \| f\|_{\LL^{\gamma, \alpha}_{e(\hat{l})} } , \  \| g \|_{\LL^{\gamma, \alpha}_{e(\hat{l})} } \le M
\]
we can estimate:
\[
    \| \log(f) - \log(g) \|_{\LL^{\gamma, \alpha}_{e(l)} } \lesssim_{M,r} \| f - g \|_{\LL^{\gamma, \alpha}_{ e(\hat{l})} }.
\]
\end{lemma}

\begin{proof}
As before we use the classical definition of H\"older spaces and we only treat spatial regularity, as the time regularity follows similarly. We also just discuss the local Lipschitz continuity, since then the first statement follows by choosing $g=1$. For $f,g \in \mathcal{A}$ and for all $\xi \ge f(t,x) \wedge g(t,x)$ we find  
\[
    \frac{1}{\xi} \lesssim e(r)(x).
\] 
Thus by the mean value theorem and for $l \ge \hat{l} + r$: 
\[\norm{ \log (f(t)) {-} \log (g(t))}_{\infty, e(l)}  \lesssim \norm{f(t) {-} g(t)}_{\infty, e(\hat{l}) }\]
uniformly in $t \in [0,T]$. If $\alpha >1$ we can bound also the derivative in a similar way:
\begin{align*}
	\norm{ \frac{\partial_x f(t)}{f(t)} - \frac{ \partial_x g(t)}{g(t)}}_{\infty, e(l)}   & \le \norm{\bigg\vert \frac{1}{f(t)} \bigg\vert |\partial_x f(t) - \partial_x g(t)| + |\partial_x g(t)|\bigg\vert \frac{1}{f(t) g(t)} \bigg\vert |f(t) - g(t)|}_{\infty, e(l)} \\
	& \lesssim_M \| \partial_x f(t) {-}\partial_x g(t) \|_{\infty, e(\hat{l})} {+} t^{{-}\gamma}\| f(t) - g(t) \|_{\infty, e(\hat{l})}
\end{align*}
at the cost of taking a possibly larger $l$.

To treat the case $\alpha<1$ let $|x-y| \le 1$ and observe that
\[
\frac{f(y)}{f(x)}\vee \frac{g(y)}{g(x)} \lesssim_M e(\hat{l}+r)(x),
\]
so we can apply again the mean value theorem to the logarithm:
\begin{align*}
    & \frac{ \big| \log(f(t,x)) - \log(g(t,x)) - \big( \log(f(t,y)) - \log(g(t,y)) \big) \big| }{e(l)(x)\bv{x-y}^{\alpha}} =  \frac{\Big\vert \log \bp{\frac{f(x)}{f(y)}} - \log \bp{\frac{g(x)}{g(y)}} \Big\vert }{e(l)(x) \bv{x-y}^{\alpha}} \\
	& \ \ \ \lesssim \frac{e(\hat{l}+r)(x)}{e(l)(x)} \frac{| \bq{(f{-}g)(t,x) - (f{-}g)(t,y)}g(t,y) + (f{-}g)(t,y)(g(t,y) - g(t,x)) |}{ |f(t,y)g(t,y)|\bv{x-y}^{ \alpha }} \\
    & \ \ \ \ \ \ \ \  \lesssim_M  \norm{f(t)-g(t)}_{\alpha, e(\hat{l})} {+} t^{-\gamma} \norm{f(t)-g(t)}_{\infty, e(\hat{l})} ,
\end{align*}
once more for an appropriate choice of $l > \hat{l} + r$. Calculations similar to the one above show that also in the case $\alpha > 1$ we can find an $l$ such that
\[
    \bigg\|  \frac{\partial_x f}{f} - \frac{\partial_x g}{g} \bigg\|_{\alpha - 1, e(l)}  \lesssim_M  \| f(t){-}g(t)\|_{\alpha, \hat{l}} {+} t^{{-}\gamma}\|f(t){-}g(t)\|_{\infty, e(\hat{l})}.
\]
This guarantees that
\[ \sup_{t \in [0, T]} t^\gamma\|\log(f(t)){-} \log(g(t))\|_{\alpha, e(l)} \lesssim_M \|f{-}g\|_{\LL^{\gamma, \alpha}_{e(l)}}.\]
We can argue similarly for the time regularity and this concludes the proof.

\end{proof}

\section{Operations on the Extended Data}
Here we discuss some operations on the space of extended data, more precisely translation and parabolic scaling.
For this purpose we consider a slightly different space of external data. In contrast to the space $\mY_{kpz}$ we will not take into account the norm of $Y$, so that we can forget about its linear growth; we will also allow some of the terms in the extended data to have inhomogeneous initial conditions which arise from shifting the data in time; finally we consider a small time explosion for the resonant product $\partial_x Y^{\BZ} \reso X$: this also arises naturally when shifting the data.

\begin{definition}\label{def:singular-Y}
Consider $\zeta \in [0, 1), b\ge 0$ and a time horizon $T_h \ge 0$. Let us write $\mY_{kpz}^{\zeta, b}([0, T_h])$ for the closure in
\begin{align*}
	C\CC^{\alpha - 1}_{p(a)}([0, T_h])  \times \LL^{2 \alpha}_{p(a)}([0, T_h]) \times & \LL^{\alpha +1}_{p(a)}([0, T_h]) \times \\
	& \times \LL^{2\alpha +1}_{p(a)}([0, T_h]) \times \LL^{2\alpha +1}_{p(a)}([0, T_h]) \times \MM^{\zeta}\mathcal{C}^{2\alpha-1}_{p(b)}([0, T_h])
\end{align*}
of the map $\YY(\theta, c^{\BA}, c^{\BD})$ defined on $\smooth\times \RR \times \RR$ by:
\[
\YY(\theta, c^{\BA}, c^{\BD}) = \big( X, Y^{\BA}, Y^{\BB}, Y^{\BC}, Y^{\BD}, \partial_x Y^{\scalebox{0.8}{\RS{r}}} \reso X \big)
\]
where every tree $Y^{\bullet}$ solves the same PDE as in Table \ref{table:kpz} under the condition that 
\[
Y^{\BC}(0) =0 , \ \ Y^{\BD}(0) =0 , \ \ Y^{\scalebox{0.8}{\RS{r}}}(0) =0 
\]
implying that all other trees are allowed to have inhomogeneous initial conditions. To keep the notation simple we omit from writing the dependence on these inhomogeneous initial conditions. We will also sometimes omit the explicit dependence on $T_h$ and write $\mY_{kpz}^{\zeta, b}$.
\end{definition}

Next we will formulate translations and parabolic scaling operations on the extended data. As for the scaling, we will only zoom into small scales and therefore the scaling parameter $\lambda$ will be small: $\lambda \in (0,1]$. Thus, for any $\tau \in [0,T)$ we write $\theta_{\tau, \lambda}(t,x) = \theta(\tau {+} \lambda^2 t, \lambda x).$ We change the time horizon accordingly to $T_{\tau,\lambda} = \lambda^{{-}2}(T {-} \tau)$.
The question we want to answer is whether for given $\YY(\theta^n) \in \myi$ converging to $\YY$ in $\mY_{kpz}$ we can show also the convergence of $\YY(\theta^n_{\tau, \lambda})$ to some $\YY_{\tau, \lambda}$: this is the content of the following result. 

\begin{proposition}\label{prop:rescale_tranlate_external_data}
For any $\tau, \lambda$ as above and for every $\YY$ in $\mY_{kpz}$ and $\zeta \in (1/2 - \alpha, \alpha]$ there exists a $\YY_{\tau, \lambda}$ in $\mY_{kpz}^{\zeta, 2a}([0, T_{\tau, \lambda}])$ such that whenever $\theta^n \in \smooth$ is such that $\YY(\theta^n, c^{\BA}_n, c^{\BD}_n)$ converges to $\YY  \text{ in } \mY_{kpz}$, then
\[
  \YY(\lambda^2\theta^n_{\tau, \lambda}, \lambda^2c^{\BA}_n, \lambda^2 c^{\BD}_n ) \to \YY_{\tau, \lambda} \text{ in }\mY_{kpz}^{\zeta, 2a}([0, T_{\tau,\lambda}]),
\]
where for a given smooth noise $\theta \in \smooth$ we define $\YY(\lambda^2 \theta_{\tau, \lambda})$ by \[Y( \lambda^2 \theta_{\tau, \lambda})(t,x) = Y(\theta)(\tau {+}\lambda^2 t, \lambda x)\] and similarly for the elements $Y^{\BA}(\theta_{\tau, \lambda}), Y^{\BB}(\theta_{\tau. \lambda})$. The elements 
\[Y^{\BC}(\theta_{\tau, \lambda}),  \ \ Y^{\BD}(\theta_{\tau, \lambda}), \ \ Y^{\BZ}(\theta_{\tau, \lambda})\]
are defined respectively as the solution to
\begin{align*}
\LLL  Y^{\BC}(\theta_{\tau, \lambda}) &= \partial_x Y^{\BB}(\theta_{\tau, \lambda}) \partial_x Y(\theta_{\tau, \lambda}),  & Y^{\BC}(\theta_{\tau, \lambda})(0) = 0, \\
\LLL  Y^{\BD}(\theta_{\tau, \lambda}) &= \partial_x Y^{\BA}(\theta_{\tau, \lambda}) \partial_x Y^{\BA}(\theta_{\tau, \lambda}), & Y^{\BD}(\theta_{\tau, \lambda})(0) = 0, \\
\LLL  Y^{\BZ}(\theta_{\tau, \lambda}) &= \partial_x Y(\theta_{\tau, \lambda}), & Y^{\BZ}(\theta_{\tau, \lambda})(0) = 0.
\end{align*}
Furthermore we have the estimate:
\begin{equation*}
\sup_{\lambda \in (0,1]} \sup_{\tau \in [0,T)} \| \YY_{\tau, \lambda}\|_{\mY_{kpz}^{\zeta, 2a}([0, T_{\tau,\lambda}])} \lesssim_{\zeta} \| \YY \|_{\mY_{kpz}}.
\end{equation*}
\end{proposition}

\begin{proof}

We concentrate on the proof of the uniform bound. The convergence result then follows from the fact that the rescaling operator is linear. Let us write $Y^{\bullet}_{\tau, \lambda}$ for $Y^{\bullet}(\theta_{\tau, \lambda})$ defined as above. Note that 
\[Y^{\BC}_{\tau, \lambda}(t) = \Lambda_{\lambda} \big[ Y^{\BC}(\tau{+}\lambda^2 t) {-} P_{\lambda^2 t}Y^{\BC}(\tau) \big] \]
and similarly for $Y^{\BD},$ where $\Lambda_{\lambda}$ is the spatial rescaling operator with $\Lambda_\lambda f(x) = f(\lambda x)$. We can find the required bounds for $X_{\tau, \lambda}$ and for the tree terms $Y^{\bullet}_{\tau, \lambda}$ in view of \cite[Lemma A.4]{singular_GIP}. While the results from \cite{singular_GIP} do not treat weighted spaces, they hold nonetheless in the weighted setting. To see this one has to take care of the effect of rescaling, and it is only because we are zooming in ($\lambda \le 1$) that the scaling does not affect the weight. The most complicated object that we have to consider is the ill-posed product
\[
\sup_{\lambda \in (0,1]} \sup_{\tau \in [0,T]}\sup_{t \in [0, T_{\tau, \lambda}]} t^{\zeta}\| \partial_x Y^{\scalebox{0.8}{\RS{r}}}_{\tau, \lambda}(t) \reso\partial_x Y_{\tau, \lambda}(t) \|_{\CC^{2\alpha{-}1}_{p(2a)}}.
\]
First, we treat the rescaling parameter $\lambda.$ Note that $Y^{\scalebox{0.8}{\RS{r}}}_{\tau, \lambda}(t) = \Lambda_{\lambda}
\Big(\frac{1}{\lambda}Y^{\scalebox{0.8}{\RS{r}}}_{\tau}(\lambda^2 t) \Big),$ with the convention $Y^{\scalebox{0.8}{\RS{r}}}_{\tau} = Y^{\scalebox{0.8}{\RS{r}}}_{\tau, 1}$. Hence, we can write
\begin{align*}
\partial_x Y^{\scalebox{0.8}{\RS{r}}}_{\tau, \lambda}(t)& \reso \partial_x Y_{\tau, \lambda}(t) = \lambda \big[ \rel \partial_x Y^{\scalebox{0.8}{\RS{r}}}_{\tau}(\lambda^2 t) \reso \rel  \partial_x Y_{\tau}(\lambda^2t) \big] \\
& =\lambda \rel \big[ \partial_x Y^{\scalebox{0.8}{\RS{r}}}_{\tau}(\lambda^2 t) \reso  \partial_x Y_{\tau}(\lambda^2t) \big] + \lambda \text{ Commutator },
\end{align*}
where ``$\text{Commutator}$'' is defined implicitly through the formula. An application of \cite[Lemma B1]{singular_GIP} (taking into account the remark about the weights from above) tells us that:
\[
\lambda \| \text{ Commutator } \|_{\CC^{2\alpha{-}1}_{p(2a)}} \lesssim \lambda^{2\alpha}\| \partial_x Y^{\scalebox{0.8}{\RS{r}}}_{\tau}(\lambda^2 t)\|_{\CC^{\alpha}_{p(a)}} \| \partial_x Y_{\tau}(\lambda^2 t)\|_{\CC^{\alpha{-}1}_{p(a)}}
\]
and an application of \cite[Lemma A4]{singular_GIP} tells us that:
\[
\lambda \| \rel \big[ \partial_x Y^{\scalebox{0.8}{\RS{r}}}_{\tau}(\lambda^2 t) \reso  \partial_x Y_{\tau}(\lambda^2t) \big] \|_{\CC^{2\alpha{-}1}_{p(2a)}} \lesssim \lambda^{2\alpha}\| \partial_x Y^{\scalebox{0.8}{\RS{r}}}_{\tau}(\lambda^2 t) \reso  \partial_x Y_{\tau}(\lambda^2t)\|_{\CC^{2\alpha{-}1}_{p(2a)}}.
\]
Now we can estimate the norm of the ill-posed product uniformly by:
\begin{align*}
\lambda^{2\alpha}t^{\zeta}\| \partial_x Y^{\scalebox{0.8}{\RS{r}}}_{\tau}(\lambda^2 t) \reso  \partial_x Y_{\tau}(\lambda^2t)\|_{\CC^{2\alpha{-}1}_{p(2a)}} \lesssim \lambda^{2\alpha{-2\zeta}}(\lambda^2t)^{\zeta}\| \partial_x Y^{\scalebox{0.8}{\RS{r}}}_{\tau}(\lambda^2 t) \reso  \partial_x Y_{\tau}(\lambda^2t)\|_{\CC^{2\alpha{-}1}_{p(2a)}} 
\end{align*}
and since $\zeta \le \alpha$, this can be bounded uniformly over $\lambda$ and $t$ by the quantity:
\[
\sup_{\tau \in [0,T]}\sup_{t \in [0, T{-}\tau]} t^{\zeta}\| \partial_x Y^{\scalebox{0.8}{\RS{r}}}_{\tau}(t) \reso \partial_x Y_{\tau}(t) \|_{\CC^{2\alpha{-}1}_{p(2a)}} + \| \partial_x Y^{\scalebox{0.8}{\RS{r}}}_{\tau}(t)\|_{\CC^{\alpha}_{p(a)}} \| \partial_x Y_{\tau}(t)\|_{\CC^{\alpha{-}1}_{p(a)}}.
\]
It is easy to estimate the last term uniformly over $t$ and $\tau.$ Let us consider the first term. Here we have to take into account that $ Y^{\scalebox{0.8}{\RS{r}}}_{\tau}(t) =  Y^{\scalebox{0.8}{\RS{r}}}(\tau+t)-P_t Y^{\BZ}(\tau)$, where we recall that $P_t$ indicates convolution with the heat kernel and that it commutes with derivatives. Since we have no a priori estimates for $P_t \partial_x Y^{\BZ}(\tau) \reso \partial_x Y_{\tau}(t)$ we need to apply the usual estimates for the resonant product. For that purpose note that \[\|P_t \partial_x Y^{\scalebox{0.8}{\RS{r}}}(\tau)\|_{\CC^{2\zeta{+}\alpha}_{p(a)}} \lesssim t^{{-}\zeta}\| \partial_x Y^{\scalebox{0.8}{\RS{r}}}(\tau)\|_{\CC^{\alpha}_{p(a)}}\] and since $2\zeta{+}2\alpha{-}1>0$ we can bound the norm of the ill-posed product by
\begin{align*}
\| \partial_x Y^{\scalebox{0.8}{\RS{r}}}_{\tau} & (t) \reso  \partial_x Y(\tau{+}t) \|_{\CC^{2\alpha{-}1}_{p(2a)}} \\
& \lesssim \|P_t \partial_x Y^{\scalebox{0.8}{\RS{r}}}(\tau )\|_{\CC^{2\zeta{+}\alpha}_{p(a)}} \| \partial_x Y (t{+}\tau)\|_{C\CC^{\alpha{-}1}_{p(a)}} + \|  \partial_x Y^{\scalebox{0.8}{\RS{r}}} \reso \partial_x Y (t{+}\tau)\|_{C\CC^{\alpha{-}1}_{p(a)}} \lesssim  t^{{-}\zeta}.
\end{align*}
Now all the required properties follow promptly.
\end{proof}

\section{Asymmetric Approximation of a Resonant Product}

Next we prove a result which is a slightly asymmetric version of the computations in \cite[Section 9.5]{kpz}. Indeed we show convergence of $X^{\BZ, n} \reso X$ to $X^{\BZ} \reso X$, that is we only regularize one of the two factors.

\begin{lemma}\label{lem:renormalisation_for_measure}
Let $\xi$ be a white noise on $[0,T]\times \RR.$ Consider the sequence $(\xi^n, Y^n_0, c_n^{\BA}, c_n^{\BD})$ as in Theorem \ref{thm:renormalisation}. Then for any $a >0$ and $\alpha < 1/2$ the resonant product
\[
\partial_x Y^{\BZ, n} \reso \partial_x Y \to \partial_x Y^{\BZ} \reso \partial_x Y \text{ in }  L^p(\Omega; C\CC^{2\alpha{-}1}_{p(a)})
\]
for some $p = p(\alpha, a) \in [1, {+}\infty)$.
\end{lemma}

\begin{proof}
Following the notation of \cite[Section 9]{kpz} we have a representation of $X \reso X^{\BZ,n}$ via a product of Wiener-It\^o integrals as
\[
X \reso X^{\BZ,n}(t,x) =  \int\limits_{(\RR \times [0, T])^2} G^{X \reso X^{\BZ,n}}(t,x, \eta_{12} ) W(d \eta_1)W(d \eta_2)
\]
with kernel
\[
G^{X\reso X^{\BZ,n}}(t,x, \eta_{12}) = e^{i k_{[12]}x} \psi_0 (k_1, k_2) H_{t-{s_1}}(k_1) \varphi(n^{-1} k_2)\int\limits_{\RR} d\sigma H_{t{-}\sigma}(k_2) H_{\sigma{-}s_2}(k_2).
\] 
Decomposing the above integral into different Wiener-It\^o chaoses and setting $\eta_1 = (s_1, k_1)$ and $\eta_{-1} = (s_1, -k_1)$, the contribution to the chaos of order zero is given by:
\begin{align*}
	\int_{\RR \times \RR} d\eta_{1}G^{X\reso X^{\BZ,n}}& (t,x, \eta_{1(-1)}) \\
	& =  \int_{\RR \times \RR} d\eta_{1} H_{t-{s_1}}(k_1) \varphi(- n^{-1} k_1) \int\limits_{\RR} d\sigma H_{t{-}\sigma}(-k_1) H_{\sigma{-}s_1}(-k_1) = 0,
\end{align*}
where we used that $H_r (-k_1) = - H_r(k_1)$ and $\varphi(-n^{-1} k_1) = \varphi(n^{-1} k_1)$ to see that the integrand is antisymmetric under the change of variables $k_1 \to -k_1$ and therefore its integral must vanish.  Hence we have to consider only the second Wiener-It\^o chaos. To show convergence of the sequence we first compute  for $q \ge {-}1$:
\begin{align*}
	\EE \big[|\Delta_q (X\reso X^{\BZ} {-} &  X \reso X^{\BZ,n}) |^2(t,x) \big] \lesssim \int\limits_{(\RR\times [0, T])^2} d\eta_{12} |\varrho_{q}(k_{[12]}) G^{X\reso X^{\BZ,n}}(t,x, \eta_{1 2})|^2 \\
	& = \int dk_{12} ds_1 ds_2 \ \varrho_{q}(k_{[12]}) \psi_0(k_1, k_2)^2 |H_{t-{s_1}}|^2(k_1) \cdot \\
	& \qquad \qquad \qquad \cdot  (1{-}\varphi(n^{-1} k_2))^2 \Big\vert \int\limits_{\RR} d\sigma H_{t{-}\sigma}(k_2) H_{\sigma{-}s_2}(k_2) \Big\vert^2 \\
	& \lesssim \int dk_{12} \ \varrho_{q}(k_{[12]}) \psi_0(k_1, k_2)^2 (1{-}\varphi(n^{-1} k_2))^2  k_2^{-2} 
\end{align*}
where $k_{[12]} = k_1{+}k_2, \ dk_{12} = dk_1 dk_2$, and to obtain the inequality in the third line we have estimated $\int_0^t ds_1 \ |H_{t{-}s_1}|^2(k_1) \lesssim 1$ uniformly over $k_1$ and
\begin{align*}
\int_0^t ds_2 \ \Big\vert & \int\limits_{\RR} d\sigma H_{t{-}\sigma}(k_2) H_{\sigma{-}s_2}(k_2) \Big\vert^2 = \int_0^t ds \ (k_2^2 s)^2 e^{-2s k_2^2} \lesssim k_2^{-2}.
\end{align*}
In particular we see that the for some annulus $\mA$ such that $supp(\varrho_j) \subset \mA_j = 2^j \mA$ we have that $k_{[12]} \in \mA_q$ and $\psi_0(k_1, k_2) >0$ implies $k_2 \in \mA_j$ for $j \gtrsim q$ so that we can estimate the integral via
\begin{align*}
\int dk_{[12]} dk_2 \ \varrho_{q}(k_{[12]}) \psi_0(k_{[12]}{-}k_2, k_2)^2 (1{-}\varphi(n^{-1} k_2))^2  k_2^{-2} \lesssim 2^{q} \sum_{j \gtrsim q \vee (c n)}  2^{-{j}} \lesssim 2^{q{-}(q\vee c n)},
\end{align*}
for some $c > 0.$

At this point it is possible to conclude, since for $p$ such that $pa > 1$ and any $\delta >0$:
\begin{align*}
	\EE \big[ \| X \reso X^{\BZ,n}& - X \reso X^{\BZ}\|^p_{B^{-\kappa}_{p,p}(p(a))} \big] = \sum_{q} 2^{{-}\kappa q p} \int_{\RR} \frac{ \EE \big[ |\Delta_q  (X\reso (X^{\BZ} {-}X^{\BZ,n}) ) |^2 (t,x) \big]^{p/2}}{1{+}|x|^{ap}} \lesssim \delta
\end{align*}
where the last estimate holds if we choose $q_0$ such that $\sum_{q \ge q_0}2^{-{\kappa}q p} \le \delta$ and $n$ such that $2^{(q_0 {-} cn)p/2} \le \delta$ in order to split up the last sum at the $q_0$-th term and to obtain two small terms. The time regularity follows similarly, see also Section~9.5 of~\cite{kpz}. By Besov embedding we find the required convergence.
\end{proof}

\section{Schauder Estimates}

In this section we review classical Schauder theory for space-time distributions. Such results are well-known in literature: we adapt them in order to deal with time-dependent weights and blow-ups at time $t = 0.$ The method of proof we use is essentially the same of \cite[Theorem 1]{gubinelli_tindel}, which is based on the construction of the Young integral. Such integral is the content of the next lemma.

\begin{lemma}\label{lem:young_integration_with_time_blowup}
Let $T_h \ge 0, \ \beta \in [0,1), \ \alpha, \gamma \in (0,2)$ such that $\alpha{+}\gamma > 2.$ Then for $f \in \LL^{\beta, \alpha}_{z_1}([0, T_h])$ and $h \in \LL^{\gamma}_{z_2}([0, T_h]),$ where $z_i : \RR_{\ge 0} \to \weights$ are point-wise increasing, it is possible to define the Young integral:
$$
I_t = \int_0^t f(r)dh(r) 
$$
such that $t \mapsto t^{\beta}I_t$ lies in $C^{\gamma/2} L^{\infty}_{z_1 z_2}([0, T_h])$. This map is bilinear and satisfies the bound:
$$
\| t \mapsto t^{\beta}I_t \|_{C^{\gamma/2} L^{\infty}_{z_1 z_2}([0, T_h])} \lesssim_{T_h} \|f\|_{\LL^{\beta, \alpha}_{z_1}([0, T_h])} \| h\|_{\LL^{\gamma}_{z_2}([0, T_h])}.
$$
For $f,h \in C^{\infty}_b([0, T_h] \times \RR; \RR)$, $I$ is the unique map that satisfies $\partial_t I_t = f \partial_t h, \ I_0 = 0$.
\end{lemma}
\begin{proof}
Following the classical construction via the sewing lemma (cf. \cite[Lemma 4.2]{frizRP}) we can build the integral $\int_s^t f(r) dh(r)$ for any $t \ge s>0.$ We repeat the construction in order to get a tight control on the singularity in zero and the weights involved.
We will prove the result for time independent weights. The general case then follows from the identity: $\smallint_0^t f(s) dh(s) = \smallint_0^t f(s\wedge t)dg(s).$
Define for $n \ge 0$ and $t^n_k = k/2^n$
$$
I^n_t = \sum_{k = 0}^{{+}\infty} f(t^n_{k+1} )(h(t^n_{k+1} \wedge t) {-} h(t^n_k \wedge t)).
$$
unlike the more traditional integration scheme we choose a right base-point to remove some tedious, but only technical difficulties when dealing with time blow ups. We want to estimate the following quantity
$$
\sum_{n \ge 0} \sup_{0 \le s \le t \le T} \frac{\| t^{\beta}(I^{n+1}_t{-}I^{n}_t)-s^{\beta}(I^{n+1}_s{-}I^{n}_s) \|_{L^{\infty}_{z_1z_2}}}{|t{-}s|^{\gamma/2}}.
$$
We will treat only the case $\beta >0$, since $\beta = 0$ follows similarly. We fix $n$ and estimate one of the terms above. We will divide the estimate in two parts.

\textit{Step 1.} First we look on large scales, that is $|t{-}s| > 2^{-n}$. To lighten the notation we write $g_{u,v} = g(u) {-} g(v)$. We also write $t_n \ (\text{ resp. } t^n)$ for the nearest left (resp. right) dyadic point to $t$: 
$$k_n(t) = \arg \min_{k | k \le 2^n t} |t{-}t^n_k|, \ \ t_n = k_n/2^n, \ \ t^n = t_n {+}1/2^n. $$  We start by considering $t \le 2s.$ in this case we will estimate the terms $t^\beta (I^{n+1}_{s,t}{-}I^{n}_{s,t})$ and $(t^{\beta}-s^{\beta})(I^{n+1}_s{-}I^{n}_s)$ separately. Let us start with the first one. Since $|t{-}s|> 2^{-n}$ we have in particular that $s^n < t$ and a close inspection of the sums reveals that:
\begin{align*}
	I^n_{t,s}{-}I^{n+1}_{t,s} = f_{s^n, s^{n+1}} h_{s^{n+1}, s} + \sum_{k = k_n(s)+1}^{k_n(t)-1} f_{t^{n+1}_{2k+2}, t^{n+1}_{2k+1}} &  h_{t^{n+1}_{2k+1}, t^{n+1}_{2k}} \\
	& +  f_{t^{n}, t_{n+1}} h_{t_{n+1}, t_n} + f_{t^n, t^{n+1}} h_{t, t_{n+1}}. 
\end{align*}
Now if we call $\ol{f}(t) = t^{\beta}f(t)$ we can write an increment of $f$ as:
$$
f_{u, v} = u^{-\beta}\ol{f}_{u,v} + (u^{-\beta}{-} v^{- \beta})\ol{f}(v).
$$
Note that $\ol{f}$ is $\alpha/2-$H\"older continuous so that substituting this formula we find:
\begin{align*}
\|I^n_{t,s}{-} & I^{n+1}_{t,s}\|_{L^{\infty}_{z_1z_2}} \lesssim \|f\|_{\LL^{\beta, \alpha}_{z_1}} \| h\|_{\LL^{\gamma}_{z_2}} \cdot \\
& \cdot \bigg\{ \frac{1}{2^{n(\alpha+\gamma)/2}}\bigg[ (s^{n+1})^{-\beta} + \sum_{k = k_n(s)+1}^{k_n(t)-1} (t^{n+1}_{2k+1})^{-\beta} + (t_{n+1})^{-\beta} \bigg] \\
& + \frac{1}{2^{n(\gamma/2+1)}}\bigg[ (s^{n+1})^{-\beta-1+\alpha/2} + \sum_{k = k_n(s)+1}^{k_n(t)-1} (t^{n+1}_{2k+1})^{-\beta-1+\alpha/2} + (t_{n+1})^{-\beta-1+\alpha/2} \bigg] \bigg\} \\
	\lesssim & \|f\|_{\LL^{\beta, \alpha}_{z_1}} \| h\|_{\LL^{\gamma}_{z_2}} \bigg[ \sum_{k = k_{n+1}(s)+1}^{k_{n+1}(t)} (t^{n+1}_{k})^{-\beta}\frac{1}{2^{(n+1)(\alpha+\gamma)/2}} \bigg] \\
	\lesssim & \frac{1}{2^{n\varrho}}\|f\|_{\LL^{\beta, \alpha}_{z_1}} \| h\|_{\LL^{\gamma}_{z_2}} \frac{1}{2^{(n+1)(1-\beta)}}\big[ (k_{n+1}(t){+}1)^{1-\beta} - (k_{n+1}(s){+}1)^{1-\beta} \big]
\end{align*}
where for the first inequality we have used that for $u \ge v$ we can estimate $v^{- \beta}{-} u^{-\beta} \le (u{-}v) v^{-\beta{-}1}$ and for the second inequality we have used that since $\alpha/2 \le 1,$ we can estimate $(k/2^n)^{-1+\alpha/2} \le (1/2^n)^{-1+\alpha/2}.$ In the last line, since $\beta \in [0, 1)$, we estimated $\sum_a^b k^{-\beta} \le 2^{\beta}\smallint_a^{b+1} x^{-\beta}dx$ for $a \ge 1$ and we set $\varrho = (\alpha{+}\beta)/2{-}1>0$. Now using the assumptions on $t$ and $s$ we can estimate the last quantity as follows:
\begin{align*}
\frac{t^{\beta} \| I^n_{t,s}{-}I^{n+1}_{t,s} \|_{L^{\infty}_{z_1z_2}}}{|t{-}s|^{\gamma/2}} & \lesssim 2^{-n\varrho} t^{\beta} \frac{((t^{n+1})^{1-\beta}{-}(s^{n+1})^{1-\beta})}{(t{-}s)^{\gamma/2}} \le 2^{-n\varrho}\frac{(s^{n+1})^{-\beta}(t^{n+1}{-}s^{n+1})}{(t{-}s)^{\gamma/2}} \\
& \le 2^{-n\varrho}t^{\beta}\frac{t^{\beta}}{s^{\beta}}\big( |t{-}s|^{1-\gamma/2} + 2^{-n-1}|t-s|^{-\gamma/2}\big) \lesssim 2^{-n\varrho}. 
\end{align*}
now we treat the term $(t^{\beta}{-}s^{\beta})(I^{n+1}_{s}-I^{n}_s).$ Here we only need to adapt the previous calculations:
\begin{align}\label{eqn:technical_appendix_schuader_estimate}
\frac{\| (t^{\beta}{-}s^{\beta})(I^{n+1}_{s}-I^{n}_s) \|_{L^{\infty}_{z_1z_2}}}{|t{-}s|^{\gamma/2}} & \lesssim (t{-}s)^{1-\gamma/2}s^{\beta-1}\|I^{n+1}_{s}-I^{n}_s \|_{L^{\infty}_{z_1z_2}} \\
& \lesssim 2^{-n\varrho}s^{\beta-1}((s^{n+1})^{1-\beta}{-} 2^{-(n+1)(1-\beta)}) \lesssim 2^{-n\varrho} \nonumber
\end{align}
where in the last line we have estimated $(s^{n+1})^{1-\beta} \le s^{1-\beta} + 2^{-(n+1)(1-\beta)}$ together with the fact that $s \ge 2^{-(n+1)}$ since $|t-s|> 2^{-n}$ and $t\le 2s.$

We now consider $t \ge 2s.$ This time we will estimate the two terms $t^{\beta}I_t$ and $s^{\beta}I_s$ separately. Indeed in this case we find $t-s \ge t/2$ and we can estimate:
$$
\frac{\| t^{\beta}(I^{n+1}_t{-}I^{n}_t)-s^{\beta}(I^{n+1}_s{-}I^{n}_s) \|_{L^{\infty}_{z_1z_2}}}{|t{-}s|^{\gamma/2}} \lesssim t^{-\gamma/2+\beta}\|I^{n+1}_t{-}I^{n}_t\|_{L^{\infty}_{z_1z_2}} {+} s^{-\gamma/2{+}\beta}\| I^{n+1}_s{-}I^{n}_s \|_{L^{\infty}_{z_1z_2}},
$$
so that similar calculations to the ones from Equation \eqref{eqn:technical_appendix_schuader_estimate} pull thorough.

\textit{Step 2.} Now we have to consider the small-scales case with respect to $n$. If $|t-s| \le 2^{-n}$ the increment $I^n_{t,s}{-}I^{n+1}_{t,s}$ can assume only one of the following three forms according to the position of $s, t$ w.r.t to dyadic points, assuming $s,t$ are not dyadic themselves (i.e. of the form $k/2^{n+1}$):
$$
f_{t^n, t^{n+1}}h_{t,s} \text{ if } s^{n+1} > t, \ \ f_{t^n, s^{n+1}} h_{t,s} \text{ if } s^n > t, s^{n+1} < t, \ \ f_{t^n, t^{n+1}}h_{t, s^n} \text{ if } s^{n+1}<t. 
$$
and these formulas can be extended continuously in the case that one or both of the points $s,t$ is dyadic. Here the difficulty is a formality, namely that in our notation the points $t^n, t^{n+1}$ do not depend continuously on $t$. By continuity if is sufficient to control the H\"older seminorm for non-dyadic points. We show how to estimate the second term. The others follow similarly. As in the previous discussion we have:
\begin{align*}
t^{\beta}|f_{t^n, s^{n+1}} h_{t,s}| & \le t^{\beta}\big( (s^{n+1})^{-\beta}2^{-n\alpha/2}  + (s^{n+1})^{-\beta-1+\alpha/2}2^{-n} \big)|t{-}s|^{\gamma/2} \lesssim 2^{-n\alpha/2} |t{-}s|^{\gamma/2}.
\end{align*}
where in the last step we used that $s^{n+1} = t_{n+1}$ and $s^{n+1} \ge 2^{-(n+1)}$ together with the estimate $t^{\beta} \le (t_{n+1})^{\beta} + 2^{-n\beta}$.
The result regarding smooth functions follows from the properties of the Riemann integral.
\end{proof}

We also need the following reformulation of \cite[Lemma A.8]{singular_GIP}.

\begin{lemma}\label{lem:small_technical_schauder_estimate}
For $\alpha \in \RR, \delta \ge 0, t >0$ it is possible to estimate
$$
\|(Id {-} P_t)u \|_{\CC^{\alpha}_z} \lesssim t^{\delta/2} \| u \|_{\CC^{\alpha {+} \delta}_z}.
$$
\end{lemma}
\begin{proof}
For any $q \ge{-1}$ it follows from \cite[Lemma A.8]{singular_GIP}:
$$
	\|(Id {-} P_t)\Delta_q u \|_{L^{\infty}_z} \lesssim  t^{\delta/2} \| \Delta_q u\|_{\CC^{\delta}_z} \lesssim t^{\delta/2} 2^{\delta q} \sum_{j \sim q} \|\Delta_j u \|_{L^{\infty}_z}.
$$
The result then follows by multiplying with $2^{q\alpha}$ and summing over $q$.
\end{proof}

Finally we can prove Schauder estimates for time-dependent distributions.

\begin{lemma}\label{lem:young_heat_kernel_convolution}
For any $\beta \in [0,1), \ \alpha, \gamma \in (0,2)$ such that $\alpha{+}\gamma>2$ and $0\le \eta<\gamma {-}2a/\delta$ and $0\le T_\ell \le T_r \le T_h$ there exists a continuous bilinear map
$$
V_{T_\ell}: \LL^{\beta, \alpha}_{e(l+t)}([T_{\ell}, T_r]) \times \LL^{\gamma}_{p(a)}([T_{\ell}, T_r]) \to \LL^{\beta, \eta}_{e(l+t)}([T_{\ell}, T_r]), 
$$
such that
$$
\| V(f, h)\|_{\LL^{\beta, \eta}_{e(l+t)}([T_{\ell}, T_r])} \lesssim_{T_h} \| f\|_{\LL^{\beta, \alpha}_{e(l+t)}([T_{\ell}, T_r])} \| h\|_{\LL^{\gamma}_{p(a)}([T_{\ell}, T_r])}.
$$
Similarly, under the same assumption on $\beta, \alpha, \gamma,$ for given $z_i : \RR_{\ge 0} \to \weights$ point-wise increasing, $i = 1, 2$, we can build a map between the spaces
$$
V:\LL^{\beta, \alpha}_{z_1}([0, T_h]) \times \LL^{\gamma}_{z_2}([0, T_h]) \to \LL^{\beta, \eta}_{z_1 z_2}([0, T_h]),
$$
such that
$$
\| V(f, h)\|_{\LL^{\beta, \eta}_{z_1 z_2}([0, T_h])} \lesssim_{T_h} \| f\|_{\LL^{\beta, \alpha}_{z_1}([0, T_h])} \| h\|_{\LL^{\gamma}_{z_2}([0, T_h])}
$$
for any $0 \le \eta < \gamma$. For $f, h \in C^{\infty}_b([0, T] \times \RR; \RR)$ the map $V_{T_\ell}$ is the convolution with the heat kernel:
$$
V_{T_\ell}(f, h)(t) = \int_{T_\ell}^t P(s{-}T_\ell)f_s\partial_s h_s ds,
$$
with the silent identification of $V_0$ with $V.$
\end{lemma}

\begin{proof}
We prove only the result regarding exponential weights, since the second one follows via the same calculations and is simpler. In addition, for clarity and without loss of generality we assume $[T_\ell, T_r] = [0, T].$ Let us fix $f \in \LL^{\beta, \alpha}_{e(l+t)}$ and $h \in \LL^{\gamma}_{p(a)}$ and let $X$ be the Young integral from Lemma \ref{lem:young_integration_with_time_blowup}: $X_t = \smallint_0^t f_s dh_s$. We approximate the convolution with the heat kernel $V(f, h)$ in the following way:
$$
V^n_t = \sum_{k = 3}^{\lfloor 2^n t \rfloor -1} P(t{-}t^n_k)X_{t^n_{k+1}, t^{n}_k}.
$$
This approximation has the advantage of simplifying our calculations, but the disadvantage of not being continuous in time.
We want to show that $V^n_t$ converges to some $V_t$ which lies in $\LL^{\beta, \eta}_{e(l+t)}$ for $\eta < \gamma{-}2a/\delta$. For this reason fix $\kappa, \ve \ge 0$ small at will such that $\gamma{+}\kappa > 2$ and $\zeta = (\eta{+}\kappa{+}\ve)/2 {+} a/\delta < 1.$ We divide the proof in two steps, estimating the spatial and the temporal regularity differently.
For $t_n, t^n, t^n_k, k_n(t)$ we use the same definition as in the proof of Lemma \ref{lem:young_integration_with_time_blowup}.

\textit{Step 1.} We show that for fixed $t$ the sequence $V^n_t$ converges in $\CC^{\eta}_{e(l+t)}$. As in the previous proof we show that:
$$
\sum_{n \ge 0} t^{\beta}\|V^{n+1}_t {-} V^n_t \|_{\CC^{\eta}_{e(l+t)}} \lesssim \| f\|_{\LL^{\beta, \alpha}_{e(l+t)}} \|h \|_{\LL^{\gamma}_{p(a)}}
$$
thus deducing the existence of a limit $V_t$ with $\sup_{t \in [0, T]} t^{\beta}\|V_t \|_{\CC^{\eta}_{e(l+t)}} \lesssim \| f\|_{\LL^{\beta, \alpha}_{e(l+t)}} \|h\|_{\LL^{\gamma}_{p(a)}}$. We can write the difference of the increments as:
\begin{align*}
V^n_t {-} V^{n+1}_t = & \sum_{k = 3}^{\lfloor 2^nt \rfloor -1}(P(t^{n+1}_{2k+1}{-}t^{n+1}_{2k}){-}Id) P(t{-}t^{n+1}_{2k+1}) X_{t^{n+1}_{2k+2}, t^{n+1}_{2k+1}} \\
& - P(t{-}t_n) X_{t_{n+1}, t_n}  1_{\{ k_{n+1}(t) > 5\}} - P(t{-} t^{n+1}_3) X_{t^{n+1}_4, t^{n+1}_3} 1_{\{ k_{n+1}(t) = 5 \}}
\end{align*}
Now we can estimate
\begin{align*}
	\| V^n_t {-} V^{n+1}_t\|_{\CC^{\eta}_{e(l+t)}} & \lesssim  \sum_{k = 3}^{\lfloor 2^nt \rfloor -1}\frac{1}{2^{(n+1)\kappa/2}}\|P(t{-}t^{n+1}_{2k+1})(X_{t^{n+1}_{2k+2}, t^{n+1}_{2k+1}})\|_{\CC^{\eta+\kappa}_{e(l+t)}} \\
	& + \frac{\| X_{t_{n+1}, t_n}\|_{\CC^{-\ve}_{e(l+t)}}}{|t{-} t_n|^{(\eta{+}\ve)/2}}1_{ \{ k_{n+1}(t) > 5 \}} + \frac{\| X_{t^{n+1}_{4}, t^{n+1}_3}\|_{\CC^{-\ve}_{e(l+t)}}}{|t{-} t^{n+1}_3|^{(\eta{+}\ve)/2}}1_{ \{ k_{n+1}(t) = 5 \}}  \\
	& \lesssim \sum_{k = 3}^{\lfloor 2^nt \rfloor -1}\frac{ \|X_{t^{n+1}_{2k+2}, t^{n+1}_{2k+1}}\|_{\CC^{-\ve}_{p(a)e(l+t)}}}{2^{(n+1)\kappa/2} | t{-} t^n_k|^{(\eta+\kappa+\ve)/2 +a/\delta}} \\
	& + \frac{ \| X_{t_{n+1}, t_n}\|_{\CC^{-\ve}_{p(a)e(l+t)}}}{|t{-} t_n|^{(\eta{+}\ve)/2+a/\delta}} 1_{ \{ k_{n+1}(t) > 5 \}}  + \frac{\| X_{t^{n+1}_{4}, t^{n+1}_3}\|_{\CC^{-\ve}_{p(a)e(l+t)}}}{|t{-} t^{n+1}_3|^{(\eta{+}\ve)/2 + a/ \delta}}1_{ \{ k_{n+1}(t) = 5 \}}
\end{align*}
where we have applied the first Schauder estimate from Proposition \ref{prop:schauder_estimate} and the bound $\|f\|_{\CC^{\nu}_{p(a)e(l+t)}} \lesssim |t{-}s|^{-a/\delta}\|f\|_{\CC^{\nu}_{e(l+s)}}$, for $t\ge s$. Now we have to estimate the norm of the increment. Here the time explosions come into play: we write $X_{u, v} = u^{-\beta}\ol{X}_{u,v} + (u^{-\beta}{-} v^{- \beta})\ol{X}(v)$, with $\ol{X}(t) = t^{\beta}X(t) \in C^{\gamma/2}L^{\infty}_{p(a)e(l+t)}$ according to the result of Lemma \ref{lem:young_integration_with_time_blowup}. Since $v^{- \beta}{-} u^{-\beta} \le (u{-}v) v^{-\beta{-}1}$ we can estimate
\begin{align*}
\|X_{t^{n+1}_{2k+2}, t^{n+1}_{2k+1}}\|_{\CC^{-\ve}_{p(a)e(l+t)}} \lesssim ( t^{n+1}_{2k+1})^{-\beta}2^{-(n+1)\gamma/2} \| f\|_{\LL^{\beta, \alpha}_{e(l+t)}} \| h \|_{\LL^{\gamma}_{p(a)}}.
\end{align*}
At this point we can conclude, since:
\begin{align*}
	\| V^n_t {-} & V^{n+1}_t\|_{\CC^{\eta}_{e(l+t)}} \lesssim \| f\|_{\LL^{\beta, \alpha}_{e(l+t)}} \|h \|_{\LL^{\gamma}_{p(a)}} \bigg( 2^{-n\varrho} \sum_{k = 1}^{\lfloor 2^nt \rfloor -1} 2^{-n} |t{-} t^n_k|^{-\zeta}(t^{n+1}_{2k+1})^{-\beta} \\
	& + \frac{(t_n)^{-\beta}2^{-(n+1)\gamma/2} }{|t{-} t_n|^{(\eta{+}\ve)/2+a/\delta}}1_{\{ k_{n+1}(t) >5, t_n < t_{n+1} \}} + \frac{(t^{n+1}_3)^{-\beta} 2^{-(n+1)\gamma/2}}{|t{-} t^{n+1}_3|^{(\eta{+}\ve)/2+a/\delta}} 1_{\{ k_{n+1}(t) = 5 \}}\bigg)
\end{align*}
with $\varrho = (\kappa {+} \gamma)/2{-}1.$ Now the sum can be dominated by an integral:
\begin{align*}
\sum_{k = 3}^{\lfloor 2^nt \rfloor -1} 2^{-n} |t{-} t^n_k|^{-\zeta}(t^{n+1}_{2k+1})^{-\beta} \lesssim \int_{3 2^{-n}}^{\lfloor 2^n t \rfloor} |t{-}s|^{-\zeta}s^{-\beta} ds \lesssim t^{1-\zeta-\beta} \lesssim t^{-\beta}
\end{align*}
where we used that $2^{-n}(t^{n+1}_{2k+1})^{-\beta} \le 2^{-n}(t^{n}_{k})^{-\beta} \le 2^\beta \smallint_{t^n_k}^{t^n_{k+1}} s^{-\beta} ds$ as well as the fact that $\zeta < 1.$ Finally we bound the two rest terms. Under the condition $t_n < t_{n+1}$ we can estimate
\begin{align*}
 \frac{(t_n)^{-\beta}2^{-(n+1)\gamma/2} }{|t{-} t_n|^{(\eta{+}\ve)/2+a/\delta}} \lesssim (t_n)^{-\beta}2^{-(n+1)(\gamma-\eta-\ve)/2- a/\delta}  \lesssim t^{-\beta} 2^{-n\varrho}2^{n(\zeta-1)}
\end{align*}
since $(t_n/t)^{-\beta} \lesssim (1{-} 1/t2^{n})^{-\beta} \lesssim 2^{\beta}$, because $k_{n+1}(t) > 5$ and thus $t \ge 2^{-n+2}$. Similarly we can treat the last term: 
$$\frac{(t^{n+1}_3)^{-\beta} 2^{-(n+1)\gamma/2}}{|t{-} t^{n+1}_3|^{(\eta{+}\ve)/2+a/\delta}} \lesssim t^{-\beta} 2^{ -n\varrho} 2^{n(\zeta-1)}.$$
Since $\varrho > 0$ and $\zeta < 1$ this allows to deduce the spatial regularity.

\textit{Step 2.} Now we address the temporal regularity. Our aim is to estimate:
$$
\sup_{0 \le s \le t \le T} \frac{ \| t^\beta V_t {-} s^\beta V_s\|_{L^{\infty}_{e(l+t)}}}{|t{-}s|^{\eta/2}} \lesssim \| f\|_{\LL^{\beta, \alpha}_{e(l+t)}} \| h\|_{\LL^{\gamma}_{p(a)}}.
$$
For simplicity, as there is no difference w.r.t. the previous case we omit taking care of the norm of the functions. Let us first consider $t \ge s$ such that $t \le 2s.$ Here we rewrite the above quantity as:
$$
t^\beta V_t {-} s^\beta V_s = t^{\beta} (V_{t}{-} P(t{-}s) V_s) + t^{\beta}(P(t{-}s){-} Id)V_s - (t^{\beta}{-}s^{\beta})V_s.
$$
And we estimate all three these terms separately. The first one can be written as the limit:
$$
V_{t}{-} P(t{-}s) V_s = \lim_n W^n_{t,s}, \ \ W^n_{t,s} = \sum_{\lfloor s2^n \rfloor{+}1}^{\lfloor t2^n \rfloor{-}1} P(t{-} t^n_k)X_{t^{n}_{k+1}, t^n_k}
$$
since the rest terms vanish in the limit. At this point we estimate $W$ as before:
\begin{align*}
\|W^{n}_{t,s} {-} W^{n+1}_{t,s} \|_{L^{\infty}_{e(l+t)}} & \le \sum_{\lfloor s2^n \rfloor{+}1}^{\lfloor t 2^n \rfloor{-}1} \|(P( t^{n+1}_{2k+1} {-} t^{n+1}_{2k}){-} Id)P(t {-} t^{n+1}_{2k+1})X_{t^{n+1}_{2k+2}, t^{n+1}_{2k+1}}\|_{L^{\infty}_{e(l+t)}} \\
& + \| (P(t{-}s^{n+1})X_{s^n, s^{n+1}} + P(t{-}t_n)X_{t_{n+1}, t_n})1_{|t{-}s|\ge 2^{-n}}\|_{L^{\infty}_{e(l+t)}} \\
& \lesssim \sum_{\lfloor s2^n \rfloor{+}1}^{\lfloor t2^n \rfloor{-}1} \frac{(t^{n+1}_{2k+1})^{-\beta}}{2^{(n+1)(\kappa + \gamma)/2}|t{-}t^n_k|^{\kappa/2+a/\delta}} + \frac{(s^{n+1})^{-\beta}}{2^{n(\gamma/2-a/\delta)}}1_{|t{-}s|\ge 2^{-n}}\\
& \lesssim 2^{-n\widetilde{\varrho}}\bigg( \int_s^t u^{-\beta}(t{-}u)^{-\kappa/2-a/\delta} du + t^{-\beta}|t{-}s|^{\eta/2}\bigg) \lesssim 2^{-n\widetilde{\varrho}}|t{-}s|^{\eta/2}t^{-\beta}
\end{align*}
with $\widetilde{\varrho} = [(\gamma{-}\eta)/2{-}a/\delta]\wedge[(\kappa{+}\gamma)/2{-}1]$, where we used the fact that $\|f\|_{\CC^0_z} \lesssim \|f\|_{L^{\infty}_z}$ and since $t \le 2s$ and $1{-}\kappa/2{-}a/\delta>\eta/2$ we have estimated the integral by:
\begin{align*}
\int_s^t u^{-\beta}(t{-}u)^{-\kappa/2-a/\delta} du \lesssim t^{-\beta-\kappa/2-a/\delta+1}\int_{s/t}^1 (1{-}u)^{-\kappa/2-a/\delta} du \lesssim t^{-\beta}(t{-}s)^{1-\kappa/2-a/\delta}.
\end{align*}

As for the second term since $t \le 2s$ and via Lemma \ref{lem:small_technical_schauder_estimate} and the results of the first step we estimate 
$$
\|t^{\beta}(P(t{-}s){-} Id)V_s\|_{L^{\infty}_{e(l+t)}} \lesssim |t{-}s|^{\eta/2}.
$$
Finally, for the third term we estimate:
$$
\|(t^{\beta}{-}s^{\beta})V_s\|_{L^{\infty}_{e(l+t)}} \lesssim (t{-}s)^{1{-}\mu}s^{\beta{-}1+\mu}\|V_s\|_{L^{\infty}_{e(l+t)}}
$$
for any $\mu \in (0,1).$ For this purpose we follow the calculations in the first step. Indeed
\begin{align*}
\| V^n_t {-} V^{n+1}_t\|_{L^{\infty}_{e(l+t)}} \lesssim &  2^{-n\varrho} \int_0^t |t{-}s|^{-\ve/2-a/\delta}s^{-\beta} ds \lesssim t^{1-\beta-\ve/2-a/\delta}
\end{align*}
so that the result follows from the previous estimate by choosing $\mu = \ve/2 {-} a/\delta,$ since $1{-}\mu \ge \eta/2.$ If we suppose that $t> 2s$ we can estimate:
$$
\frac{\|t^{\beta}V_t {-}s^{\beta}V_s\|_{L^{\infty}_{(e(l+t)}}}{|t{-}s|^{\eta/2}} \le \sup_{t} t^{\beta-\eta/2}\| V_t \|_{L^{\infty}_{e(l+t)}}
$$
and this quantity can be bounded via the arguments we just used. This concludes the proof. The result regarding smooth functions follows again via Riemann integration.
\end{proof}

\bibliographystyle{alpha}
\bibliography{kpz1}

\end{document}